\documentclass{ps}
\usepackage{amsmath,amsfonts}
\usepackage{color}
\usepackage{hyperref}
\usepackage[applemac]{inputenc}
\usepackage{mathrsfs}

\def \O{\mathcal{O}}

\def\leftB{[\![}
\def\rightB{]\!]}

\newcommand \A[1]{{\bf (#1)}}

\def\O{{\cal{O}}}

\def\F{{\cal F}}

\def\P{{\mathbb{P}}  }

\def\bint#1^#2{\displaystyle{\int_{#1}^{#2}}}
\def\bsum#1^#2{\displaystyle{\sum_{#1}^{#2}}}

\def\xdt_#1{X_#1(\Delta t)}

\newtheorem{THM}{Theorem}[section]
\newtheorem{DEF}{Definition}[section]
\newtheorem{PROP}{Proposition}[section]

\newtheorem{COROL}{Corollary}[section]

\newtheorem{LEMME}{Lemma}[section]
\newtheorem{REM}{Remark}[section]

\newcommand{\mysection}{\setcounter{equation}{0} \section}

\newcommand{\blackboard}[1]{\mathbf{#1}} 
\newcommand{\R}{\blackboard{R}}%		reels
\newcommand{\N}{\blackboard{N}}%		entiers
\renewcommand{\P}{\blackboard{P}}%	 	probabilite
\newcommand{\E}{\blackboard{E}}% 		esperance
\newcommand{\delimleft}[2]{\ifcase #1\or%	delimiteur gauche
    \bigl#2\or %
    \Bigl#2\or %
    \biggl#2\or %
    \Biggl#2\or %
    \left#2\fi}
\newcommand{\delimright}[2]{\ifcase #1\or%	delimiteur droite
    \bigr#2\or %
    \Bigr#2\or %
    \biggr#2\or %
    \Biggr#2\or %
    \right#2\fi}

\begin{document}
%%-----------------------------
%%      the top matter
%%-----------------------------
\title{Transport-Entropy inequalities and deviation estimates for stochastic approximations schemes}% At most 5 thanks
\author{M. Fathi}\address{LPMA, Universit\'e Pierre et Marie Curie, 4 Place Jussieu, 75005 Paris Cedex, max.fathi@etu.upmc.fr}
\author{N. Frikha}\address{LPMA, Universit\'e Denis Diderot, 175 Rue du Chevaleret 75013 Paris, frikha@math.univ-paris-diderot.fr}
\date{\today}
\begin{abstract} We obtain new transport-entropy inequalities and, as a by-product, new deviation estimates for the laws of two kinds of discrete stochastic approximation schemes. The first one refers to the law of an Euler like discretization scheme of a diffusion process at a fixed deterministic date and the second one concerns the law of a stochastic approximation algorithm at a given time-step. Our results notably improve and complete those obtained in \cite{Frikha2012}. The key point is to properly quantify the contribution of the diffusion term to the concentration regime. We also derive a general non-asymptotic deviation bound for the difference between a function of the trajectory of a continuous Euler scheme associated to a diffusion process and its mean. Finally, we obtain non-asymptotic bound for stochastic approximation with averaging of trajectories, in particular we prove that averaging a stochastic approximation algorithm with a slow decreasing step sequence gives rise to optimal concentration rate.   
\end{abstract}
\maketitle

\section{Introduction}
In this work, we derive transport-entropy inequalities and, as a consequence, non-asymptotic deviation estimates for the laws at a given time step of two kinds of discrete-time and $d$-dimensional stochastic evolution scheme of the form
\begin{equation}
\label{SCHEME_GEN}
X_{n+1}=X_n + \gamma_{n+1} H(n,X_n,U_{n+1}), \ n\ge 0, X_0=x\in \R^d,
\end{equation}

\noindent where $(\gamma_n)_{n\ge 1} $ is a deterministic positive sequence of time steps, the $(U_i)_{i\in \N^*}$ are i.i.d. $\R^q $-valued random variables defined on some probability space $(\Omega,\F,\P)$ with law $\mu$ and the function $H:\N\times\R^d\times \R^q\rightarrow \R^d$ is a measurable function satisfying for all $x\in \R^d$, for all $n \in \N$, $H(n,x,.) \in \mathcal{L}^{1}(\mu)$, and $\mu(du)$-a.s., $H(n,.,u)$ is continuous. Here and below, we will also assume that $\mu$ satisfies a \textit{Gaussian concentration property}, that is there exists $\beta>0$ such that for every real-valued 1-Lipschitz function $f$ defined on $\R^q$ and for all $\lambda\ge 0$:
\begin{equation}
\label{CONC_GAUSS}
 \E[\exp(\lambda f(U_1))]\le \exp(\lambda \E[f(U_1)]+\frac{\beta\lambda^2}{4})\tag{$GC(\beta)$}.
\end{equation}

\noindent It is well known that \eqref{CONC_GAUSS} implies the following deviation bound 
$$
\P[f(U_1)-\E[f(U_1)]\ge r] \le \exp(-\frac{r^2}{\beta}) \ \ \forall r\ge 0,
$$

Examples of random variables satisfying this property include Gaussians, as well as bounded random variables. A characterization of \eqref{CONC_GAUSS} is given by Gaussian tail of $U_1$, that is there exists $\varepsilon>0$ such that $\E[\exp(\varepsilon |U_1|^2)]<+\infty$, see e.g. Bolley and Villani \cite{boll:vill:05}. The two claims are actually equivalent.

We are interested in furthering the discussion, initiated in \cite{Frikha2012}, about giving non asymptotic deviation bounds for two specific problems related to evolution schemes of the form \eqref{SCHEME_GEN}. The first one is the deviation between a function of an Euler like discretization scheme of a diffusion process at a fixed deterministic date and its mean. The second one refers to the deviation between a stochastic approximation algorithm at a given time-step and its target. Under some mild assumptions, in particular the assumption that the function $u \mapsto H(n, x, u)$ is lipschitz uniformly in space and time, it is proved in \cite{Frikha2012} that both recursive schemes share the Gaussian concentration property of the innovation. 

In the present work, we point out the contribution of the diffusion term to the concentration rate which to our knowledge is new. This covers many situations and gives rise to different regimes ranging from exponential to Gaussian. We also derive a general non-asymptotic deviation bound for the difference between a function of the trajectory of a \emph{continuous Euler scheme} associated to a diffusion process and its mean.  It turns out that, under mild assumptions, the concentration regime is log-normal. Finally, we study non-asymptotic deviation bound for stochastic approximation with averaging of trajectories according to the \emph{averaging principle of Ruppert \& Polyak}, see e.g. \cite{Ruppert1991} and \cite{Polyak1992}. 

\subsection{Euler like Scheme of a Diffusion Process}
\label{DESCR_EUL}

We consider a Brownian diffusion process $(X_{t})_{t\geq0}$ defined on a filtered probability space $(\Omega,\F,(\F_t)_{t\ge 0},\P)$, satisfying the usual conditions, and solution to the following stochastic differential equation (SDE)
\begin{equation}
\label{EDS}
X_t=x+\int_0^t b(s,X_s)ds+\int_0^t\sigma(s,X_s)dW_s, \tag{$SDE_{b,\sigma}$}
\end{equation}

\noindent where $(W_t)_{t\ge 0}$ is a $q$-dimensional $(\F_t)_{t\ge 0} $ Brownian motion and the coefficients $b,\sigma $ are assumed to be uniformly Lipschitz continuous in space and measurable in time.

A basic problem in Numerical Probability is to compute quantities like $\E_x[f(X_T)]$ for a given Lipschitz continuous function $f$ and a fixed deterministic time horizon $T$ using Monte Carlo simulation. For instance, it appears in mathematical finance and represents the price of a European option with maturity $T$ when the dynamics of the underlying asset is given by \eqref{EDS}. Under suitable assumptions on the function $f$ and the coefficients $b,\sigma $, namely smoothness or non degeneracy, it can also be related to the Feynman-Kac representation of the heat equation associated to the generator of $X$. 
To this end, we first introduce some discretization schemes of \eqref{EDS} that can be easily simulated. For a fixed time step $\Delta = T/N, \ N \in \N^{*}$, we set $\ t_i:=i\Delta$, for all $i \in \N$ and define an Euler like scheme by
\begin{eqnarray}
\label{EULER}
X_0^\Delta=x,\
\forall i\in \leftB 0, N-1\rightB, X_{t_{i+1}}^\Delta=X_{t_i}^\Delta+ b(t_i,X_{t_i}^\Delta)\Delta +\sigma(t_i,X_{t_i}^\Delta) \Delta^{1/2} U_{i+1},
\end{eqnarray}

\noindent where $(U_i)_{i\in \N^*}$ is a sequence of $\R^q $-valued i.i.d. random variables with law $\mu$ satisfying: $\E[U_1]=0_q,\ \E[U_{1}U_1^*]=I_{q}$, where $U_1^*$ denotes the transpose of the column vector $U_1$ and $0_q,I_q $ respectively denote the zero vector of $\R^q$ and the identity matrix of $\R^q\otimes \R^q$. We also assume that $\mu$ satisfies \eqref{CONC_GAUSS} for some $\beta>0$. The main advantage of such a situation is that it includes the case of the standard Euler scheme where $U_1\overset{{\rm d}}{=}{\cal N}(0,I_q)$ (satisfying \eqref{CONC_GAUSS} with $\beta=2$) and the case of the Bernoulli law where $U_1\overset{{\rm d}}{=} (B_1,\cdots, B_q),\ (B_k)_{k\in \leftB 1,q\rightB}$ are i.i.d random variables with law $\mu=\frac{1}{2} (\delta_{-1}+\delta_1)$, which turns out to be one of the only realistic options when the dimension is large. 

The weak error $\mathcal{E}_{D}(f, \Delta, T,b, \sigma)  = \E_x[f(X_T)]-\E_x[f(X^{\Delta}_T)]$ corresponds to the discretization error when replacing the diffusion $X$ by its Euler scheme $X^\Delta$ for the computation of $\E_x[f(X_T)]$. It has been widely investigated in the literature. Since the seminal work of \cite{tala:tuba:90}, it is known that, under smoothness assumption on the coefficients $b,\ \sigma$, the \emph{standard Euler scheme} produces a weak error of order $\Delta$. In a hypoelliptic setting for the coefficients $b$ and $\sigma$ and for a bounded measurable function $f$, Bally and Talay obtained the expected order using Malliavin calculus. For a uniformly elliptic diffusion coefficient $\sigma \sigma^{*}$ and if $b,\ \sigma$ are three times continuously differentiable, the same order for the weak error is established in \cite{kona:mamm:02}. The same order $\Delta$ is still valid in the situation where the Gaussian increments are replaced by (non necessarily continuous) random variables $(U_{i})_{1\leq i \leq N}$ having the same covariance matrix and odd moments up to order 5 as the law ${\cal N}(0,I_q)$ and if $b, \ \sigma, \ f$ are smooth enough. Let us finally mention the recent work \cite{Alfonsi2012} where the authors study the \emph{weak trajectorial error} using coupling techniques.  More precisely, they prove that the Wasserstein distance between the law of a uniformly elliptic and one-dimensional diffusion process and the law of  its \emph{continuous Euler scheme} $X^{c,\Delta}$ with time step $\Delta:=T/N$ is smaller than $\mathcal{O}(N^{-2/3+\epsilon})$, $\forall \epsilon>0$. 

The expansion of $\mathcal{E}_{D}$ also allows to improve the convergence rate to $0$ of the discretization error using Richardson-Romberg extrapolation techniques, see e.g. \cite{tala:tuba:90}.

In order to have a global control of the numerical procedure for the computation of $\E_x[f(X_T)]$, it remains to approximate the expectation $\E_x[f(X^{\Delta}_T)]$ using a Monte Carlo estimator $M^{-1} \times \sum_{k=1}^{M} f((X_T^{\Delta,x})^j)$ where the $((X_T^{\Delta,x})^j)_{j\in \leftB 1,M\rightB}$ are $M$ independent copies of the scheme \eqref{EULER} starting at the initial value $x$ at time $0$. This gives rise to an \emph{empirical error} defined by $\mathcal{E}_{Emp}(M, f, \Delta, T, b, \sigma) = \E_x[f(X^{\Delta}_T)]-M^{-1} \times \sum_{j=1}^{M} f((X_T^{\Delta,x})^j)$. Consequently, the global error associated to the computation of $\E_x[f(X_T)]$ writes as
\begin{align*}
\mathcal{E}_{Glob}(M,\Delta) & = \E_x[f(X_T)] - \E_x[f(X^{\Delta}_T)] + \E_x[f(X^{\Delta}_T)] - \frac{1}{M} \times \sum_{j=1}^{M} f((X_T^{\Delta,x})^j) \\
& := \mathcal{E}_{D}(f, \Delta, T,b, \sigma) + \mathcal{E}_{Emp}(M, f, \Delta, T, b, \sigma).
\end{align*}

It is well-known that if $f(X_T^{\Delta,x})$ belongs to $L^{2}(\P)$ the central limit theorem provides an \emph{asymptotic} rate of convergence of order $M^{1/2}$. If $f(X_T^{\Delta,x}) \in L^{3}(\P)$, a non-asymptotic result is given by the Berry-Essen theorem. However, in practical implementation, one is interested in obtaining deviation bounds in probability for a fixed $M$ and a given threshold $r>0$, that is explicitly controlling the quantity $\P\left(\mathcal{E}_{Emp}(M,\Delta) \geq r\right)$. 

In this context, Malrieu and Talay \cite{Malrieu2006} obtained Gaussian deviation bounds in an ergodic framework and for a constant diffusion coefficient. Concerning the \emph{standard} Euler scheme, Menozzi and Lemaire \cite{lema:meno:10} obtained two-sided Gaussian bounds up to a systematic bias under the assumptions that the diffusion coefficient is uniformly elliptic, $\sigma\sigma^{*}$ is H\"older-continuous, bounded and that $b$ is bounded. Frikha and Menozzi \cite{Frikha2012}, getting rid of the non-degeneracy assumption on $\sigma$, recently obtained Gaussian deviation bound under the mild smoothness condition that $b, \ \sigma$ are uniformly Lipschitz-continuous in space (uniformly in time) and that $\sigma$ is bounded. The main tool of their analysis is to exploit similar decompositions used in \cite{tala:tuba:90} for the analysis of the weak error. It should be noted that it is the boundedness of $\sigma$ that gives rise to the Gaussian concentration regime for the deviation of the \emph{empirical error}.

Using optimal transportation techniques, Blower and Bolley \cite{blow:boll:06} obtained Gaussian concentration inequalities and transportation inequalities for the joint law of the first $n$ positions of a stochastic processes with state space some Polish space. However, continuity assumptions in Wasserstein metric need to be checked which can be hard in practice, see conditions (ii) in their Theorems 1.1, 1.2 and 2.1. The authors provide a computable sufficient condition which notably requires the smoothness of the transition law, see Proposition 2.2. in \cite{blow:boll:06}.  

In the current work, we get rid of the boundedness of $\sigma$ and we only need the Gaussian concentration property of the innovation. We suppose that the coefficients satisfy the following smoothness and domination assumptions
\begin{trivlist}
\item[\A{HS}] The coefficients $b,\sigma$ are uniformly Lipschitz continuous in space uniformly in time.
\item[\A{HD$_{\alpha}$}] There exists a $\mathcal{C}^{2}(\R^{d},\R^{*}_{+})$ function $V$ satisfying $\exists C_{V}>0, |\nabla V|^2 \leq C_V V, \ \eta:=\frac{1}{2}\sup_{x \in \R^{d}} \left\|\nabla^{2} V(x)\right\| <+\infty$ and $\exists \alpha \in (0,1]$, such that for all $x\in \R^{d}$,
$$
\exists C_b>0, \ \ \sup_{t\in [0,T]}|b(t,x)|^2 \leq C_b V(x), \ \ , \ \exists C_{\sigma}>0, \ \ \sup_{t\in [0,T]}Tr(a(t,x)) \leq C_{\sigma} V^{1-\alpha}(x). 
$$

\noindent where $a=\sigma \sigma^{*}$.
\end{trivlist}

The idea behind assumption \A{HD$_{\alpha}$} is to parameterize the growth of the diffusion coefficient in order to quantify its contribution to the concentration regime. Indeed, under \A{HS} and \A{HD$_{\alpha}$}, with $\alpha \in [1/2,1]$, and if the innovations satisfy \eqref{CONC_GAUSS}, for some positive $\beta$, we derive non-asymptotic deviation bounds for the statistical error $E_M^\Delta(x,T,f)-\E_x[f(X_T^\Delta)]$ ranging from exponential (if $\alpha=1/2$) to Gaussian (if $\alpha=1$) regimes. Therefore, we greatly improve the results obtained in \cite{Frikha2012}.
 
 Our approach here is different from \cite{Frikha2012}. Indeed, in \cite{Frikha2012}, the key tool consists in writing the deviation using the same kind of decompositions that are exploited in \cite{tala:tuba:90} for the analysis of the discretization error. In the current work, we will use the fact that the Euler-like scheme \eqref{EULER} defines an inhomogenous Markov chain having Feller transitions $P_k$, $k=0,\cdots, N-1$, defined for non negative or bounded Borel function $f:\mathbb{R}^d\rightarrow \mathbb{R}$ by
$$
P_k(f)(x) = \E\left[\left. f(X_{t_{k+1}}^\Delta) \right| X_{t_{k}}^\Delta=x\right] = \mathbb{E}\left[f\left(x+ b(t_k,x)\Delta +\sigma(t_k,x) \Delta^{1/2} U \right)\right], \ \ k=0,\cdots,N-1.
$$

For every $k, \ p \in \left\{0, \cdots, N-1\right\}$, $k\leq p$, we also define the iterative kernels for a non negative or bounded Borel function $f:\mathbb{R}^d\rightarrow \mathbb{R}$
$$
P_{k,p}(f)(x) = P_{k}\circ \cdots \circ P_{p-1}(f)(x) = \mathbb{E}\left[\left. f(X_{t_{p}}^\Delta) \right|  X_{t_{k}}^\Delta=x\right].
$$

For a $1$-Lipschitz function $f$ and $\lambda\geq0$, using that the law $\mu$ of the innovation satisfies \eqref{CONC_GAUSS} for some positive $\beta$, we obtain
\begin{eqnarray*}
P_{N-1}(\exp(\lambda f))(x) & = &\mathbb{E}\left[\exp\left(\lambda f\left(x+ b(t_{N-1},x)\Delta +\sigma(t_{N-1},x) \Delta^{1/2} U \right)\right)\right] \\
 & \leq & \exp\left(\lambda P_{N-1}(f)(x) + \beta \frac{\lambda^2}{4} \Delta |\sigma(t_{N-1},x)|^2 \right)
\end{eqnarray*}

If $\sigma$ is bounded, the Gaussian concentration property will readily follow provided the iterated kernel functions $P_{k,p}(f)$ are uniformly Lipschitz. Under the mild smoothness assumption \A{HS}, this can be easily derived, see Proposition \ref{Lipschitz_control_Euler}. Otherwise, using \A{HD$_{\alpha}$}, we obtain
\begin{equation}
\label{first_step_to_conc_Euler}
P_{N-1}(\exp(\lambda f))(x) \leq \exp\left(\lambda P_{N-1}(f)(x) + \frac{C_{\sigma} \beta \Delta}{4}  \lambda^2 V^{1-\alpha}(x) \right).
\end{equation}

 The last inequality is the first step of our analysis. To investigate the empirical error, the key idea is to exploit recursively from \eqref{first_step_to_conc_Euler} that the increments of the scheme \eqref{EULER} satisfy \eqref{CONC_GAUSS} and to adequately quantify the contribution of the diffusion term $V^{1-\alpha}(x)$ to the concentration rate. Under \A{HS} and \A{HD$_{\alpha}$}, the latter is addressed using flow techniques and integrability results on the law of the scheme \eqref{EULER}, see Propositions \ref{CONT_LAPLACEV} and \ref{CONT_LAPLACE}.

 \subsection{Stochastic Approximation Algorithm}
 \label{RM_SEC}
Beyond concentration bounds of the empirical error for Euler-like schemes, we want to look at non asymptotic bounds for stochastic approximation algorithms. Introduced by H. Robbins and S. Monro \cite{Robbins1951}, these recursive algorithms aim at finding a zero of a continuous function $h:\mathbb{R}^{d} \rightarrow \mathbb{R}^{d}$ which is unknown to the experimenter but can only be estimated through experiments. Successfully and widely investigated since this seminal work, such procedures are now commonly used in various contexts such as convex optimization since minimizing a function amounts to finding a zero of its gradient. 

To be more specific, the aim of such an algorithm is to find a solution $\theta^{*}$ to the equation $h(\theta):=\mathbb{E}[H(\theta,U)]=0$, where $H: \mathbb{R}^{d} \times \mathbb{R}^{q} \rightarrow \mathbb{R}^{d}$ is a Borel function and $U$ is a given $\mathbb{R}^{q}$-valued random variable with law $\mu$. The function $h$ is generally not computable, at least at a reasonable cost. Actually, it is assumed that the computation of $h$ is costly compared to the computation of $H$ for any couple $(\theta,u) \in \mathbb{R}^{d}\times \mathbb{R}^{q}$ and to the simulation of the random variable $U$.

A stochastic approximation algorithm corresponds to the following simulation-based recursive scheme
\begin{equation}
\label{RM}
\theta_{n+1} = \theta_{n} - \gamma_{n+1} H(\theta_{n},U_{n+1}), \ n \geq 0, \ \theta_{0} \in \mathbb{R}^{d},
\end{equation}

\noindent where $(U_{n})_{n\geq1}$ is an i.i.d. $\mathbb{R}^{q}$-valued sequence of random variables with law $\mu$ defined on a probability space $(\Omega, \mathcal{F}, \mathbb{P})$ and $(\gamma_{n})_{n\geq1}$ is a sequence of non-negative deterministic steps satisfying the usual assumption
\begin{equation}
\label{STEP}
\sum_{n\geq1} \gamma_{n} = + \infty, \ \ \mbox{and} \ \ \sum_{n\geq1} \gamma_{n}^{2} < + \infty.
\end{equation}

\noindent When the function $h$ is the gradient of a potential, the recursive procedure \eqref{RM} is a stochastic gradient algorithm. Indeed, replacing $H(\theta_n,U_{n+1})$ by $h(\theta_n)$ in \eqref{RM} leads to the usual deterministic descent gradient method. When $h(\theta)=M(\theta)-\ell$, $\theta \in \mathbb{R}$, where $M$ is a monotone function, say increasing, we can write $M(\theta)=\E[N(\theta,U)]$ where $N: \mathbb{R} \times \mathbb{R}^{q} \rightarrow \mathbb{R}$ is a Borel function and $\ell$ is a given constant such that the equation $M(\theta)=\ell$ has a solution. Setting $H=N-\ell$, the recursive procedure \eqref{RM} then corresponds to the seminal Robbins-Monro algorithm and aims at computing the level of the function $M$. 

The key idea of stochastic approximation algorithms is to take advantage of an averaging effect along the scheme due to the specific form of $h(\theta):=\E[H(\theta,U)]$. This allows to avoid the numerical integration of $h$ at each step of a classical first-order optimization algorithm.

\noindent In the present paper, we make no attempt to provide a general discussion concerning convergence results of stochastic approximation algorithms. We refer readers to \cite{Duflo1996}, \cite{Kushner2003} for some general results on the $a.s.$ convergence of such procedures under the existence of a so-called \emph{Lyapunov function}, $i.e.$ a continuously differentiable function $L:\mathbb{R}^{d}\rightarrow \mathbb{R}_{+}$ such that $\nabla L$ is Lipschitz, $|\nabla L|^{2} \leq C(1+L)$ for some positive constant $C$ and 
$$
 \left\langle \nabla L , h \right\rangle \geq 0.
$$

\noindent See also \cite{Laruelle2012} for a convergence theorem under the existence of a \emph{pathwise Lyapunov function}. For the sake of simplicity, in the sequel it is assumed that $\theta^{*}$ is the unique solution of the equation $h(\theta)=0$ and that the sequence $(\theta_{n})_{n\geq0}$ defined by \eqref{RM} converges $a.s.$ towards $\theta^{*}$.

We assume that the law $\mu$ of the innovation satisfies \eqref{CONC_GAUSS} for some $\beta>0$ and that the step sequence $(\gamma_{n})_{n\geq1}$ satisfies \eqref{STEP}. We also suppose that the following assumptions on the function $H$ are in force:
\begin{trivlist}
\item[\A{HL}] For all $u\in \R^{q}$, the function $H(.,u)$ is Lipschitz-continuous with a Lipschitz modulus having linear growth in the variable $u$, that is: 
$$
\exists C_H>0, \ \forall u \in \R^{q}, \ \sup_{(\theta, \theta^{'}) \in (\R^{d})^{2}} \frac{|H(\theta,u)-H(\theta^{'},u)|}{|\theta-\theta^{'}|} \leq C_{H} (1+|u|).
$$

\item[\A{HLS}$_\alpha$] (\textit{Lyapunov Stability-Domination})  There exists a $\mathcal{C}^{2}(\R^{d},\R^{*}_+)$ function $L$ satisfying $\exists C_{L}>0, |\nabla L|^2 \leq C_L L, \ \eta:=\frac{1}{2}\sup_{x \in \R^{d}} \left\|\nabla^{2} L(x)\right\| <+\infty$ such that
$$
\forall \theta \in \R^d, \ \  \left\langle h(\theta), \nabla L(\theta) \right\rangle \geq  0, \ \ \ \mbox{ and } \ \ \exists C_{h}>0, \ \forall \theta \in \R^d, \ |h(\theta)|^2 \leq C_{h} L(\theta).
$$

\noindent and $\exists \alpha \in (0,1]$,
$$
  \ \ \ \exists C_{\alpha}>0, \ \forall \theta \in \R^{d}, \ \sup_{(u,u^{'}) \in (\R^{q})^{2}} \frac{|H(\theta,u)-H(\theta,u^{'})|}{|u-u^{'}|} \leq C_{\alpha} L^{\frac{1-\alpha}{2}}(\theta) 
$$

\item[\A{HUA}] (\textit{Uniform Attractivity}) The map $h:\theta \in \R^d \mapsto \E[H(\theta,U)]$ is continuously differentiable in $\theta $ and there exists $\underline{\lambda}>0 $ s.t. $\forall \theta \in \R^{d}, \ \forall \xi\in \R^d, \  \underline{\lambda} |\xi|^2 \le \langle Dh(\theta)\xi,\xi\rangle$. 
\end{trivlist}

Compared to \cite{Frikha2012}, our assumptions are weaker. Indeed, it is assumed in \cite{Frikha2012} that the map $(\theta, u) \in \R^d \times \R^q \mapsto H(\theta,u)$ is uniformly Lipschitz continuous. In our current framework, this latter assumption is replaced by \A{HL} and \A{HLS}$_\alpha$. 

The last assumption \A{HUA}, which already appeared in \cite{Frikha2012}, is introduced to derive a sharp estimate of the concentration rate in terms of the step sequence. Let us note that such assumption appears in the study of the weak convergence rate order for the sequence $(\theta_{n})_{n\geq1}$ as described in \cite{Duflo1996} or \cite{Kushner2003}. Indeed, it is commonly assumed that the matrix $Dh(\theta^{*})$ is \emph{uniformly attractive} that is $\mathcal{R}e(\lambda_{min})>0$ where $\lambda_{min}$ is the eigenvalue with the smallest real part. In our current framework, this local condition on the Jacobian matrix of $h$ at the equilibrium is replaced by the uniform assumption \A{HUA}. This allows to derive sharp estimates for the concentration rate of the sequence $(\theta_n)_{n\geq1}$ around its target $\theta^{*}$ and to provide a sensitivity analysis for the bias $\delta_n:=\E[|\theta_{n}-\theta^*|]$ with respect to the starting point $\theta_0$. %The sensitivity of the scheme \eqref{RM} w.r.t. to the initial point is only reflected in the bias $\delta_n:=\E[|\theta_{n}-\theta^*|] $.

Let us note that under \A{HUA} and the \emph{linear growth assumption}
$$
\forall \theta \in \R^{d}, \ \ \E\left[\left|H(\theta,U)\right|^{2}\right] \leq C(1+|\theta-\theta^{*}|^{2}),
$$

\noindent which is satisfied if \A{HL} and \A{HLS}$_\alpha$, with $\alpha \in [0,1]$, hold and if $\mu$ satisfies \eqref{CONC_GAUSS} for some $\beta>0$, the function $L:\theta \mapsto \frac{1}{2}\left|\theta - \theta^{*}\right|^{2}$ is a Lyapunov function for the recursive procedure defined by \eqref{RM} so that one easily deduces that $\theta_{n} \rightarrow \theta^{*}$, $a.s.$ as $n\rightarrow + \infty$.

The global error between the stochastic approximation procedure $\theta_{n}$ at a given time step $n$ and its target $\theta^{*}$ can be decomposed as \emph{an empirical error} and \emph{a bias} as follows
\begin{eqnarray}
 \left|\theta_{n} - \theta^{*}\right| & = & \left|\theta_{n} - \theta^{*}\right| - \E_{\theta_0}[\left|\theta_{n} - \theta^{*}\right|] + \E_{\theta_0}[\left|\theta_{n} - \theta^{*}\right|] \nonumber\\
      &  := & \mathcal{E}_{Emp}(\gamma,n,H,\underline{\lambda}, \alpha) + \delta_{n}\label{DEC_RM}
\end{eqnarray}

\noindent where we introduced the notations $\mathcal{E}_{Emp}(\gamma,n,H,\underline{\lambda}, \alpha)=\left|\theta_{n} - \theta^{*}\right| - \E_{\theta_0}[\left|\theta_{n} - \theta^{*}\right|]$ and $\delta_n := \E_{\theta_0}[\left|\theta_{n} - \theta^{*}\right|]$.

The \emph{empirical error} $\mathcal{E}_{Emp}(\gamma,n,H,\underline{\lambda}, \alpha)$ is the difference between the absolute value of the error at time $n$ and its mean whereas \emph{the bias} $\delta_n$ corresponds to the mean of the absolute value of the difference between the sequence  $(\theta_{n})_{n\geq0}$ at time $n$ and its target $\theta^{*}$. Unlike the Euler like scheme, a bias systematically appears since we want to derive a deviation bound for the difference between $\theta_n$ and its target $\theta^{*}$. This term strongly depends on the choice of the step sequence $(\gamma_{n})_{n\geq1}$ and the initial point $\theta_0$, see Proposition \ref{CTR_BIAS} for a sensitivity analysis.

 As for Euler like schemes, our strategy is different from \cite{Frikha2012}. Indeed, we exploit again the fact that the stochastic approximation scheme \eqref{RM} defines an inhomogenous Markov chain having Feller transitions $P_k$, $k=0,\cdots, N-1$, defined for non negative or bounded Borel function $f:\mathbb{R}^d\rightarrow \mathbb{R}$ by
$$
P_k(f)(\theta) = \E\left[\left. f(\theta_{k+1}) \right| \theta_{k}=\theta\right] = \mathbb{E}\left[f\left(\theta-\gamma_{k+1}H(\theta,U) \right)\right], \ \ k=0,\cdots,N-1.
$$

For every $k, \ p \in \left\{0, \cdots, N-1\right\}$, $k\leq p$, we also define the iterative kernels for a non negative or bounded Borel function $f:\mathbb{R}^d\rightarrow \mathbb{R}$ as follows
$$
P_{k,p}(f)(\theta) = P_{k}\circ \cdots \circ P_{p-1}(f)(\theta) = \mathbb{E}\left[\left. f(\theta_{p}) \right|  \theta_{k}=\theta\right].
$$

For a $1$-Lipschitz function $f$ and for all $\lambda\geq0$, using \A{HLS}$_\alpha$ and that the law $\mu$ of the innovation satisfies \eqref{CONC_GAUSS} for some positive $\beta$, we obtain
\begin{eqnarray}
P_{N-1}(\exp(\lambda f))(\theta) & = &\mathbb{E}\left[\exp\left(\lambda f\left(\theta- \gamma_{N} H(\theta,U) \right)\right)\right]  \leq \exp\left(\lambda P_{N-1}(f)(\theta) + \beta \frac{\lambda^2}{4} C^2_{\alpha} \gamma^2_{N} L^{1-\alpha}(\theta)\right) \label{first_step_to_conc_RM}
\end{eqnarray}

Let us note the similarity between \eqref{first_step_to_conc_Euler} and \eqref{first_step_to_conc_RM}. If \A{HLS}$_\alpha$ holds with $\alpha=1$ then the last term appearing in the right hand side of the last inequality is uniformly bounded in $\theta$. This latter assumption corresponds to the framework developed in \cite{Frikha2012} and leads to a Gaussian concentration bound. 

Otherwise, the problem is more challenging. Under the mild domination assumption \A{HLS}$_\alpha$, the key idea consists again in exploiting recursively from \eqref{first_step_to_conc_RM} that the increments of the stochastic approximation algorithm \eqref{RM} satisfy \eqref{CONC_GAUSS} and in properly quantifying the contribution of the diffusion term $L^{1-\alpha}(\theta)$ to the concentration rate.

As already noticed in \cite{Frikha2012}, the concentration rate and the bias strongly depends on the choice of the step sequence. In particular, if $\gamma_n = \frac{c}{n}$, with $c>0$ then the optimal concentration rate and bias is achieved if $c>\frac{1}{2\underline{\lambda}}$, see Theorem 2.2. in \cite{Frikha2012}. Otherwise, they are sub-optimal. This kind of behavior is well-known concerning the weak convergence rate for stochastic approximation algorithm. Indeed, if $c > \frac{1}{2 \mathcal{R}e(\lambda_{min})}$ we know that a Central Limit Theorem holds for the sequence $(\theta_{n})_{n\geq1}$ (see e.g. \cite{Duflo1996}). Let us note that the condition $c>\frac{1}{2\underline{\lambda}}$ as well as $c > \frac{1}{2 \mathcal{R}e(\lambda_{min})}$ is difficult to handle and may lead to a blind choice in practical implementation. 

To circumvent such a difficulty, it is fairly well-known that the key idea is to carefully smooth the trajectories of a converging stochastic approximation algorithm by averaging according to the \emph{Ruppert \& Polyak averaging principle}, see e.g. \cite{Ruppert1991} and \cite{Polyak1992}. It consists in devising the original stochastic approximation algorithm \eqref{RM} with a slow decreasing step 
$$
\gamma_{n} = \left(\frac{c}{b+n}\right)^{\nu}, \ \ \nu \in \left(\frac{1}{2},1\right), c,b>0,
$$

\noindent and to simultaneously compute the empirical mean $(\bar{\theta}_{n})_{n\geq1}$ of the sequence $(\theta_{n})_{n\geq0}$ by setting
\begin{align}
\label{RM_AV}
\bar{\theta}_{n} & = \frac{\theta_0 + \cdots + \theta_{n-1}}{n} = \bar{\theta}_{n-1} - \frac{1}{n}\left(\bar{\theta}_{n-1}-\theta_{n-1}\right).
\end{align}

We will not enter into the technicalities of the subject but under mild assumptions (see e.g. \cite{Duflo1996}, p.169) one shows that 
$$
\sqrt{n}(\bar{\theta}_{n}-\theta^{*}) \stackrel{\mathcal{L}}{\rightarrow} \mathcal{N}(0,\Sigma^{*}), \ n\rightarrow + \infty,
$$

\noindent where $\Sigma^{*}$ is the optimal covariance matrix. For instance, for $d=1$, one has $\Sigma^{*} = \frac{Var(H(\theta^{*},U))}{(h^{'}(\theta^{*}))^2}$. Hence, the optimal weak rate of convergence $\sqrt{n}$ is achieved for free without any condition on the constants $c$ or $b$. However, this result is only asymptotic and so far, to our best knowledge, non-asymptotic estimates for the deviation between the empirical mean sequence $(\bar{\theta}_{n})_{n\geq0}$ at given time step and its target $\theta^{*}$, that is non-asymptotic averaging principle were not investigated. 

The sequence $(z_{n})_{n\geq0}$ defined by $z_n:=(\bar{\theta}_{n+1},\theta_n)$ is $\mathcal{F}$-adapted, i.e. for all $n\geq0$, $z_{n}$ is $\mathcal{F}_n$-measurable, where $\mathcal{F}_n:=\sigma(\theta_0, U_k, k \leq n)$. Moreover, it defines an inhomogenous Markov chain having Feller transitions $K_k$, $k=0,\cdots, N-1$, defined for non negative or bounded Borel function $f:\R^d \times \R^d \rightarrow \R$ by
\begin{align*}
K_k(f)(z) & = \E[\left. f(z_{k+1}) \right| z_{k} = z ] = \E[\left. f(\bar{\theta}_{k+2},\theta_{k+1}) \right| (\bar{\theta}_{k+1}, \theta_k)= (z_1,z_2)],  \\
& = \E\left[f\left(\frac{k+1}{k+2}z_1+\frac{1}{k+2}(z_2-\gamma_{k+1}H(z_2,U)), z_2-\gamma_{k+1}H(z_2,U)\right)\right].
\end{align*}

For every $k,p \in \left\{0,\cdots, N-1\right\}$, $k\leq p$, we define the iterative kernels for a non negative or bounded Borel function $f: \R^d \times \R^d \rightarrow \R$
$$
K_{k,p}(f)(z) = K_k \circ \cdots K_{p-1}(f)(z) = \E[\left. f(z_p) \right| z_k=z].
$$

Hence, for any $1$-Lipschitz function and for all $\lambda \geq 0$, using again \A{HLS}$_\alpha$ and that the law $\mu$ of the innovation satisfies \eqref{CONC_GAUSS} for some positive $\beta$, one has for all $k \in \left\{0, \cdots, N-1\right\}$
\begin{align}
K_{k}(\exp(\lambda f))(z) & = \mathbb{E}\left[ \left. \exp\left(\lambda f\left(z_{k+1}\right)\right) \right| z_{k} = z\right] \nonumber \\
 & \leq  \exp\left(\lambda K_{k}(f)(z) + \beta \frac{\lambda^2}{4} \left(C_{\alpha} \gamma_{k+1}(\frac{1}{k+2}+1) L^{\frac{1-\alpha}{2}}(z_2) \right)^2 \right) \nonumber \\
 & \leq \exp\left(\lambda K_{k}(f)(z) + \beta \lambda^2 C^2_{\alpha} \gamma^2_{k+1} L^{1-\alpha}(z_2)\right) \label{first_step_to_conc_RM_AV}
\end{align}

\noindent where we used that for all $(z_1,z_2) \in \R^d \times \R^d$, the functions $u\mapsto f\left(\frac{k+1}{k+2}z_1+\frac{1}{k+2}(z_2-\gamma_{k+1}H(z_2,u)), z_2-\gamma_{k+1}H(z_2,u)\right)$ are Lipschitz-continuous with Lipschitz modulus equals to $C_{\alpha} \gamma_{k+1}(\frac{1}{k+2}+1) L^{\frac{1-\alpha}{2}}(z_2)$. 

Here again, \eqref{first_step_to_conc_RM} and \eqref{first_step_to_conc_RM_AV} are quite similar and if $\alpha=1$ the concentration regime turns out to be Gaussian. Otherwise, an analysis along the lines of the methodology developed so far provides the concentration regime of the stochastic approximation algorithm with averaging of trajectories. 

\subsection{Transport-Entropy inequalities}
As a by-product of our analysis, we derive transport-entropy inequalities for the law of both stochastic approximation schemes. We recall here basic definitions and properties. For a complete overview and recent developments in the theory of transport inequalities, the reader may refer to the recent survey \cite{goz:leon:2010}. We will denote by $\mathcal{P}(\R^d)$ the set of probability measures on $\R^d$.

For $p\geq1$, we consider the set $\mathcal{P}_p(\R^d)$ of probability measures with finite moment of order $p$. The Wasserstein metric $W_p(\mu,\nu)$ of order $p$  between two probability measures $\mu,  \nu \in \mathcal{P}_p(\R^d)$ is defined by 
$$
W^p_p(\mu, \nu) = \inf\left\{\int_{\R^d \times \R^d} |x-y|^p \pi(dx,dy): \ \pi \in \mathcal{P}(\R^d \times \R^d), \ \pi_0=\mu, \ \pi_1 = \nu \right\}
$$

\noindent where $\pi_0$ and $\pi_1$ are two probability measures standing for the first and second marginals of $\pi \in \mathcal{P}(\R^d \times \R^d)$. For $\mu \in \mathcal{P}(\R^d)$, we define the relative entropy w.r.t $\nu \in \mathcal{P}(\R^d)$ as
$$
H(\mu,\nu)=\int_{\R^d} \log\left( \frac{d\mu}{d\nu}\right) d\mu
$$

\noindent if $\mu \ll \nu$ and $H(\mu,\nu)=+\infty$ otherwise. We are now in position to define the notion of transport-entropy inequality. Here as below, $\Phi: \R_+ \rightarrow \R_+$ is a convex, increasing function with $\Phi(0)=0$.
\begin{DEF}
A probability measure $\mu$ on $\R^d$ satisfies a transport-entropy inequality with function $\Phi$ if for all $\nu \in \mathcal{P}(\R^d)$, one has
$$
\Phi(W_1(\nu,\mu)) \leq H(\nu,\mu)
$$

\noindent For the sake of simplicity, we will write that $\mu$ satisfies $T_\Phi$. 
\end{DEF}

The following proposition comes from Corollary 3.4. of \cite{goz:leon:2010}.
\begin{PROP} The following propositions are equivalent:
\begin{itemize}

\item The probability measure $\mu$ satisfies $T_\Phi$.

\item For all 1-Lipschitz function $f$, one has
$$
\forall \lambda \geq 0, \ \ \int \exp(\lambda f ) d\mu \leq \exp\left(\lambda \int f d\mu+ \Phi^{*}(\lambda)\right), 
$$

\noindent where $\Phi^*$ is the monotone conjugate of $\Phi$ defined on $\R_+$ as $\Phi^*(\lambda) = \sup_{\rho \geq 0 }\left\{\lambda \rho - \Phi(\rho)\right\}$.
\end{itemize}

\end{PROP}

Such transport-entropy inequalities are very attractive especially from a numerical point of view since they are related to the concentration of measure phenomenon which allows to establish non-asymptotic deviation estimates. The three next results put an emphasis on this point. Suppose that $(X_n)_{n\geq 1}$ is a sequence of independent and identically distributed $\R^d$-valued random variables with common law $\mu$.
 
\begin{COROL}If $\mu$ satisfies $T_\Phi$ then for all 1-Lipschitz function $f$ and for all $r\geq0$, for all $M\geq 1$, one has
$$
\P\left(|\frac1M \sum_{k=1}^M f(X_k)-\E[f(X_1)] | \geq r\right) \leq 2\exp(-M\Phi(r))
$$
\end{COROL}

\begin{PROP}
If $\mu$ satisfies $T_\Phi$ then the empirical measure $\mu^{n}$ defined as $\mu^{n}= \frac{1}{n} \sum_{k=1}^n \delta_{X_k}$ satisfies the following concentration bound
$$
\P\left( W_1(\mu^n,\mu)\geq \E[W_1(\mu^n,\mu)]+r\right)\leq \exp\left(-n\Phi(r)\right).
$$

\noindent where for $x \in \R^d$, $\delta_{x}$ stands for the Dirac mass at point $x$.
\end{PROP}

The quantity $\E[W_1(\mu,\mu^n)]$ will go to zero as $n$ goes to infinity, by convergence of empirical measures, but we still need quantitative bounds. The next result is an adaptation of a result of \cite{rach:rusch:1998} on similar bounds but for the distance $W_2$. For sake of completeness, we provide a proof in Appendix \ref{APPEND_A}.

\begin{PROP}
\label{PROP_CONT_W1}
Assume that $\mu$ has a finite moment of order $d+3$. Then, one has
$$
\E[W_1(\mu^n \mu)]\leq C(d,\mu) n^{-1/(d+2)}
$$

\noindent  where 
$$
C(d,\mu) := 4\sqrt{d} + 2\sqrt{\int_{\R^d}{(1 + |x|^{d+1})^{-1}dx}}\sqrt{2^{-2d} + 2^{3-d}\int{|y|^{d+3}\mu(dy)} + 2^{3-d}d(d+3)!}.
$$
\end{PROP}

\noindent In view of Kantorovich-Rubinstein duality formula, namely
$$
W_1(\mu,\nu) = \sup\left\{ \int f d\mu - \int f d\nu: [f]_1 \leq 1  \right\}
$$

\noindent where $[f]_1$ denotes the Lipschitz-modulus of $f$, the latter result provides the following concentration bounds
$$
\forall r\geq0, \ \forall M\geq 1, \ \ \P\left(\sup_{f: [f]_1\leq 1} \left(\frac{1}{M} \sum_{k=1}^M f(X_k) - \E[f(X_1)] \right)\geq  C(d,\mu) M^{-1/(d+2)}+r\right) \leq \exp\left(-M\Phi(r)\right).
$$

Similar results were first obtained for different concentration regimes by Bolley, Guillin, Villani \cite{boll:guil:vill:07} relying on a non-asymptotic version of Sanov's Theorem. Some of these results have also been derived by Boissard \cite{bois:11} using concentration inequalities, and were also extended to ergodic Markov chains up to some contractivity assumptions in the Wasserstein metric on the transition kernel. 

Some applications are proposed in  \cite{boll:guil:vill:07}. Such results can indeed provide non-asymptotic deviation bounds for the estimation of the density of the invariant measure of a Markov chain. Let us note that the (possibly large) constant $C(d,\mu)$ appears as a trade-off to obtain uniform deviations over all Lipschitz functions. 

As a consequence of  the transport-entropy inequalities obtained for the laws at a given time step of Euler like schemes and stochastic approximation algorithm, we will derive non-asymptotic deviation bounds in the Wasserstein metric.

\mysection{Main Results}
 
\subsection{Euler like schemes and diffusions}
\begin{THM}[Transport-Entropy inequalities for Euler like schemes]
\label{THM_TE_EULER}
Denote by $X^{\Delta,0,x}_T$ the value at time $T$ of the scheme \eqref{EULER} associated to the diffusion \eqref{EDS} starting at point $x$ at time $0$. Denote the Lipschitz modulus of $b$ and $\sigma$ appearing in the diffusion process \eqref{EDS} by $[b]_1$ and $[\sigma]_1$, respectively and by $\mu^{\Delta,0,x}_T$ the law of $X^{\Delta,0,x}_T$. Assume that  the innovations $(U_i)_{i\geq 1} $ in \eqref{EULER} satisfy \eqref{CONC_GAUSS} for some $\beta>0 $ and that the coefficients $b,\sigma$ satisfy \A{HS} and \A{HD$_{\alpha}$} for $\alpha \in [\frac12,1]$. 

Then, $\mu^{\Delta,0,x}_T$ satisfies $T_{\Phi^{*}_{\alpha}}$ with $\Phi^{*}_{\alpha}(\lambda)=\sup_{\rho \geq 0}\left\{\lambda \rho - \Phi_{\alpha}(\rho)\right\}$

\noindent with: 
\begin{itemize}

\item If $ \alpha\in (\frac12,1]$, for all $\rho\geq0$
$$
\Phi_{\alpha}(\rho) = \Psi_{\alpha}(T,\Delta,b,\sigma, x)(\rho^2 \vee \rho^{\frac{2\alpha}{2\alpha-1}}),
$$

\noindent with $\Psi_{\alpha}(T,\Delta,b,\sigma, x) = K_{3.1}(\varphi(T,b,\sigma,\Delta)^{2} \vee\varphi(T,b,\sigma,\Delta)^{\frac{\alpha}{2\alpha-1}})$, $\varphi(T,b,\sigma,\Delta)= C_{\sigma}\beta\frac{(1+C(\Delta)\Delta)}{4C(\Delta)}e^{3C(\Delta) T}$, $C(\Delta) := 2[b]_1 + [\sigma]^2_{1} + \Delta [b]^2_1$ and the constant $K_{3.1}$ being defined in Corollary \ref{CONTLAPLACEBIS}.

%\item If $\alpha=\frac12$, for all $\rho \in [0, \varphi(T,b,\sigma,\Delta)^{-\frac{1}{2}} \min(1, \varepsilon_{\beta}^{\frac{1}{2}} (\eta C_{\sigma} T \exp(C T))^{-\frac{1}{2}}))$
%$$
%\Phi_{1/2}(\rho) = \rho^2 \varphi(T,b,\sigma,\Delta) V^{\frac{1}{2}}(x) + \frac{1}{2} \log\left(\E\left[e^{\rho^2 \varphi(T,b,\sigma,\Delta)\eta C_\sigma T e^{CT}|U_1|^2}\right]\right)
%$$
%
%
%\noindent and $\Phi_{1/2}(\rho) = + \infty$ otherwise, with $C= \frac{1}{2}(C_V C_{b})^{\frac{1}{2}} + 2 (1+2\eta \Delta)^2+ \frac{1}{2} \eta C_b \Delta$.   

\item If $\alpha=\frac12$, for all $\rho \in [0,\varphi(T,b,\sigma,\Delta)^{-1/2} \lambda_{3.2})$
$$
\Phi_{1/2}(\rho) = K_{3.2} \frac{(\rho\varphi(T,b,\sigma,\Delta)^{1/2} /\lambda_{3.2})^2}{1-(\rho \varphi(T,b,\sigma,\Delta)^{1/2} / \lambda_{3.2})}
$$

\noindent where the positive constants $\lambda_{3.2}$ and $K_{3.2}$ are defined in Corollary \ref{CONTLAPLACETER}.

\end{itemize}

%\begin{equation*}{}
%\Phi_{\alpha}(\lambda) = \left\{
%\begin{array}{l}
%   K_{3.1}(\varphi(f,T,b,\sigma,\Delta)^{2} \vee\varphi(f,T,b,\sigma,\Delta)^{\frac{\alpha}{2\alpha-1}})\max(\lambda^2,\lambda^{\frac{2\alpha}{2\alpha-1}}), \ \ \mbox{if  }\ \alpha\in (\frac12,1] \vspace*{0.2cm} \\
%   \lambda^2 \varphi(f,T,b,\sigma,\Delta) V^{\frac{1}{2}}(x) + \frac{1}{2} \log\left(\E\left[e^{\lambda^2 \varphi(f,T,b,\sigma,\Delta)\eta C_\sigma T e^{CT}|U_1|^2}\right]\right), \ \ \mbox{if  }\ \alpha = \frac12.   
%\end{array}
%\right.
%\end{equation*}  

%$\Phi_{\alpha}(\lambda)=K_{3.1}(\varphi(f,T,b,\sigma,\Delta)^{2} \vee\varphi(f,T,b,\sigma,\Delta)^{\frac{\alpha}{2\alpha-1}})\max(\lambda^2,\lambda^{\frac{2\alpha}{2\alpha-1}})$, the constant $K_{3.1}$ being defined in Corollary \ref{CONTLAPLACEBIS} and $\varphi(f,T,b,\sigma,\Delta)= [f]^2_1C_{\sigma}\beta\frac{(1+C(\Delta)\Delta)}{4C(\Delta)}e^{3C(\Delta) T}$ and $C(\Delta) := 2[b]_1 + [\sigma]^2_{1} + \Delta [b]^2_1$. 
%
%\noindent if $\alpha=\frac{1}{2}$, for all $0 \leq \lambda < \varphi(f,T,b,\sigma,\Delta)^{-\frac{1}{2}} \min(1, \varepsilon_{\beta}^{\frac{1}{2}} (\eta C_{\sigma} T \exp(C T))^{-\frac{1}{2}})$, one has
%$$
%\E_x\left[\exp(\lambda f(X^{\Delta}_{T}))\right] \leq \exp(\lambda \E_x\left[f(X^{\Delta}_{T})\right])\exp\left(\lambda^2 \varphi(f,T,b,\sigma,\Delta) V^{\frac{1}{2}}(x) + \frac{1}{2} \log\left(\E_{x}\left[e^{\lambda^2 \varphi(f,T,b,\sigma,\Delta)\eta C_\sigma T e^{CT}|U_1|^2}\right]\right)\right),
%$$
%
%\noindent with $C:= \frac{1}{2}(C_V C_{b})^{\frac{1}{2}} + 2 (1+2\eta \Delta)^2+ \frac{1}{2} \eta C_b \Delta$.  
%  
%  
 \end{THM}

Note that in the above theorem, we do not need any non-degeneracy condition on the diffusion coefficient. 

In the case $\alpha \in  (\frac12,1]$, one easily gets the following explicit formula: 

\begin{itemize}
\item If $\lambda \in [0, 2\Psi]$, then $\Phi^{*}_{\alpha}(\lambda) = \frac{1}{4\Psi} \lambda^2$;

\item If $\lambda \in [\frac{2\alpha}{2\alpha - 1}\Psi, +\infty)$, then $\Phi^{*}_{\alpha}(\lambda) = \frac{1}{2\alpha}\left(\frac{2\alpha -1}{2\alpha \Psi}\right)^{2\alpha -1} \lambda^{2\alpha}$;

\item If $\lambda \in (2\Psi, \frac{2\alpha}{2\alpha - 1}\Psi)$,then $\Phi^{*}_{\alpha}(\lambda) = \lambda - \Psi$.
\end{itemize}

Let us note that the linear behavior of $\Phi^{*}_{\alpha}$ on a small interval is due to the fact that $\Phi_\alpha$ is not $\mathcal{C}^1$. One may want to replace $\rho^2 \vee \rho^{\frac{2\alpha}{2\alpha-1}}$ by $\rho^2 + \rho^{\frac{2\alpha}{2\alpha-1}}$ (up to a factor 2) in the expression of $\Phi_\alpha$. However, in this case, an explicit expression for $\Phi^{*}_{\alpha}$ does not exist (except for the case $\alpha=1$) and only its asymptotic behavior  can be derived so that one is led to compute it numerically in practical situations. 
 
In the case $\alpha = 1/2$, tedious but simple computations show that
\begin{align}
\Phi^{*}_{1/2}(\lambda) &= \left( \left(1 + \frac{\lambda_{3.2}}{K_{3.2} \varphi(T,b,\sigma,\Delta)^{1/2}}  \lambda\right)^{\frac{1}{2}} -1\right)^{2}. \notag
\end{align}

%In the case $\alpha = 1/2$, we need more information on the random variable $U$ to get explicit bounds on $\Phi^{*}$. For the \emph{standard Euler scheme} in dimension $1$, $d=q=1$, i.e. $U$ a standard Gaussian random variable, for all $\rho \in [0, \varphi(T,b,\sigma,\Delta)^{-\frac{1}{2}} \min(1, \varepsilon_{\beta}^{\frac{1}{2}} (\eta C_{\sigma} T \exp(C T))^{-\frac{1}{2}}))$ we have 
%\begin{align}
%\Phi_{1/2}(\rho) &= \rho^2 \varphi(T,b,\sigma,\Delta) V^{\frac{1}{2}}(x) + \frac{1}{2} \log \left( \frac{1}{\sqrt{1 - 2\rho^2 \varphi(T,b,\sigma,\Delta)\eta C_\sigma T e^{CT}}}\right). \notag
%\end{align}

%\noindent The Legendre transform of this function can be explicitly computed, but the obtained expression is complicated, and not very clear. However, the following simple asymptotic behaviors can be easily obtained:
%\begin{itemize}
%\item When $\lambda$ is small, $\Phi^{*}_{1/2}(\lambda) \sim \frac{1}{4\varphi(T,b,\sigma,\Delta) V^{\frac{1}{2}}(x) + 2\varphi(T,b,\sigma,\Delta)\eta C_\sigma T e^{CT}}\lambda^2;$
%
%\item When $\lambda$ goes to infinity, $\Phi^{*}_{1/2}(\lambda) \sim \frac{1}{\sqrt{ 2\varphi(T,b,\sigma,\Delta)\eta C_\sigma T e^{CT}}} \lambda$.
%\end{itemize}

\noindent This behavior corresponds to a concentration profile that is Gaussian at short distance, and exponential at large distance.
\begin{COROL}{(Non-asymptotic deviation bounds)}
\label{COROL_DEV_INEQ}
Under the same assumptions as Theorem \ref{THM_TE_EULER}, one has:

\begin{itemize}

\item for all real-valued 1-Lipschitz function $f$ defined on $\R^d$, for all $\alpha \in [1/2,1]$ for all $M\geq 1$ and all $r\ge 0$, \begin{eqnarray*} 
&&\P_x\left(|\frac{1}{M}\bsum{k=1}^M f((X_T^\Delta)^k)-\E_x[f(X_T^\Delta)]|\ge r\right)\le 2\exp(-M\Phi^{*}_{\alpha}(r)),
\end{eqnarray*}

\item for all $\alpha \in [1/2,1]$, for all $M\geq 1$ and all $r\ge 0$, 
$$
 \P_x\left(\sup_{f: [f]_1\leq 1} \left(\frac{1}{M} \sum_{k=1}^M f((X_T^\Delta)^k) - \E_x[f(X_T^\Delta)] \right)\geq  C(d,\mu^{\Delta,0,x}_T) M^{-1/(d+2)}+r\right) \leq \exp\left(-M\Phi^{*}_{\alpha}(r)\right),
$$

\end{itemize}

\noindent where the $((X_T^\Delta)^k)_{1\leq k \leq M}$ are $M$ independent copies of the scheme \eqref{EULER} starting at point $x$ at time $0$ and evaluated at time $T$. 
\end{COROL}

\begin{REM}[Extension to smooth functions of a finite number of time step] The previous transport-inequalities and non-asymptotic bounds could be extended to smooth functions of a finite number of time step such as the maximum of a scalar Euler like scheme. In that case, it suffices to introduce the additional state variable $(M_{t_i}^\Delta)_{i\geq 1}:=(\max_{k\in \leftB 0,i\rightB}X_{t_k}^\Delta)_{i\geq 1} $. Now, the couple $(X_{t_i}^\Delta,M_{t_i}^\Delta)_{1 \leq i \leq N} $ is Markovian and similar arguments could be easily extended to the couple for Lipschitz functions of both variables.
\end{REM}

\begin{REM}[Transport-Entropy inequalities for the law of a diffusion process] The previous transport-inequalities and non-asymptotic bounds could be extended to the law at time $T$ of the diffusion process solution to \eqref{EDS} by passing to the limit $\Delta \rightarrow 0$. Indeed, it is well-known that under \A{HS}, one has $X^{\Delta}_T \stackrel{a.s.}{\longrightarrow} X_T$, as $\Delta \rightarrow 0$ and by Lebesgue theorem, one deduces from the first result of Corollary \ref{COROL_DEV_INEQ} that the empirical error (empirical mean) of $X_T$ itself satisfies a non-asymptotic deviation bound with a similar deviation function (just pass to the limit $\Delta \rightarrow 0$ in all constants). Then, using Corollary 5.1 in \cite{goz:leon:2010} (equivalence between deviation of the empirical mean and transport-entropy inequalities), one easily derives that the law of $X_T$ satisfies a similar transport-entropy inequalities when $\alpha \in (1/2,1]$.
\end{REM}

We want to point out that it is the growth of $\sigma$ that gives the concentration regime ranging from Gaussian concentration bound if $\alpha=1$ to exponential when $\alpha = \frac12$. However, in many popular models in finance, the diffusion coefficient is linear, for instance practitioners often have to deal with Black-Scholes like dynamics of the form 
$$
X_t=x_0+\int_0^t b(X_s)X_sds +\int_0^t\sigma(X_s)X_sdW_s
$$ 

\noindent for smooth, bounded coefficients $b,\sigma$. For the estimation of $\E_x[f(X_T^\Delta)] $ for a Lipschitz function $f:\R^d \rightarrow \R$, or even in more general situations, the estimation of $\E_x[f(X^\Delta)] $ for a Lipschitz function $f:\mathcal{C} \rightarrow \R$, where $\mathcal{C}:= \mathcal{C}([0,T],\R^d)$ stands for the space of $\R^d$-valued continuous functions on $[0,T]$, equipped with the uniform norm $||f||_{\infty}:=\sup_{0\leq t \leq T}|f(t)|$, the expected concentration is the log-normal one. To deal with the latter case, we consider the continuous Euler scheme $X^{c,\Delta}$ associated to \eqref{EDS} and writing
\begin{equation}
\forall t \in [0,T], \ X^{c,\Delta}_{t}=x+ \int_0^t b(\phi(s),X^{c,\Delta}_{\phi(s)}) ds + \int_0^t \sigma(\phi(s),X^{c,\Delta}_{\phi(s)}) dW_s, \ \ \ x\in \R^d.
\label{CONT_EULER}
\end{equation}

\noindent where we set $\phi(t):=t_i$ for $t_i\leq t < t_{i+1}$, $i\in \N$. The next result provides a general non-asymptotic deviation bound for the empirical error under very mild assumptions.
\begin{THM}[General non-asymptotic deviation bounds]
\label{GEN_CONC_EULER}
Denote by $X^{c,\Delta}:=(X^{c,\Delta}_{t})_{0\leq t \leq T}$ the path of the scheme \eqref{CONT_EULER} with step $\Delta$ starting from point $x$ at time $0$. Assume that $\forall t\in [0,T]$, the coefficients $b(t,.)$ and $\sigma(t,.)$ are continuous functions in $x$ and that they satisfy the linear growth assumption:
$$
 \forall x\in \R^d, \ \ \sup_{t \in [0,T]}|b(t,x)| \leq C_b (1+ |x|), \ \ \sup_{t\in [0,T]}Tr(a(t,x)) \leq C_{\sigma}(1+|x|^2).
$$

Then,  for all $1$-Lipschitz function $f:\mathcal{C}\rightarrow \R$, for all $M\in \N^{*}$, for all $r\geq0$, one has
\begin{equation*}{}
\P_x\left(|\frac{1}{M}\bsum{k=1}^M f((X^{c,\Delta})^k)-\E_x[f(X^{c,\Delta})]|\ge r\right) \leq \left\{
\begin{array}{l}
2\exp\left(-\frac{r^2M}{(2 (1+|x|))^{2}\exp(2\kappa(b,\sigma,T))}\right)
, \ \mbox{if } \ r \leq \frac{1}{\sqrt{M}} 2 (1+|x|) e^{\kappa(b,\sigma,T)} \vspace*{0.2cm} \\
  2 \exp\left(-\frac{1}{4 \kappa(b,\sigma,T)}\log\left(\frac{r^2M}{(2(1+|x|))^2}\right)^2\right), \ \mbox{otherwise}
\end{array}
\right.
\end{equation*} 

\noindent where $\kappa(b,\sigma,T):= 28(1+ (C_{\sigma}\vee C_{b}) T)$ and $((X^{c,\Delta})^k)_{1\leq k \leq M}$ are $M$ independent copies  of the scheme \eqref{CONT_EULER}. The result remains valid when one considers the path of the diffusion $X$ solution to \eqref{EDS} instead of  the continuous Euler scheme.

\end{THM}
\subsection{Stochastic approximation algorithms}
 \begin{THM}[Transport-Entropy inequalities for stochastic approximation algorithms]
 \label{THM_TE_SA}
Let $N\in \N^*$. Assume that the function $H$ of the recursive procedure $(\theta_{n})_{0 \leq n \leq N}$ (with starting point $\theta_{0} \in \R^{d}$) defined by \eqref{RM} satisfies \A{HL}, \A{HUA} and \A{HLS}$_\alpha$ for $\alpha \in [\frac12,1]$, and that the step sequence $\gamma=(\gamma_{n})_{n\geq 0}$ satisfies \eqref{STEP}. Suppose that the law of the innovation satisfies \eqref{CONC_GAUSS}, $\beta>0 $. Denote by $\mu^{\gamma,0,\theta_0}_{N}$ the law of $\theta_{N}$.
 
Then, $\mu^{\gamma,0,\theta_0}_{N}$ satisfies $T_{\Phi^{*}_{\alpha}}$ with $\Phi^{*}_{\alpha,N}(\lambda)=\sup_{\rho \geq 0}\left\{\lambda \rho - \Phi_{\alpha,N}(\rho)\right\}$ and one has: 
\begin{itemize}

%\item If $ \alpha\in (\frac12,1]$, for all $\rho\geq0$
%$$
%\Phi_{\alpha,N}(\rho) = \varphi(\gamma,\alpha,H)C^{\gamma}_{N}(\rho^2\vee\rho^{\frac{2\alpha}{2\alpha-1}})
%$$
%
%\noindent where $\varphi(\gamma,\alpha,H)= K_{5.1}(1+\gamma^2_1)^{\frac{1-\alpha}{2\alpha-1}}   \frac{\beta C^2_\alpha}{4}\vee (\frac{\beta C^2_\alpha}{4})^{\frac{\alpha}{2\alpha-1}}\exp\left(\frac{C^{\gamma}_{N}}{2\alpha-1} \right)$, the constant $K_{5.1}$ being defined in Corollary \ref{CONTLAPLACELBIS}.
\item If $ \alpha\in (\frac12,1]$, for all $\rho\geq0$
$$
\Phi_{\alpha,N}(\rho) = \varphi_{\alpha}(\gamma,H, \theta_0) (C^{\gamma}_N\rho^2 \vee C^{\gamma, \alpha}_N\rho^{\frac{2\alpha}{2\alpha-1}})  
$$

\noindent with the two concentration rates  $C^{\gamma}_{N}:=\sum_{k=0}^{N-1}\gamma^2_{k+1}\frac{\Pi_{1,N}}{\Pi_{1,k}}$, with $\Pi_{1,N} := \prod_{k=0}^{N-1}(1-2\underline{\lambda}\gamma_{k+1} + C_{H,\mu} \gamma^2_{k+1})$ and $C^{\gamma, \alpha}_N := \sum_{k=0}^{N-1} \gamma^{\frac{2\alpha}{2\alpha-1}}_{k+1} (\frac{\Pi_{1,N}}{\Pi_{1,k}})^{\frac{2\alpha}{2\alpha-1}} ((k+1)\log^2(k+4))^{\frac{1-\alpha}{2\alpha-1}}$ for all $N\geq1$, where $C_{H,\mu}:=2 C_{H}^2(1+\E[|U|^2])$ and $\varphi_{\alpha}(\gamma, H, \theta_0)$ is an explicit constant defined in Proposition \ref{CONT_LAPLACERM}.

%\item If $\alpha=\frac12$, for all $\rho \in[0, 2C^{-1}_{1/2}\exp(-C^{\gamma}_{N})(\beta (1+\gamma^2_1))^{-\frac12}\min(1, \varepsilon^{\frac12}_{\beta} (2 C_{1/2}^2\prod_{k=0}^{N-1}(1+C_{1/2}^2 \gamma_{k+1}^2))^{-\frac12}))$ 
%$$
%\Phi_{1/2,N}(\rho) = C^{\gamma}_{N} \left( (L^{\frac12}(\theta_0) + \underline{C} \sum_{p=0}^{N-1}\gamma^2_{p+1})\tilde{\Pi}_{2,N} \frac{\beta C^2_{1/2}}{4} \rho^2 + \left(\frac{1}{2} \sum_{p=0}^{N-1} \gamma^2_{p+1}\right) \log\E\left[e^{\frac{\beta C_{1/2}^4 }{2}(1+\gamma^2_{1})\tilde{\Pi}_{2,N} \rho^2 |U_1|^2}\right] \right),
%$$
%
%\noindent and $\Phi_{1/2}(\rho) = + \infty$ otherwise and  $\tilde{\Pi}_{2,N}:=\prod_{k=0}^{N-1}(1+C_{1/2}^2 \gamma_{k+1}^2) \exp(2C^{\gamma}_{N})$. 
\item If $\alpha=\frac12$, for all $\rho \in [0, \lambda_{4.1}/ \tilde{s}_{N})$,
$$
\Phi_{1/2,N}(\rho) = 2 \varphi_{1/2}(\gamma, H, \theta_0) C^{\gamma}_N   \frac{(\rho / \lambda_{4.1})^2}{1-(\rho \tilde{s}_N / \lambda_{4.1})}
$$

\noindent with $ \tilde{s}_N := \max_{0 \leq k \leq N-1} (k+1)^{1/2} \log(k+4) \gamma_{k+1} \left(\frac{\Pi_{1,N}}{\Pi_{1,k}}\right)^{\frac12} \exp(\sum_{p=0}^{N-1} \frac{1}{(p+1) \log^2(p+4)})$
%$S^{\gamma}_N := \sum_{k=0}^{N-1} (k+1) \log^2(k+4) \gamma^4_{k+1} \frac{\Pi^2_{1,N}}{\Pi^2_{1,k}}$ 
and the (positive) constants $\varphi_{1/2}(\gamma, H, \theta_0)$ and $\lambda_{4.1}$ are defined in Proposition \ref{CONT_LAPLACERM}.
\end{itemize}
 
\end{THM}

As in the case of Euler like schemes, for $\alpha \in (\frac12,1]$, we have:
\begin{itemize} 

\item if $\lambda \in [0, 2\varphi (C^{\gamma}_{N} / (C^{\gamma, \alpha}_{N})^{2\alpha-1})^{\frac{1}{2(1-\alpha)}}]$, then $\Phi^{*}_{\alpha,N}(\lambda) =  \lambda^2/(4\varphi C^{\gamma}_{N})$;

\item If $\lambda \in [\frac{2\alpha}{2\alpha-1} \varphi  (C^{\gamma}_{N} / (C^{\gamma, \alpha}_{N})^{2\alpha-1})^{\frac{1}{2(1-\alpha)}}, +\infty)$, then $\Phi^{*}_{\alpha,N}(\lambda) = \frac{1}{2\alpha} \left(\frac{2\alpha-1}{2\alpha \varphi }\right)^{2\alpha-1} (\lambda^{2\alpha}/(C^{\gamma, \alpha }_{N})^{2\alpha-1})$;

\item If $\lambda \in (2\varphi (C^{\gamma}_{N} / (C^{\gamma, \alpha}_{N})^{2\alpha-1})^{\frac{1}{2(1-\alpha)}}, \frac{2\alpha}{2\alpha-1} \varphi  (C^{\gamma}_{N} / (C^{\gamma, \alpha}_{N})^{2\alpha-1})^{\frac{1}{2(1-\alpha)}})$, then $\Phi^{*}_{\alpha,N}(\lambda)= (\frac{C^{\gamma}_N}{C^{\gamma, \alpha}_{N}})^{\frac{2\alpha-1}{2(1-\alpha)}} \lambda - \varphi \frac{(C^{\gamma}_{N})^{\frac{\alpha}{1-\alpha}}}{(C^{\gamma,\alpha}_N)^{\frac{2\alpha-1}{1-\alpha}}}$.

\end{itemize}
%As in the case of Euler like schemes, in the case $\alpha \in (\frac12,1]$, we have:
%\begin{itemize} 
%
%\item if $\lambda \in [0, 2\varphi C^{\gamma}_{N}]$, then $\Phi^{*}_{\alpha,N}(\lambda) = \frac{1}{4\varphi C^{\gamma}_{N}} \lambda^2$;
%
%\item If $\lambda \in [\frac{2\alpha}{2\alpha-1}\varphi C^{\gamma}_{N}, +\infty)$, then $\Phi^{*}_{\alpha,N}(\lambda) = \frac{1}{2\alpha} \left(\frac{2\alpha-1}{2\alpha \varphi C^{\gamma}_{N} }\right)^{2\alpha-1} \lambda^{2\alpha}$;
%
%\item If $\lambda \in (2\varphi C^{\gamma}_{N}, \frac{2\alpha}{2\alpha-1}\varphi C^{\gamma}_{N})$, then $\Phi^{*}_{\alpha,N}(\lambda)= \lambda - \varphi C^{\gamma}_{N}$ .
%
%\end{itemize}

For $\alpha=\frac12$, we obtain the following explicit bound for the Legendre transform of $\Phi_{1/2,N}$ 
$$
\forall  \lambda \geq0, \ \ \Phi^{*}_{1/2,N}(\lambda)= \frac{2 \varphi C^{\gamma}_N }{\tilde{s}^2_N} \left( \left(1+ \frac{\tilde{s}_N \lambda_{4.1} \lambda }{2\varphi C^{\gamma}_N }\right)^{\frac12} - 1\right)^2
$$
 
 Hence, for $N \geq1$ being fixed, the following simple asymptotic behaviors can be easily derived:
\begin{itemize}
\item When $\lambda$ is small, $\Phi^{*}_{1/2,N}(\lambda) \sim \lambda^2_{4.1} \lambda^2 /(2\varphi C^{\gamma}_{N});$
\vspace*{.2cm}
\item When $\lambda$ goes to infinity, $\Phi^{*}_{1/2}(\lambda) \sim \lambda_{4.1}\lambda / \tilde{s}_N$.
\end{itemize}

\begin{COROL}(Non-asymptotic deviation bounds)
\label{COROL_SA}
 Under the same assumptions as Theorem \ref{THM_TE_SA}, one has
$$
\P_{\theta_0}\left(\left|\theta_{N} - \theta^{*}\right| \geq r + \delta_{N}\right) \leq \exp\left(-\Phi^{*}_{\alpha,N}(r)\right)
$$

\noindent and $\delta_{N}:=\E_{\theta_0}\left[\left|\theta_{N}-\theta^{*}\right|\right]$. Moreover, the bias $\delta_{N}$ at step $N$ satisfies
$$
\delta_{N} \leq e^{-\underline{\lambda} \Gamma_{1,N} + C_{\alpha,\mu}\Gamma_{2,N}} \left|\theta_0 - \theta^{*}\right| + (2C_{\alpha,\mu})^{\frac{1}{2}} \left(\sum_{k=0}^{N-1} \gamma^2_{k+1} e^{-2\underline{\lambda} (\Gamma_{1,N}-\Gamma_{1,k+1}) + 2C_{\alpha,\mu} (\Gamma_{2,N}-\Gamma_{2,k+1})}\right)^{\frac{1}{2}},
$$

\noindent where $\Gamma_{1,N} :=\sum_{k=1}^{N} \gamma_{k}$, $\Gamma_{2,N} :=\sum_{k=1}^{N} \gamma^2_{k}$, $C_{\alpha,\mu}:=\underline{\lambda}^2 /2+2 C_{\alpha}K\E[|U|^2]$ with $K>0$.
\end{COROL}

Now, we investigate the impact of the step sequence $(\gamma_{n})_{n\geq1}$ on the concentration rate sequences $C^{\gamma}_N$, $C^{\gamma,\alpha}_N$, $\tilde{s}_N$ and the bias $\delta_N$. Let us note that a similar analysis has been performed in  \cite{Frikha2012}. We obtain the following results:
\begin{itemize} 
\item If we choose $\gamma_{n}=\frac{c}{n}$, with $c>0$. Then $\delta_{N} \rightarrow 0$, $N\rightarrow+\infty$, $\Gamma_{1,N}=c \log(N) + c'_{1} + r_{N}$, $c'_{1}>0$ and $r_{N}\rightarrow 0$, so that $\Pi_{1,N} =\O( N^{-2c \underline{\lambda}})$.
\begin{itemize}

\item If $c < \frac{1}{2 \underline{\lambda}}$, the series $\sum_{k=1}^{N}\gamma^{2}_{k}/\Pi_{1,k}$, $\sum_{k=0}^{N-1} \gamma^{\frac{2\alpha}{2\alpha-1}}_{k+1} (1/\Pi_{1,k}^{\frac{2\alpha}{2\alpha-1}}) ((k+1)\log^2(k+4))^{\frac{1-\alpha}{2\alpha-1}}$ converge so that we obtain $C^{\gamma}_N= \mathcal{O}(N^{-2c \underline{\lambda}})$, $C^{\gamma, \alpha}_N= \mathcal{O}(N^{- \frac{2\alpha}{2\alpha -1}c \underline{\lambda}})$, $\tilde{s}_N=\O(N^{-c \underline{\lambda}})$. 
%$S^{\gamma}_N = \mathcal{O}(N^{-4c \underline{\lambda}})$.
 
 \vspace*{.2cm}
 
\item If $c > \frac{1}{2 \underline{\lambda}}$, a comparison between the series and the integral yields $C^{\gamma}_N=\O(N^{-1})$, $C^{\gamma, \alpha}_N = \mathcal{O}((\log(N))^{2 \frac{1-\alpha}{2\alpha-1}} N^{- \frac{\alpha}{2\alpha -1}})$, $\tilde{s}_N=\mathcal{O}(\log(N)N^{-\frac12})$.
%$S^{\gamma}_N = \mathcal{O}(\log^2(N)N^{-2})$ .
\end{itemize}

\medskip

\noindent Let us notice that we find the same critical level for the constant $c$ as in the Central Limit Theorem for stochastic algorithms. Indeed, if $c> \frac{1}{2\mathcal{R}e(\lambda_{min})}$ where $\lambda_{min}$ denotes the eigenvalue of $Dh(\theta^{*})$ with the smallest real part then we know that a Central Limit Theorem holds for $(\theta_{n})_{n\geq1}$ (see e.g. \cite{Duflo1996}, p.169). Such behavior was already observed in \cite{Frikha2012}. 

The associated bound for the bias is the following:
$$
\delta_N\le K\left(\frac {|\theta_0-\theta^*|} {N^{\underline{\lambda} c}}+  \frac{(2C_{\alpha,\mu})^{\frac{1}{2}} }{N^{\underline{\lambda}c\wedge \frac12}}  \right).
$$

\item If we choose $\gamma_{n}=\frac{c}{n^{\rho}}$, $c>0$, $\frac{1}{2} < \rho < 1$, then $\delta_{N} \rightarrow 0$, $\Gamma_{1,N} \sim \frac{c}{1-\rho} N^{1-\rho}$ as $N \rightarrow + \infty$ and elementary computations show that there exists $C>0$ s.t. for all $N\ge 1$, $\Pi_{1,N}\le C\exp(-2\underline{\lambda}\frac{c}{1-\rho}N^{1-\rho}) $. Hence, for all $\epsilon\in (0,1-\rho) $ we have:
\begin{eqnarray*}
C^{\gamma}_N=\Pi_{1,N}\sum_{k=1}^{N} \gamma_k^2 \Pi^{-1}_{1,k} &\le& c^2 \left\{\Pi_{1,N}\Pi_{1,N-N^{\rho+\epsilon}}^{-1}\sum_{k=1}^{N-N^{\rho+\epsilon}} \frac1{k^{2\rho}}+\sum_{k=N-N^{\rho+\epsilon}+1}^{N} \frac1{k^{2\rho}}\right\}\\&\le &  c^2 \left\{C\exp(-2\underline{\lambda}\frac{c}{1-\rho}(N^{1-\rho}-(N-N^{\rho+\epsilon} )^{1-\rho})) +  \frac{N^{\rho+\epsilon}}{(N-N^{\rho+\epsilon}+1)^{2\rho}}\right\}\\
&\le & c^2 \left\{C\exp(-2\underline{\lambda}c N^{\epsilon})+  \frac1{N^{\rho-\epsilon}} \right\}.
\end{eqnarray*}

\noindent Up to a modification of $\epsilon$, this yields $C^{\gamma}_N=\Pi_{1,N}\sum_{k=1}^{N} \gamma_k^2 \Pi^{-1}_{1,k}=o(N^{-\rho+\epsilon}),\ \epsilon\in (0,1-\rho)$. Similar computations show that $C^{\gamma, \alpha}_N = o(N^{-\frac{(\rho-(1-\alpha))}{2\alpha-1}-\epsilon})$ and we clearly get $\tilde{s}_N = \mathcal{O}\left( \log(N) N^{-(\rho-\frac12)}\right)$.

\end{itemize}

Concerning the bias, from Corollary \ref{COROL_SA}, we directly obtain the following bound:
$$
\delta_N\le K\left( \exp\left( -\frac{\underline{\lambda}c }{1-\rho}  N^{1-\rho} \right)|\theta_0-\theta^*|+ \frac{(2C_{\alpha,\mu})^{\frac{1}{2}}}{N^{\frac{\rho}{2}-\epsilon}}\right),\ \forall \epsilon>0. 
$$

The impact of the initial difference $|\theta_0-\theta^*| $ is exponentially smaller compared to the case $\gamma_{n}=\frac{c}{n}$. This is natural since the step sequence is decreasing slower to $0$.

\begin{THM}[Transport-Entropy inequalities for stochastic approximation with averaging of trajectories]
\label{THM_TE_SA_AV}
Let $N\in \N^{*}$. Assume that the function $H$ of the recursive procedure $\theta=(\theta_{n})_{0 \leq n \leq N}$ (with starting point $\theta_{0} \in \R^{d}$) defined by \eqref{RM} satisfies \A{HL}, \A{HUA} and \A{HLS}$_\alpha$ for $\alpha \in [\frac12,1]$, and that the step sequence $\gamma = (\gamma_{n})_{n\geq1}$ satisfies \eqref{STEP}. Suppose that the law of the innovation satisfies \eqref{CONC_GAUSS}, $\beta>0 $. Denote by $\bar{\mu}^{\gamma,0,\theta_0}_{N}$ the law of $\bar{\theta}_{N}$ where $\bar{\theta}$ is the empirical mean of $\theta$ defined by \eqref{RM_AV}.  
Then, $\bar{\mu}^{\gamma,0,\theta_0}_{N}$ satisfies $T_{\bar{\Phi}^{*}_{\alpha,N}}$ with $\bar{\Phi}^{*}_{\alpha,N}(\lambda)=\sup_{\rho \geq 0}\left\{\lambda \rho - \bar{\Phi}_{\alpha,N}(\rho)\right\}$ and one has: 
\begin{itemize}

\item If $ \alpha\in (\frac12,1]$, for all $\rho\geq0$
$$
\bar{\Phi}_{\alpha,N}(\rho) = \varphi_{\alpha}(\gamma,H, \theta_0)  (\bar{C}^{\gamma}_N  \rho^2 \vee \bar{C}^{\gamma,\alpha}_N \rho^{\frac{2\alpha}{2\alpha-1}})
$$

\noindent where $\varphi_{\alpha}(\gamma, H, \theta_0)$ is a positive constant defined in Section \ref{SEC:PROOF:SAAV}.

\item If $\alpha=\frac12$, for all $\rho \in [0, \lambda_{4.1}/ \hat{s}_{N})$,
$$
\bar{\Phi}_{1/2,N}(\rho) = 2 \varphi_{1/2}(\gamma, H, \theta_0)  \bar{C}^{\gamma}_{N}   \frac{(\rho / \lambda_{4.1})^2}{1-(\rho \hat{s}_N / \lambda_{4.1})}
$$

\noindent where $\varphi_{1/2}(\gamma, H, \theta_0)$ and $\lambda_{4.1}$ are positive constants defined in Proposition \ref{CONT_LAPLACERM}.

%\noindent and $\bar{\Phi}_{1/2,N}(\rho) = + \infty$ otherwise and  $\tilde{\Pi}_{2,N}:=\prod_{k=0}^{N-1}(1+C_{1/2}^2 \gamma_{k+1}^2) \exp(2\bar{C}^{\gamma}_{N})$. 
\end{itemize}
 
 \noindent where the three concentration rate sequences are defined for $N\in \N^{*}$ by 
$$
 \bar{C}^{\gamma}_{N}  := \sum_{k=1}^{N-1} \bar{\gamma}^2_{k,N}, \ \ \ \ 
 \bar{C}^{\gamma, \alpha}_N  := \sum_{k=1}^{N-1}\bar{\gamma}^{\frac{2\alpha}{2\alpha-1}}_{k,N} ((k+1) \log^2(k+4))^{\frac{1-\alpha}{2\alpha-1}}, \ \ \hat{s}_N := \max_{1 \leq k \leq N-1} (k+1)^{\frac12} \log(k+4)\bar{\gamma}_{k,N} e^{\sum_{p=0}^{N-1} \frac{1}{(p+1) \log^2(p+4)}} 
 $$
 
 \noindent with $\bar{\gamma}_{k,N}:= \frac{\gamma_k}{N}(1+\sum_{j=k+1}^{N-1} (\frac{\Pi_{1,j}}{\Pi_{1,k}})^{\frac12})$, and $\Pi_{1,N} := \prod_{p=0}^{N-1}(1-2\underline{\lambda}\gamma_{p+1} + C_{H,\mu} \gamma^2_{p+1})$.
   
\end{THM}
As regards the explicit computation of the Legendre transform of $\bar{\Phi}_{\alpha,N}$, similarly to the previous theorem, we have:

\begin{itemize}
\item  for $\alpha \in (\frac12,1]$:
\begin{itemize} 

\item if $\lambda \in [0, 2\varphi (\bar{C}^{\gamma}_{N} / (\bar{C}^{\gamma, \alpha}_{N})^{2\alpha-1})^{\frac{1}{2(1-\alpha)}}]$, then $\bar{\Phi}^{*}_{\alpha,N}(\lambda) = (\lambda^2/4\varphi \bar{C}^{\gamma}_{N}) $;

\item If $\lambda \in [\frac{2\alpha}{2\alpha-1} \varphi  (\bar{C}^{\gamma}_{N} / (\bar{C}^{\gamma, \alpha}_{N})^{2\alpha-1})^{\frac{1}{2(1-\alpha)}}, +\infty)$, then $\bar{\Phi}^{*}_{\alpha,N}(\lambda) = \frac{1}{2\alpha} \left(\frac{2\alpha-1}{2\alpha \varphi }\right)^{2\alpha-1} (\lambda^{2\alpha}/ (\bar{C}^{\gamma, \alpha }_{N})^{2\alpha-1})$;

\item If $\lambda \in (2\varphi (\bar{C}^{\gamma}_{N} / (\bar{C}^{\gamma, \alpha}_{N})^{2\alpha-1})^{\frac{1}{2(1-\alpha)}}, \frac{2\alpha}{2\alpha-1} \varphi  (\bar{C}^{\gamma}_{N} / (\bar{C}^{\gamma, \alpha}_{N})^{2\alpha-1})^{\frac{1}{2(1-\alpha)}})$, then $\bar{\Phi}^{*}_{\alpha,N}(\lambda)= (\frac{\bar{C}^{\gamma}_N}{\bar{C}^{\gamma, \alpha}_{N}})^{\frac{2\alpha-1}{2(1-\alpha)}} \lambda - \varphi \frac{(\bar{C}^{\gamma}_{N})^{\frac{\alpha}{1-\alpha}}}{(\bar{C}^{\gamma,\alpha}_N)^{\frac{2\alpha-1}{1-\alpha}}}$.
\end{itemize}

\item for $\alpha=\frac12$,
$$
\forall  \lambda \geq0, \ \ \bar{\Phi}^{*}_{1/2,N}(\lambda)= \frac{2 \varphi \bar{C}^{\gamma}_N }{\hat{s}^2_N} \left( \left(1+ \frac{\hat{s}_N \lambda_{4.1} \lambda }{2\varphi \bar{C}^{\gamma}_N }\right)^{\frac12} - 1\right)^2
$$
 
 Hence, for $N \geq1$ being fixed, the following simple asymptotic behaviors can be easily derived:
\begin{itemize}
\item When $\lambda$ is small, $\bar{\Phi}^{*}_{1/2,N}(\lambda) \sim \lambda^2_{4.1} \lambda^2 /(2\varphi \bar{C}^{\gamma}_{N});$
\vspace*{.2cm}
\item When $\lambda$ goes to infinity, $\bar{\Phi}^{*}_{1/2}(\lambda) \sim \lambda_{4.1}\lambda / \hat{s}_N$.
\end{itemize}

\end{itemize}
 
\begin{COROL}(Non-asymptotic deviation bounds)
\label{COROL_SA_AV}
Under the same assumptions as Theorem \ref{THM_TE_SA_AV}, for all $N\geq 1$ for all $r\geq0$, one has
$$
\P_{\theta_0}\left(\left|\bar{\theta}_{N} - \theta^{*}\right| \geq r + \bar{\delta}_{N}\right) \leq \exp\left(-\Phi^{*}_{\alpha,N}(r)\right)
$$

\noindent and $\bar{\delta}_{N}:=\E_{\theta_0}\left[\left|\bar{\theta}_{N}-\theta^{*}\right|\right]$. 
\end{COROL}
Now, we analyze the impact of the step sequence on the concentration rate sequences $\bar{C}^{\gamma}_{N}$, $\bar{C}^{\gamma, \alpha}_{N}$, $\hat{s}_N$ and the bias $\bar{\delta}_N$. We first simplify the expression of the concentration rate.
Let us note that since the step sequence $(\gamma_{n})_{n\geq1}$ satisfies \eqref{STEP}, there exists a positive constant $K>0$ such that $(\Pi_{1,j} \Pi^{-1}_{1,k})^{\frac12} \leq K \exp(-\underline{\lambda} (\Gamma_{1,j}-\Gamma_{1,k+1}))$, $k < j$. Moreover, since the function $x\mapsto \exp(-\underline{\lambda} x)$ is decreasing on $[\Gamma_{1,p}, \Gamma_{1,p+1}]$, one clearly gets for all $i,j \in \left\{0, \cdots, N-1\right\}$, $i< j$
$$
M_j-M_i:=\sum_{p=i}^{j-1} \exp(-\underline{\lambda} \Gamma_{1,p+1}) \gamma_{p+1} =\sum_{p=i}^{j-1} \int_{\Gamma_{1,p}}^{\Gamma_{1,p+1}}\exp(-\underline{\lambda} \Gamma_{1,p+1}) dx \leq \frac{1}{\underline{\lambda}} (\exp(-\underline{\lambda} \Gamma_{1,i})-\exp(-\underline{\lambda} \Gamma_{1,j}))
$$
 
\noindent so that, using the latter bound and an Abel transform, we obtain
\begin{align*}
\sum_{j=k+1}^{N-1} \exp(-\underline{\lambda} \Gamma_{1,j+1}) &=  \sum_{j=k+1}^{N-1} (M_{j+1}-M_j) \gamma^{-1}_{j+1} \leq -\frac{1}{\underline{\lambda}} \left(\sum_{j=k+1}^{N-1}  (\exp(-\underline{\lambda} \Gamma_{1,j+1})-\exp(-\underline{\lambda} \Gamma_{1,j}))\gamma^{-1}_{j+1}\right)  \\
& \leq -\frac{1}{\underline{\lambda}} \left(e^{-\underline{\lambda} \Gamma_{1,N}} \gamma^{-1}_{N+1} - e^{-\underline{\lambda} \Gamma_{1,k+1}}  \gamma^{-1}_{k+2} - \sum_{p=k+1}^{N-1} e^{-\underline{\lambda} \Gamma_{1,p+1}}(\gamma^{-1}_{p+2} - \gamma^{-1}_{p+1}) \right) 
\end{align*}

\noindent which finally leads to the following bound
\begin{equation}
\label{bound_step}
\bar{\gamma}_{k,N} \leq \frac{K}{\underline{\lambda}}\left(\frac{\gamma_k \gamma^{-1}_{k+2}}{N} + \frac{\gamma_{k}}{N} \sum_{p=k+1}^{N-1} e^{-\underline{\lambda} (\Gamma_{1,p} -\Gamma_{1,k+1})}(\gamma^{-1}_{p+2} - \gamma^{-1}_{p+1})\right).
\end{equation}

Now, we are in position to study the impact of the step sequence $(\gamma_n)_{n\geq1}$ on the concentration rate sequences:
\begin{itemize}

\item If we select $\gamma_{n} = \frac{c}{n}$ with $c>0$, then, using that $\Gamma_{1,N}=c \log(N) + c'_{1} + r_{N}$, $c'_{1}>0$ with $r_{N}\rightarrow 0$, one easily derives from \eqref{bound_step} that there exists $C>0$ such that
$$
\bar{\gamma}_{k,N} \leq C \left( \frac{1}{N} + \frac{1}{k^{1-c\underline{\lambda}}} \frac1N \sum_{p=k}^{N-1} \frac{1}{p^{\underline{\lambda}c}}\right),
$$

\noindent and a comparison between the series and the integral yields the following bounds:
\begin{itemize}

\item If $\underline{\lambda}c < \frac12$, one has: $\bar{C}^{\gamma}_N = \mathcal{O}(N^{-2c\underline{\lambda}})$, $\bar{C}^{\gamma, \alpha}_N = \O(N^{-\frac{2\alpha}{2\alpha-1} c\underline{\lambda} })$ and $\hat{s}_{N}=\O(N^{-c \underline{\lambda}})$. 

\item If $\underline{\lambda}c >\frac12$, one has: $\bar{C}^{\gamma}_N = \mathcal{O}(N^{-1})$,  $\bar{C}^{\gamma, \alpha}_N = \mathcal{O}((\log(N))^{2 \frac{1-\alpha}{2\alpha-1}} N^{- \frac{\alpha}{2\alpha -1}})$ and $\hat{s}_{N}=\O(N^{-\frac12})$. 

\end{itemize}

Hence, we clearly see that for the case $\gamma_{n} = \frac{c}{n}$, averaging the trajectories of a stochastic approximation algorithm is not the key to circumvent the lake of robustness concerning the choice of the constant $c$.

The bound for the bias is obtained by averaging the bound previously obtained for $\delta_N$. We easily get:
$$
\bar{\delta}_{N} \leq \frac{1}{N} \sum_{k=0}^{N-1} \E_{\theta_0}[|\theta_k-\theta^{*}|] \leq  K \left( \frac{|\theta_0-\theta^{*}|}{N^{\underline{\lambda}c}} + \frac{(2C_{\alpha, \mu})^{\frac12}}{N^{\underline{\lambda} c \wedge \frac12}}\right)
$$

\item If we choose $\gamma_{n}=\frac{c}{n^{\rho}}$, $c>0$, $\frac12 < \rho < 1$ then we have for $k \leq p$
$$
\Gamma_{1,p} -\Gamma_{1,k} = \sum_{j=k+1}^{p} j^{-\rho} = \sum_{j=k+1}^{p} \int_{j}^{j+1} \frac{1}{j^\rho} dx \geq \int_{k+1}^{p+1} \frac{1}{x^{\rho}} dx \geq \frac{1}{1- \rho} \left( (p+1)^{1-\rho} - (k+1)^{1-\rho} \right)
$$

\noindent so that for some positive constant $C$ which may vary from line to line
\begin{align*}
\sum_{p=k+1}^{N-1} e^{-\underline{\lambda} (\Gamma_{1,p} -\Gamma_{1,k+1})}(\gamma^{-1}_{p+2} - \gamma^{-1}_{p+1}) & \leq C e^{\frac{\underline{\lambda}}{1-\rho} (k+1)^{1-\rho}}\left(\sum_{p=k+1}^{N-1} e^{-\frac{\underline{\lambda}}{1-\rho} (p+1)^{1-\rho}} \frac{1}{(p+1)^{1-\rho}}\right) \\
&  \leq C e^{\frac{\underline{\lambda}}{1-\rho} (k+1)^{1-\rho}}\int_{k+1}^{N} e^{-\frac{\underline{\lambda}}{1-\rho} x^{1-\rho}} x^{-(1-\rho)} dx \\
&  \leq C e^{\frac{\underline{\lambda}}{1-\rho} (k+1)^{1-\rho}} \int_{(k+1)^{1-\rho}}^{N^{1-\rho}} e^{-\frac{\underline{\lambda}}{1-\rho} x} x^{\frac{2 \rho-1}{1-\rho}} dx
\end{align*}

\noindent where we use a change of variable in the latter integral. For $k$ large enough, the function $x \mapsto e^{-\frac{\underline{\lambda}}{1-\rho} x}x^{\frac{2\rho}{1-\rho}}$ is decreasing on $[k, +\infty)$ which implies 
$$
e^{\frac{\underline{\lambda}}{1-\rho} (k+1)^{1-\rho}}  \int_{(k+1)^{1-\rho}}^{(N-1)^{1-\rho}} e^{-\frac{\underline{\lambda}}{1-\rho} x} x^{\frac{2 \rho}{1-\rho}} \frac{1}{x^{\frac{1}{1-\rho}}} dx \leq C (k+1)^{2\rho} \left[-\frac{1-\rho}{\rho} x^{-\frac{\rho}{1-\rho}} \right]^{+\infty}_{(k+1)^{1-\rho}} \leq C (k+1)^{\rho}.
$$

Hence, we finally have $ \bar{\gamma}_{k,N} = \mathcal{O}(N^{-1})$ so that $\bar{C}^{\gamma}_N = \mathcal{O}(N^{-1})$, $\bar{C}^{\gamma, \alpha}_N = \mathcal{O}((\log(N))^{2 \frac{1-\alpha}{2\alpha-1}} N^{- \frac{\alpha}{2\alpha -1}})$ and $\hat{s}_{N}=\O(\log(N) N^{-\frac12})$. Hence, averaging has allowed the concentration rate to go from the slow concentration rates $o(N^{-\rho+\epsilon})$,  $o(N^{-\frac{\rho-(1-\alpha)}{2\alpha-1}-\epsilon})$ for all $\epsilon >0$ and $\mathcal{O}\left( \log(N) N^{-(\rho-\frac12)}\right)$ to the optimal rates $\mathcal{O}(N^{-1})$, $ \mathcal{O}((\log(N))^{2 \frac{1-\alpha}{2\alpha-1}} N^{- \frac{\alpha}{2\alpha -1}})$ and $\hat{s}_{N}=\O(\log(N) N^{-\frac12})$ for free, i.e. without any condition on the step sequence parameter $c$. 

Concerning the bias, by averaging the bias sequence $(\delta_k)_{1 \leq k \leq N-1}$ we directly obtain the following bound 
$$
\bar{\delta}_{N} \leq K\left( \frac{|\theta_0-\theta^{*}|}{N}+ \frac{(2C_{\alpha, \mu})^{\frac{1}{2}}}{N^{\frac{\rho}{2}-\epsilon}}\right), \ \forall \epsilon>0
$$

Hence, we see that there is no sub-exponential decreasing of the impact of the initial condition but a decay at rate $\mathcal{O}(N^{-1})$. Consequently, this leads us to say that a stochastic approximation algorithm must be averaged after few iterations in practical implementations and not directly from the first step.
\end{itemize}

\mysection{Euler Scheme: Proof of the Main Results}
\label{SEC_EUL_PR}

In this section we will assume that \A{HS} and \A{HD$_{\alpha}$} are in force.

\subsection{Proof of Theorem \ref{THM_TE_EULER} }

The proof of Theorem \ref{THM_TE_EULER} is divided into several propositions.

\begin{PROP}
\label{CONT_LAPLACEV}
 Denote by $X^{\Delta,0,x}:=(X^{\Delta,0,x}_{t_{k}})_{0\leq k \leq N}$ the scheme \eqref{EULER} with time step $\Delta = T/N$, $N\in \N^{*}$ associated to the diffusion \eqref{EDS} starting from $x$ at time $0$. Assume that the innovations $(U_i)_{i\geq1}$ of \eqref{EULER} satisfy \eqref{CONC_GAUSS} for some $\beta>0$. Then, there exists $\varepsilon_{\beta} >0$ which only depends on the law $\mu$ such that for all $\lambda < \min(1, \varepsilon_{\beta} (2\eta \alpha C_{\sigma} T \exp(C T))^{-1})$, one has
$$
\sup_{0 \leq n \leq N}\log\left(\E_x\left[\exp(\lambda V^{\alpha}(X^{\Delta,0,x}_{t_{n}}))\right]\right) \leq  \lambda \exp(C T) V^{\alpha}(x) + \frac{1}{2} \log\left(\E\left[\exp\left(\lambda 2\eta \alpha C_{\sigma} T \exp(CT) |U_{1}|^2\right)\right]\right). 
$$

\noindent with $C:=C(b,\sigma, V,\alpha, \Delta) = \alpha (C_V C_{b})^{\frac{1}{2}} +  \beta C_{\sigma} \alpha^2 (1+2\eta \Delta)^2 (C_V + C_b) + \alpha \eta C_b \Delta$.

\end{PROP}

\begin{proof} Using the concavity of $x\mapsto x^{\alpha}$, $\alpha \in (0,1]$, we have for all $k\geq 0$ 
$$
V^{\alpha}(X^{\Delta}_{t_{k+1}}) - V^{\alpha}(X^{\Delta}_{t_{k}})) \leq \alpha V^{\alpha-1}(X^{\Delta}_{t_{k}}) (V(X^{\Delta}_{t_{k+1}})-V(X^{\Delta}_{t_{k}})).
$$

\noindent A Taylor expansion of order 2 of the function $V$, recalling that $2 \eta= \sup_{x \in \mathbb{R}^{d}} \left\|\nabla^{2} V(x)\right\| <+\infty$, yields
$$
V(X^{\Delta}_{t_{k+1}}) - V(X^{\Delta}_{t_{k}})) \leq \nabla V(X^{\Delta}_{t_{k}}).(X^{\Delta}_{t_{k+1}}-X^{\Delta}_{t_{k}}) + \eta |X^{\Delta}_{t_{k+1}}-X^{\Delta}_{t_{k}}|^2,
$$

\noindent which together with the previous inequality leads to
\begin{align*}
V^{\alpha}(X^{\Delta}_{t_{k+1}}) - V^{\alpha}(X^{\Delta}_{t_{k}}) & \leq \alpha \Delta \frac{\nabla V(X^{\Delta}_{t_{k}}).b(t_{k},X^{\Delta}_{t_{k}})}{V^{1-\alpha}(X^{\Delta}_{t_{k}})} + \alpha \Delta^{\frac{1}{2}} \frac{\nabla V(X^{\Delta}_{t_{k}}).\sigma(t_{k},X^{\Delta}_{t_{k}})U_{k+1}}{V^{1-\alpha}(X^{\Delta}_{t_{k}})} + \alpha \eta \Delta^2 \frac{|b(t_k,X^{\Delta}_{t_{k}})|^{2}}{V^{1-\alpha}(X^{\Delta}_{t_{k}})} \\
& \ \ \ + 2 \alpha \eta \Delta^{\frac{3}{2}} \frac{b(t_k,X^{\Delta}_{t_{k}}) . \sigma(t_k,X^{\Delta}_{t_{k}}) U_{k+1}}{V^{1-\alpha}(X^{\Delta}_{t_{k}})} + \alpha \eta \Delta \frac{|\sigma(t_k,X^{\Delta}_{t_{k}}) U_{k+1}|^2}{V^{1-\alpha}(X^{\Delta}_{t_{k}})}.
\end{align*}

From \A{HD$_{\alpha}$}, for all $(x,u) \in \R^{d} \times \R^{q}$, we clearly have $\sup_{t\in [0,T]}|\nabla V(x) . b(t,x)| \leq (C_V C_b)^{\frac{1}{2}} V(x)$ and $\sup_{t\in [0,T]}|\sigma(t,x)u|^2 \leq C_{\sigma} V^{1-\alpha}(x) |u|^2$ which yields 
\begin{align*}
V^{\alpha}(X^{\Delta}_{t_{k+1}}) & \leq V^{\alpha}(X^{\Delta}_{t_{k}})(1 + \alpha (C_V C_{b})^{\frac{1}{2}} \Delta + \alpha \eta C_b \Delta^2) + \alpha \Delta^{\frac{1}{2}}(1+2\eta \Delta) \frac{(\nabla V(X^{\Delta}_{t_{k}}) + b(X^{\Delta}_{t_{k}})).\sigma(X^{\Delta}_{t_{k}}) U_{k+1}}{V^{1-\alpha}(X^{\Delta}_{t_{k}})} \\
& \ \ \ \ + C_{\sigma} \alpha \eta \Delta |U_{k+1}|^2.
\end{align*}

Using \A{HD$_{\alpha}$}, $\forall x \in \R^{d}$ the functions $g(x,.):u\mapsto \frac{(\nabla V(x) + b(x)).\sigma(x) u}{V^{1-\alpha}(x)}$ are Lipschitz, and more precisely satisfy
$$
\forall x \in \R^d, \ \ \ \ \sup_{(u, u^{'}) \in (\R^{q})^{2}} \frac{|g(x,u)-g(x,u^{'})|}{|u-u^{'}|} \leq (C^{1/2}_V + C^{1/2}_b)C^{1/2}_{\sigma} V^{\frac{\alpha}{2}}(x).
$$

Hence, from the Cauchy Schwarz inequality and since the law of the innovations satisfy \eqref{CONC_GAUSS} for some $\beta>0$, there exists $\epsilon_{\beta} >0$ such that for $\lambda < \min(1, \varepsilon_{\beta} (2\eta \alpha C_{\sigma} \Delta)^{-1})$, one has
\begin{align*}
\E\left[\left. \exp(\lambda V^{\alpha}(X^{\Delta}_{t_{k+1}}))\right| \mathcal{F}_{t_{k}}\right] & \leq \exp(\lambda V^{\alpha}(X^{\Delta}_{t_{k}})(1 + \alpha (C_V C_{b})^{\frac{1}{2}} \Delta + \alpha \eta C_b \Delta^2))  \\ 
& \ \ \times \E\left[\left. \exp(2 \lambda \alpha \Delta^{\frac{1}{2}}(1+2\eta \Delta) g(X^{\Delta}_{t_{k}},U_{k+1}) )\right| \mathcal{F}_{t_{k}}\right]^{\frac{1}{2}} \times \E\left[ \left. \exp(2 \lambda \eta \alpha C_{\sigma} \Delta |U_{k+1}|^2)\right| \mathcal{F}_{t_{k}}\right]^{\frac{1}{2}} \\
& \leq \exp(\lambda V^{\alpha}(X^{\Delta}_{t_{k}})(1 + \alpha (C_V C_{b})^{\frac{1}{2}} \Delta + \alpha \eta C_b \Delta^2) )  \\
& \ \ \times \exp(\lambda^2 \beta \alpha^2 \Delta (1+2\eta \Delta)^2 (C_V + C_b) C_\sigma V^{\alpha}(X^{\Delta}_{t_{k}})) \times \E\left[ \exp(2 \lambda \eta \alpha C_{\sigma} \Delta |U_{1}|^2)\right]^{\frac{1}{2}} \\
& \leq \exp(\lambda C(\Delta)V^{\alpha}(X^{\Delta}_{t_{k}})) \E\left[ \exp(2 \lambda \eta \alpha C_{\sigma} \Delta |U_{1}|^2)\right]^{\frac{1}{2}},
\end{align*}

\noindent where $C(\Delta):=1 + \Delta \left(\alpha (C_V C_{b})^{\frac{1}{2}} + \beta C_{\sigma} \alpha^2 (1+2\eta \Delta)^2 (C_V + C_b)+ \alpha \eta C_b \Delta\right)$. Now define $V_{k}=\frac{V^{\alpha}(X^{\Delta}_{t_{k}})}{C(\Delta)^{k}}$, for $k \in \left\{0, \cdots, N\right\}$. Taking expectation in both sides of the previous inequality clearly implies
$$
\E\left[\exp(\lambda V_{k+1})\right] \leq \E\left[\exp(\lambda V_{k})\right] \E\left[ \exp\left( \lambda \frac{2\eta \alpha C_{\sigma} \Delta}{C(\Delta)^{k+1}} |U_{1}|^2\right)\right]^{\frac{1}{2}}
$$

\noindent and by a straightforward induction, for $n \in \left\{0, \cdots, N\right\}$ we have
$$
\E\left[\exp(\lambda V_{n})\right] \leq \exp(\lambda V_{0}) \prod_{k=0}^{n-1} \E\left[\exp\left(\lambda \frac{2\eta \alpha C_{\sigma} \Delta}{C(\Delta)^{k+1}} |U_{1}|^2\right)\right]^{\frac{1}{2}},
$$

\noindent which finally yields, for $\lambda < \min(1, \varepsilon_{\beta} (2\eta \alpha C_{\sigma} \Delta C(\Delta)^{n})^{-1})$,
$$
\E\left[\exp(\lambda V^{\alpha}(X^{\Delta}_{t_{n}}))\right] \leq \exp(\lambda C(\Delta)^{n} V^{\alpha}(X_0)) \prod_{k=0}^{n-1} \E\left[\exp\left(\lambda 2\eta \alpha C_{\sigma} \Delta C(\Delta)^{k+1} |U_{1}|^2\right)\right]^{\frac{1}{2}}.
$$

Observe now that $C(\Delta)^{N} \leq \exp(C T)$ with $C:=C(b, \sigma V,\alpha, \Delta)= \alpha (C_V C_{b})^{\frac{1}{2}} +  \beta C_{\sigma} \alpha^2 (1+2\eta \Delta)^2 (C_V + C_b) + \alpha \eta C_b \Delta$. Using Jensen's inequality, the latter bound clairly provides the following control of the quantity of interest for $\lambda < \min(1, \varepsilon_{\beta} (2\eta \alpha C_{\sigma} T \exp(C T))^{-1})$
$$
\sup_{0 \leq n \leq N}\log\left(\E\left[\exp(\lambda V^{\alpha}(X^{\Delta}_{t_{n}}))\right]\right) \leq  \lambda \exp(C T) V^{\alpha}(X_0) + \frac{1}{2} \log\left(\E\left[\exp\left(\lambda 2\eta \alpha C_{\sigma} T \exp(CT) |U_{1}|^2\right)\right]\right). 
$$
\end{proof}
\begin{COROL}
\label{CONTLAPLACEBIS}
Under the same assumptions as Proposition \ref{CONT_LAPLACEV}, for all $\alpha \in(\frac12,1]$, one has
$$
\forall \lambda \geq0, \ \ \sup_{0 \leq n \leq N}\log\left(\E_x\left[\exp(\lambda V^{1-\alpha}(X^{\Delta}_{t_{n}}))\right]\right) \leq K_{3.1} (\lambda \vee \lambda^{\frac{\alpha}{2\alpha-1}})
$$

\noindent where $K_{3.1}:=\max\left(\Psi_1(T,\Delta,x,b,\sigma), \Psi_2(T,\Delta,x,b,\sigma)  \right)$ and
\begin{align*}
\Psi_1(T,\Delta,x,b,\sigma) & := e^{\frac{2\alpha-1}{\alpha}\underline{\rho}^{-\frac{1-\alpha}{2\alpha-1}}}\exp\left(\underline{\rho} \frac{1-\alpha}{\alpha}e^{CT}V^{\alpha}(x)+\frac{1}{2}\log\E[e^{\frac{\epsilon_{\beta}(1-\alpha)}{2\alpha}|U|^2}]\right) + \left(V^{1-\alpha}(x)+\left(\frac{C_\sigma \E[|U|^2]}{K}\right)^{\frac{1-\alpha}{\alpha}}\right)e^{(1-\alpha)KT},  \\
\Psi_2(T,\Delta,x,b,\sigma) & :=  \underline{\rho}^{-\frac{1-\alpha}{2\alpha-1}}\frac{2\alpha-1}{\alpha} + \underline{\rho} \frac{1-\alpha}{\alpha} e^{CT} V^{\alpha}(x) +\frac{1}{2}\log\E\left[\exp\left(\frac{\epsilon_{\beta}(1-\alpha)}{2\alpha}|U|^2\right)\right], \\
\underline{\rho} & :=\frac{1}{2}\min(1, \varepsilon_{\beta} (2\eta \alpha C_{\sigma} T \exp(C T))^{-1}), \\
C & :=C(b, \sigma V,\alpha, \Delta)= \alpha (C_V C_{b})^{\frac{1}{2}} +  \beta C_{\sigma} \alpha^2 (1+2\eta \Delta)^2 (C_V + C_b) + \alpha \eta C_b \Delta \\
K & := K(V,b, \Delta)= (C_V C_b)^{\frac12} + \eta C_b \Delta
\end{align*}
\end{COROL}
\begin{proof}For $\lambda \in [0,1]$, one has 
\begin{align*}
\E_x[\exp(\lambda V^{1-\alpha}(X_{t_n}))] & = 1+ \lambda \E_x[V^{1-\alpha}(X_{t_n})] + \sum_{k\geq 2} \frac{\lambda^k}{k!}\E_x[V^{(1-\alpha)k}(X_{t_n})] \\
& \leq 1+\lambda \E_x[V^{1-\alpha}(X_{t_n})] + \lambda \sum_{k\geq 0} \frac{1}{k!} \E_x[V^{(1-\alpha)k}(X_{t_n})] \\
& \leq \exp\left(\lambda (\E_x[V^{1-\alpha}(X_{t_n})]+\E_x[e^{V^{1-\alpha}(X_{t_n})}])\right),
\end{align*}

Tedious but simple computations, in the spirit of Proposition  \ref{CONT_LAPLACEV}, show that 
$$
\E_x[V^{1-\alpha}(X_{t_n})]\leq \E_x[V^{\alpha}(X_{t_n})]^{\frac{1-\alpha}{\alpha}} \leq \left(V^{1-\alpha}(x)+\left(\frac{C_\sigma \E[|U|^2]}{K}\right)^{\frac{1-\alpha}{\alpha}}\right)e^{(1-\alpha)KT}.
$$

\noindent with $K:=K(V,b, \Delta)= (C_V C_b)^{\frac12} + \eta C_b \Delta$.

\noindent Thanks to the following Young inequality, for all $\rho>0$, for all $x \in \R^d$, $V^{1-\alpha}(x) \leq  \frac{1-\alpha}{\alpha} \rho  V^{\alpha}(x) + \frac{2\alpha-1}{\alpha} \rho^{-\frac{1-\alpha}{2\alpha-1}}$, which is valid if $\alpha \in (\frac 12,1]$, one has for $\rho=\underline{\rho}:= \frac{1}{2}\min(1, \varepsilon_{\beta} (2\eta \alpha C_{\sigma} T \exp(C T))^{-1})$
\begin{align*}
\sup_{0 \leq n \leq N}\E_x[e^{V^{1-\alpha}(X^{\Delta,0,x}_{t_{n}})}] & \leq \exp( \frac{2\alpha-1}{\alpha} \rho^{-\frac{1-\alpha}{2\alpha-1}})\sup_{0 \leq n \leq N}\E_x\left[\exp\left(\frac{1-\alpha}{\alpha} \rho V^{\alpha}(X^{\Delta,0,x}_{t_{n}})\right)\right] \\ & \leq \exp(\frac{2\alpha-1}{\alpha}\underline{\rho}^{-\frac{1-\alpha}{2\alpha-1}}) \exp\left(\underline{\rho}\frac{1-\alpha}{\alpha}e^{CT}V^{\alpha}(x)+\frac{1}{2}\log\E[\exp(\frac{\epsilon_{\beta}(1-\alpha)}{2\alpha}|U|^2)]\right) 
\end{align*}

\noindent where we used Proposition \ref{CONT_LAPLACEV} for the last inequality.

Now, for all $\lambda>1$, using the Young type inequality $\lambda V^{1-\alpha}(X_{t_n}) \leq (\frac{2\alpha -1}{\alpha})\rho^{-\frac{1-\alpha}{2\alpha-1}}\lambda^{\frac{\alpha}{2\alpha-1}} + (\frac{1-\alpha}{\alpha})\rho V^{\alpha}(X_{t_n})$, valid for all $\rho>0$ (to be chosen later on) and for all $\alpha \in (\frac12,1]$, one derives
\begin{align*}
\E_x[\exp(\lambda V^{1-\alpha}(X_{t_n}))] & \leq \exp\left((\frac{2\alpha -1}{\alpha})\rho^{-\frac{1-\alpha}{2\alpha-1}}\lambda^{\frac{\alpha}{2\alpha-1}}\right) \E_x\left[\exp\left(\left(\frac{1-\alpha}{\alpha}\right)\rho V^{\alpha}(X_{t_n})\right)\right] \\
& \leq \exp\left(K\lambda^{\frac{\alpha}{2\alpha-1}} \right)
\end{align*}

\noindent with $K(\rho):= \frac{2\alpha -1}{\alpha} \rho^{-\frac{1-\alpha}{2\alpha-1}} + \log(\E_x\left[\exp\left(\left(\frac{1-\alpha}{\alpha}\right)\rho V^{\alpha}(X_{t_n})\right)\right])$ and $\frac{1-\alpha}{\alpha}\rho<\min(1, \varepsilon_{\beta} (2\eta \alpha C_{\sigma} T \exp(C T))^{-1})$. We select $\rho = \underline{\rho}$ in the last inequality to complete the proof and use Proposition \ref{CONT_LAPLACEV} to bound the quantity $K(\underline{\rho})$.
\end{proof}
\begin{COROL}
\label{CONTLAPLACETER}Under the same assumptions as Proposition \ref{CONT_LAPLACEV}, one has
$$
\forall \lambda \in [0, \lambda_{3.2}), \ \ \sup_{0 \leq n \leq N}\log\left(\E_x\left[\exp(\lambda^2 V^{1/2}(X^{\Delta}_{t_{n}}))\right]\right) \leq K_{3.2} \frac{(\lambda/\lambda_{3.2})^2}{1-(\lambda / \lambda_{3.2})}
$$

\noindent where $K_{3.2} := \lambda^2_{3.2} \exp(CT) (2 V^{1/2}(x) + 2 \eta \alpha C_{\sigma} \E[|U_1|^2] T)$ and $\lambda_{3.2}$ satisfies $\E[\exp(\lambda^2_{3.2} 2 \eta \alpha C_{\sigma} T \exp(CT)|U_1|^2)] \leq 2$.
\end{COROL}
\begin{proof} By definition of $\lambda_{3.2}$, we have $\forall k\geq 1, \lambda^{2k}_{3.2} ( 2 \eta \alpha C_{\sigma} T \exp(CT))^k \E[|U_1|^{2k}] \leq 2 k!$. Consequently, setting temporarily $C_1 := \exp(CT) V^{1/2}(x) , C_2:= 2 \eta \alpha  C_{\sigma} T \exp(CT)$ for sake of simplicity,  simple computations show that
\begin{align*}
\log\E\left[\exp\left(\lambda^2 C_2|U_{1}|^2\right)\right] - \lambda^2 C_2 \E[|U_1|^2]  & = \log\left( 1 + \sum_{k\geq1} \frac{\lambda^{2k}C^k_2 \E[|U_1|^{2k}]}{k!} \right) - \lambda^2 C_2 \E[|U_1|^2] \\
& \leq \sum_{k\geq2} \frac{\lambda^{2k}C^k_2 \E[|U_1|^{2k}]}{k!} \\
& \leq 2 \sum_{k\geq2} \left(\frac{\lambda}{\lambda_{3.2}}\right)^{2k} \leq \left\{ \begin{array}{l}
 2 \frac{(\lambda/\lambda_{3.2})^2}{1- (\lambda/\lambda_{3.2})}
, \ \mbox{if } \ \lambda < \lambda_{3.2}, \vspace*{0.2cm} \\
 + \infty ,\  \ \mbox{otherwise}.
\end{array}
\right.
\end{align*} 

\noindent hence, using Proposition \ref{CONT_LAPLACEV} for $\alpha=\frac12$ and $\forall \lambda \in [0, \lambda_{3.2})$, we clearly get
\begin{align*}
 \sup_{0 \leq n \leq N}\log\left(\E_x\left[\exp(\lambda^2 V^{1/2}(X^{\Delta}_{t_{n}}))\right]\right) & \leq \lambda^2_{3.2} \left(C_1 + \frac{C_2\E[|U_1|^2]}{2}\right) (\lambda/\lambda_{3.2})^2  + \frac{(\lambda/\lambda_{3.2})^2}{1- (\lambda/\lambda_{3.2})} \\
 & \leq 2  \lambda^2_{3.2} \left(C_1 + \frac{C_2\E[|U_1|^2]}{2}\right) \frac{(\lambda/\lambda_{3.2})^2}{1- (\lambda/\lambda_{3.2})}.
\end{align*}

\noindent This completes the proof.
\end{proof}
\begin{PROP}{(Control of the Lipschitz modulus of iterative kernels)}
\label{Lipschitz_control_Euler}
Denote the Lipschitz modulus of $b$ and $\sigma$ appearing in the diffusion process \eqref{EDS} by $[b]_1$ and $[\sigma]_1$, respectively. Denote by $P_k$ and $P_{k,p}=P_{k} \circ \cdots \circ P_{p-1}$, $k, \ p \in \left\{0,\cdots, N-1\right\}$, $k \leq p$ the (Feller) transition kernel and the iterative kernels of the Markov chain $X^{\Delta}$ defined by the scheme \eqref{EULER}, respectively. Then for all real-valued Lipschitz function $f$ and for all $k, \ p \in \left\{0,\cdots, N-1\right\}$, $k\leq p$ the functions $P_{k}(f)$ are Lipschitz-continuous and one has
$$
[P_{k,p}(f)]_{1} : = \sup_{(x,x') \in (\R^{d})^2} \frac{\left|P_{k,p}(f)(x)-P_{k,p}(f)(x')\right|}{|x-x'|} \leq [f]_{1} (1+C(b,\sigma,\Delta) \Delta)^{\frac{p-k}{2}}
$$ 

\noindent where $[f]_{1}$ stands for the Lipschitz modulus of the function $f$ and $C(b,\sigma,\Delta)=2[b]_1 + [\sigma]^2_{1} + \Delta [b]^2_1$.
\end{PROP}

\begin{proof} Using the Cauchy Schwarz inequality and \A{HS}, for all $(x,y) \in (\R^{d})^2$ and for all $k \in \left\{0,\cdots, N-1\right\}$, one has
\begin{align*}
|P_k(f)(x) - P_k(f)(y)| & \leq [f]_1 \E\left[\left|f(x+b(t_k,x)\Delta + \sigma(t_k,x) U_1) - f(y+b(t_k,y)\Delta + \Delta^{\frac{1}{2}}\sigma(t_k,y) U_1)\right|\right] \\
& \leq [f]_1 \E\left[\left|x-y + (b(t_k,x)-b(t_k,y))\Delta + \Delta^{\frac{1}{2}}(\sigma(t_k,x)-\sigma(t_k,y))U_1 \right|^{2}\right]^{\frac{1}{2}} \\
& \leq [f]_1 (1+C(b,\sigma,\Delta)\Delta )^{\frac{1}{2}} |x-y|.
\end{align*}

\noindent A straightforward induction argument completes the proof.

\end{proof}
\begin{PROP}{(Control of the Laplace transform)}\label{CONT_LAPLACE}
Denote by $X^{\Delta}_{T}$ the value at time $T$ of the scheme \eqref{EULER} associated to the diffusion \eqref{EDS}. Assume that the innovations $(U_n)_{n\geq1}$ in \eqref{EULER} satisfy \eqref{CONC_GAUSS} for some $\beta>0$. Let $f$ be a real-valued $1$-Lipschitz-continuous function defined on $\R^{d}$. For all $\lambda \geq0$ and for all $\alpha\in(\frac{1}{2},1]$, one has 
$$
\E_x\left[\exp(\lambda f(X^{\Delta}_{T}))\right] \leq \exp(\lambda \E_x\left[f(X^{\Delta}_{T})\right]) \exp\left(K_{3.1}(\varphi(T,b,\sigma,\Delta) \vee\varphi(T,b,\sigma,\Delta)^{\frac{\alpha}{2\alpha-1}})(\lambda^2\vee\lambda^{\frac{2\alpha}{2\alpha-1}})\right),
$$

\noindent with $\varphi(T,b,\sigma,\Delta) : = C_{\sigma}\beta\frac{(1+C(\Delta)\Delta)}{4C(\Delta)}e^{3C(\Delta) T}$ and $C(\Delta) := 2[b]_1 + [\sigma]^2_{1} + \Delta [b]^2_1$. 

%\noindent if $\alpha=\frac{1}{2}$, for all $\lambda < \varphi(T,b,\sigma,\Delta)^{-\frac{1}{2}} \min(1, \varepsilon_{\beta}^{\frac{1}{2}} (\eta C_{\sigma} T \exp(C T))^{-\frac{1}{2}})$, one has
%$$
%\E_x\left[\exp(\lambda f(X^{\Delta}_{T}))\right] \leq \exp(\lambda \E_x\left[f(X^{\Delta}_{T})\right])\exp\left(\lambda^2 \varphi(T,b,\sigma,\Delta) V^{\frac{1}{2}}(x) + \frac{1}{2} \log\E\left[e^{\lambda^2 \varphi(T,b,\sigma,\Delta)\eta C_\sigma T e^{CT}|U_1|^2}\right]\right),
%$$

%\noindent with $C:= \frac{1}{2}(C_V C_{b})^{\frac{1}{2}} + 2 (1+2\eta \Delta)^2+ \frac12 \eta C_b \Delta$.

\noindent If $\alpha=\frac{1}{2}$, for all $\lambda \in [0,\varphi(T,b,\sigma,\Delta)^{-1/2} \lambda_{3.2})$, one has
$$
\E_x\left[\exp(\lambda f(X^{\Delta}_{T}))\right] \leq \exp(\lambda \E_x\left[f(X^{\Delta}_{T})\right]) \exp\left(K_{3.2} \frac{(\lambda\varphi(T,b,\sigma,\Delta)^{1/2} /\lambda_{3.2})^2}{1-(\lambda\varphi(T,b,\sigma,\Delta)^{1/2} / \lambda_{3.2})}\right).
$$

\end{PROP}

\begin{proof} As mentionned earlier on in the introduction, we begin our proof using that the law $\mu$ of the innovation satisfies \eqref{CONC_GAUSS} and \A{HD$_{\alpha}$}. Hence, for $\lambda \geq0$ and $k \in \left\{0,\cdots,N-1\right\}$, one has 
\begin{align}
P_{k}(\exp(\lambda f))(x) & = \mathbb{E}\left[\exp\left(\lambda f\left(x+ b(t_{k},x)\Delta +\sigma(t_{k},x) \Delta^{1/2} U_{k+1} \right)\right)\right] \nonumber \\
 & \leq  \exp\left(\lambda P_{k}(f)(x) + \beta \frac{\lambda^2}{4}[f]^2_1 \Delta |\sigma(t_{k},x)|^2 \right) \nonumber\\
 & \leq  \exp\left(\lambda P_{k}(f)(x) + C_{\sigma} \beta \frac{\lambda^2}{4} [f]^2_1 \Delta V^{1-\alpha}(x) \right). \label{first_eq_dem}
\end{align}

Taking expectation from both sides of the last inequality and using the H\"{o}lder inequality with conjugate exponents $(p,q)$ (to be specified later on) leads to
\begin{equation}
\label{first_step_iteration}
\E_x\left[\exp(\lambda f(X^{\Delta}_{t_{k+1}}))\right] \leq \E_x\left[\exp(\lambda p P_k(f)(X^{\Delta}_{t_k}))\right]^{\frac{1}{p}} \E_x\left[\exp\left( \frac{qC_\sigma \beta }{4} \Delta \lambda^2 [f]^2_1 V^{1-\alpha}(X^{\Delta}_{t_k}) \right)\right]^{\frac{1}{q}}.
\end{equation}

\noindent Now, we apply the last inequality for $f:= P_{k+1,N}(f)$ and obtain
$$
\E_x\left[\exp(\lambda P_{k+1,N}(f)(X^{\Delta}_{t_{k+1}}))\right] \leq \E_x\left[\exp(\lambda p P_{k,N}(f)(X^{\Delta}_{t_k}))\right]^{\frac{1}{p}} \E_x\left[\exp\left( \frac{qC_\sigma \beta }{4} \Delta \lambda^2 [P_{k+1,N}(f)]^2_1 V^{1-\alpha}(X^{\Delta}_{t_k}) \right)\right]^{\frac{1}{q}}
$$

\noindent Consequently, an elementary induction yields
\begin{align*}
\E_x\left[\exp(\lambda f(X^{\Delta}_{T}))\right] & = \E_x\left[\exp(\lambda P_{N,N}(f)(X^{\Delta}_{t_{N}}))\right]  \\ 
& \leq \E_{x}\left[\exp(\lambda p^{N}P_{0,N}(f)(x))\right]^{\frac{1}{p^{N}}}\\
& \ \ \ \ \times \prod_{k=0}^{N-1} \left(\E_{x}\left[\exp\left(\frac{C_\sigma  \beta }{4} \lambda^2 q p^{2k}\Delta [P_{N-k,N}(f)]^2_1V^{1-\alpha}(X^{\Delta}_{t_{N-k-1}}) \right)\right]^{\frac{1}{q}}\right)^{\frac{1}{p^{k}}}  \\
& \leq \exp(\lambda \E_x\left[f(X^{\Delta}_{T})\right]) \exp\left( \sum_{k=0}^{N-1}\frac{1}{p^{k}} \frac{1}{q} \sup_{0\leq n \leq N}\log\left(\E_{x}\left[e^{ \frac{C_{\sigma}\beta}{4} \lambda^2 \Delta q p^{2N}(1+C(\Delta)\Delta)^{N}V^{1-\alpha}(X^{\Delta}_{t_{n}}) }\right]\right)\right)
\end{align*}

\noindent where we used Proposition \ref{Lipschitz_control_Euler} for the last inequality. Observe now that since $(p,q)$ are conjugate exponents, we have $\frac{1}{q}\sum_{k=0}^{N-1}\frac{1}{p^{k}} = \frac{1}{q} (1-\frac{1}{p^{N}})\frac{1}{1-\frac{1}{p}} \leq \frac{1}{q} \frac{p}{p-1} =1$, so that 
$$
\E_x\left[\exp(\lambda f(X^{\Delta}_{T}))\right] \leq \exp(\lambda \E_x\left[f(X^{\Delta}_{T})\right]) \exp\left(\sup_{0\leq n \leq N}\log\left(\E_{x}\left[e^{ \frac{C_{\sigma}\beta}{4} \lambda^2 \Delta q p^{2N}(1+C(\Delta)\Delta)^{N}V^{1-\alpha}(X^{\Delta}_{t_{n}}) }\right]\right)\right).
$$

\noindent Setting $p:=1+C(\Delta)\Delta$, $q= \frac{p}{p-1}=\frac{1+C(\Delta) \Delta}{C(\Delta) \Delta}$ and using the straightforward inequality $(1+C(\Delta)\Delta)^{3N}\leq \exp(3C(\Delta)T)$, we derive
$$
\E_x\left[\exp(\lambda f(X^{\Delta}_{T}))\right] \leq \exp(\lambda \E_x\left[f(X^{\Delta}_{T})\right]) \exp\left(\sup_{0\leq n \leq N}\log\left(\E_{x}\left[e^{ \frac{C_{\sigma}\beta(1+C(\Delta))}{4C(\Delta)}e^{3C(\Delta)T} \lambda^2 V^{1-\alpha}(X^{\Delta}_{t_{n}}) }\right]\right)\right).
$$

\noindent We set $\varphi(T,b,\sigma,\Delta) : = C_{\sigma}\beta\frac{(1+C(\Delta)\Delta)}{4C(\Delta)}e^{3C(\Delta) T}$. For $\alpha \in (\frac12,1]$, Corollary \ref{CONTLAPLACEBIS}  clearly implies
$$
\E_x\left[\exp(\lambda f(X^{\Delta}_{T}))\right] \leq \exp(\lambda \E_x\left[f(X^{\Delta}_{T})\right]) \exp\left(K_{3.1}(\varphi(T,b,\sigma,\Delta) \vee\varphi(T,b,\sigma,\Delta)^{\frac{\alpha}{2\alpha-1}})(\lambda^2 \vee \lambda^{\frac{2\alpha}{2\alpha-1}})\right)
$$ 

%\noindent and for $\alpha=\frac12$, according to Proposition \ref{CONT_LAPLACEV},  for $\lambda^2 < \varphi(T,b,\sigma,\Delta)^{-1} \min(1, \varepsilon_{\beta} (\eta C_{\sigma} T \exp(C T))^{-1})$, one has
%$$
%\E_x\left[\exp(\lambda f(X^{\Delta}_{T}))\right] \leq \exp(\lambda \E_x\left[f(X^{\Delta}_{T})\right])\exp\left(\lambda^2 \varphi(T,b,\sigma,\Delta) V^{\frac{1}{2}}(x) + \frac{1}{2} \log\E\left[e^{\lambda^2 \varphi(T,b,\sigma,\Delta)\eta C_\sigma T e^{CT}|U_1|^2}\right]\right).
%$$ 

\noindent and for $\alpha=\frac12$, according to Proposition \ref{CONTLAPLACETER},  for $\lambda < \varphi(T,b,\sigma,\Delta)^{-1/2} \lambda_{3.2}$, one has
$$
\E_x\left[\exp(\lambda f(X^{\Delta}_{T}))\right] \leq \exp(\lambda \E_x\left[f(X^{\Delta}_{T})\right])\exp\left(K_{3.2} \frac{(\lambda\varphi(T,b,\sigma,\Delta)^{1/2} /\lambda_{3.2})^2}{1-(\lambda\varphi(T,b,\sigma,\Delta)^{1/2} / \lambda_{3.2})}\right).
$$

\end{proof}
\subsection{Proof of Theorem \ref{GEN_CONC_EULER}}

We will prove the result for the process $X$ solution of \eqref{EDS}. The proof for the continuous Euler scheme is similar.

\begin{LEMME}
\label{LEM_CONT_DIFF}
Under the assumptions of Theorem \ref{GEN_CONC_EULER}, for all $p\geq1$, one has
$$
\E_{x}[\sup_{0\leq t \leq T}|X_t|^{2p}] \leq (1+|x|)^{2p} \exp(26p^2(1+(C_b\vee C_{\sigma})T)).
$$
\end{LEMME}

\begin{proof}
Let $g:x \mapsto \sqrt{1+|x|^2}$ satisfying for all $x\in \R^d$, $\nabla g(x)= g^{-1}(x) x$,  $\nabla^2 g(x)=g^{-1}(x)I_d - g^{-3}(x)xx^*$ and $V:x\mapsto g^{2p}(x)$. We apply It\^{o}'s formula to the process $V(X_t)$ with $\nabla V(x)=2p g(x)^{2p-1}\nabla g(x)$ and $\nabla^{2}V(x)=2p g(x)^{2p-1} \nabla^2 g(x) + 2p(2p-1)g(x)^{2p-2} \nabla g(x) \nabla g(x)^*$ noticing that for all $t \in [0,T]$
\begin{align*}
\nabla V(x).b(t,x) + \frac12 Tr(\sigma^*\nabla^2V \sigma)(t,x) & \leq 2p C_b g(x)^{2p-1}(1+|x|) + \frac12 C_{\sigma} (1+|x|^2) ||\nabla^2 V(x)||\\
 &\leq 4pC_b g(x)^{2p} + \frac12 C_{\sigma} (1+|x|^2) (4p g(x)^{2p-2}+ 2p(2p-1)g(x)^{2p-2})\\
 &\leq 4p(C_b\vee C_{\sigma}) g(x)^{2p} + 2p (C_b\vee C_{\sigma}) g(x)^{2p} +p(2p-1) (C_b\vee C_{\sigma}) g(x)^{2p} \\
 & \leq 8p^2 (C_b\vee C_{\sigma}) V(x)
\end{align*}

\noindent we clearly obtain,
\begin{equation}
\label{ineq_marting}
V(X^{\tau_{m}}_{t}) \leq  V(x) + 8p^2 (C_b\vee C_{\sigma}) \int_{0}^{t} V(X^{\tau_m}_{s})ds + \int_0^{t\wedge \tau_{m} } (\nabla V^*\sigma)(X^{\tau_m}_s) dW_s,
\end{equation}

\noindent where we classically introduced the stopping time $\tau_{m}:=\inf \left\{t\geq0: |X_t-x| \geq m \right\}$ for $m \in \N^{*}$ and the notation $X^{\tau_m}:=(X_{t \wedge \tau_m})_{t \geq0}$. The stochastic integral $
M^{m}_t := \int_0^{t\wedge \tau_{m} } (\nabla V^*\sigma)(X^{\tau_m}_s) dW_s$ defines a continuous martingale so that taking expectation in the previous inequality clearly yields 
$$
\E_x[V(X^{\tau_m}_{t})] \leq  V(x) + 8p^2 (C_b\vee C_{\sigma}) \int_{0}^t \E_{x}[V(X^{\tau_m}_{s})] ds.
$$

\noindent Now, using Gronwall's lemma we derive 
$$
\forall m \in \N^{*}, \ \ \sup_{t \in [0,T]}\E_x[V(X^{\tau_m}_{t})] \leq (1+|x|)^{2p} \exp(8p^2(C_b\vee C_\sigma) T)
$$

As $\tau_{m}\rightarrow +\infty$ $a.s.$, as $m\rightarrow +\infty$ (since $\sup_{s\in [0,t]} |X_s| < + \infty$) using Fatou's lemma, we finally obtain for all $p\geq1$
\begin{equation}
\label{first_bound}
\sup_{0\leq t \leq T}\E_x[V(X_t)] = \sup_{0\leq t \leq T}\E_x[g(X_t)^{2p}]   \leq  (1+|x|)^{2p} \exp(8p^2(C_b\vee C_\sigma) T).
\end{equation}

\noindent We then observe that It\^{o}'s formula also implies 
\begin{equation}
\label{final_estimate}
\E_x[\sup_{0\leq s\leq t}V(X^{\tau_m}_t)] \leq  V(x) + 8p^2 (C_b\vee C_{\sigma}) \int_{0}^t \E_x[\sup_{0\leq u\leq s}V(X^{\tau_m}_u)]  ds + \E_x[(M^{m}_t)^{*}]
\end{equation}

\noindent where $(M^{m}_t)^{*}:= \sup_{0\leq s \leq t} M^{m}_s$. Combining Jensen's and Doob's inequalities, one clearly gets
\begin{align*}
\E_x[(M^{m}_t)^{*}]^2 & \leq \E_x[((M^{m}_t)^{*})^2] \leq 4 \E_x[(M^{m}_t)^2]  \leq 16 p^2 C_{\sigma} \int_0^t \E_x[g(X^{\tau_m}_s)^{4p}] ds \\
& \leq 16p^2 C_{\sigma} T (1+|x|)^{4p}\exp(32p^2 (C_b\vee C_{\sigma})T)
\end{align*}

\noindent where we used $\forall x \in \R^d, \ (\nabla V^* \sigma)^2(x) \leq 4p^2 C_{\sigma} g(x)^{4p-2} (1+|x|^2)=4p^2 C_{\sigma} g(x)^{4p}$ and \eqref{first_bound} for the last inequality. Consequently, plugging the latter estimate into \eqref{final_estimate}, one has for all $t\in [0,T]$
\begin{align*}
\E_x[\sup_{0\leq s\leq t}V(X^{\tau_m}_t)] & \leq V(x) + 4p (C_\sigma T)^{\frac12} (1+|x|)^{2p}  \exp(16p^2(C_b\vee C_{\sigma})T)+ 8p^2 (C_b\vee C_{\sigma}) \int_{0}^t \E_x[\sup_{0\leq u\leq s}V(X^{\tau_m}_u)]  ds \\
& \leq (1+|x|)^{2p} (1+ 4p (C_\sigma T)^{\frac12}  \exp(16p^2(C_b\vee C_{\sigma})T) ) + 8p^2 (C_b\vee C_{\sigma}) \int_{0}^t \E_x[\sup_{0\leq u\leq s}V(X^{\tau_m}_u)]  ds
\end{align*}

\noindent so that using Gronwall's lemma yields and passing to the limit $m\rightarrow +\infty$, for all $p\geq 1$
$$
\E_{x}[\sup_{0\leq t \leq T}|X_t|^{2p}] \leq \E_x[\sup_{0\leq s\leq T}V(X_t)]  \leq 2 (1+|x|)^{2p} \exp(26p^2(1+(C_b\vee C_{\sigma})T)).
$$
\end{proof}

For all real-valued and $1$-Lipschitz function $f$ defined on $\mathcal{C}$ and for all $p\geq1$, one has
\begin{align}
\E_x[|f(X)-\E_x[f(X)]|^{2p}] & = \E_x[|f(X)-f(0)+ f(0)-\E_x[f(X)]|^{2p}] \leq 2^{2p}\E_{x}[||X||_{\infty}^{2p}] \nonumber \\
& \leq 2^{2p+1} (1+|x|)^{2p} \exp(26p^2(1+(C_b\vee C_{\sigma})T))  \label{CONT_LIP1} 
%& \leq 2^{4p} (1+|x|)^{2p} \exp(26p^2(1+(C_b\vee C_{\sigma})T))
\end{align}

\noindent where we used Lemma \ref{LEM_CONT_DIFF} for the last inequality. Now, combining the Chebyshev and Rosenthal inequalities for independent zero-mean random variables (see e.g. \cite{Johnson1985}), for all $p\geq 1$, there exists $C_{2p}>0$ such that 
\begin{align*}
\P_{x}\left(\frac{1}{M}|\sum_{k=1}^{M} f(X^{k})-\E_{x}[f(X)]| \geq r \right) & \leq \frac{\E_{x}[(\sum_{k=1}^{M} f(X^{k})-\E_{x}[f(X)])^{2p}]}{r^{2p}M^{2p}}  \leq C_{2p} \frac{\E_{x}[|f(X)-\E_{x}[f(X)]|^{2p}]}{r^{2p}M^{p}} \\
& \leq 2 \frac{ (2(1+|x|))^{2p} \exp(28p^2(1+(C_b\vee C_{\sigma})T))}{r^{2p}M^{p}}:=2 \exp(-\varphi(p))
\end{align*}

\noindent with $\varphi(p):=-\kappa(b,\sigma,T)p^2+p\log(\frac{r^2M}{(2(1+|x|))^2})$ and where we used for all $p \geq1$, $C_{2p}\leq (2p)^{2p} \leq \exp(2p^2)$, see e.g. p.235-236 in \cite{Johnson1985}, and \eqref{CONT_LIP1} for the last inequality. Optimizing the latter inequality with respect to $p$ with $p\geq 1$, i.e. selecting $p=\frac{1}{2\kappa(b,\sigma,T)}\log(\frac{r^2M}{(2(1+|x|))^2})$, we obtain
$$
\P_{x}\left(\frac{1}{M}|\sum_{k=1}^{M} f(X^{k})-\E_{x}[f(X)]| \geq r \right) \leq 2 \exp\left(-\frac{1}{4 \kappa(b,\sigma,T)}\log\left(\frac{r^2M}{(2(1+|x|))^2}\right)^2\right)
$$

\noindent for $r^2 M\geq (2(1+|x|))^2 \exp(2\kappa(b,\sigma,T))$. Otherwise, using the Jensen and Rosenthal inequalities, one has for all $p \in [0,1]$
\begin{align*}
\E_x[(\sum_{k=1}^{M}f(X^{k})-\E_{x}[f(X)])^{2p}] & \leq \E_x[(\sum_{k=1}^{M}f(X^{k})-\E_{x}[f(X)])^2]^{p} \leq \left(M C_2 \E_x[|f(X)-\E_x[f(X)]|^{2}] \right)^{p} \\
& \leq M^{p} \left( 4 (2(1+|x|))^{2} \exp(\kappa(b,\sigma,T)) \right)^{p}
\end{align*}

\noindent where we used \eqref{CONT_LIP1} for the last inequality. Now, noticing that we have $4e \leq \exp(\kappa(b,\sigma,T))$, Chebyshev's inequality yields
\begin{align*}
\P_{x}\left(\frac{1}{M}|\sum_{k=1}^{M} f(X^{k})-\E_{x}[f(X)]| \geq r \right) & \leq \frac{C^{p}}{r^{2p}M^p}\leq 2 \frac{(Cp)^p}{r^{2p}M^{p}} \leq 2 \exp(-\varphi(p))
\end{align*}

\noindent with $\varphi(p):=-p\log(p)+p\log(\frac{r^2M}{C})$, $C:=(2 (1+|x|))^{2}\exp(2\kappa(b,\sigma,T)-1)$ and where we used that for all $p\geq0$, $C^{p}\leq 2(Cp)^{p}$ since the function $p \mapsto 2 p^p$ is minimized for $p=\exp(-1)$ and $2\exp(-1/e))>1$. Consequently, optimizing over $p$ such that $p\leq 1$, i.e. selecting $p=\frac{r^2M}{C e}$, one has
$$
\P_{x}\left(\frac{1}{M}|\sum_{k=1}^{M} f(X^{k})-\E_{x}[f(X)]| \geq r \right) \leq 2\exp\left(-\frac{r^2M}{(2 (1+|x|))^{2}\exp(2\kappa(b,\sigma,T))}\right)
$$ 

\noindent for $r^2M\leq C e=(2 (1+|x|))^{2}\exp(2\kappa(b,\sigma,T))$. This completes the proof.
\mysection{Stochastic Approximation Algorithm: Proof of the main Results}
\label{PROOF_RM}

Throughout this section we will assume that \A{HL}, \A{HLS}$_\alpha$ and \A{HUA} are in force.

\subsection{Proof of Theorem \ref{THM_TE_SA}}

The proof of Theorem \ref{THM_TE_SA} is divided into several propositions.

\begin{PROP}
\label{CONT_LAPLACEL}
Denote by $\theta:=(\theta_{n})_{0\leq n \leq N}$ the scheme \eqref{RM} with step sequence $\gamma=(\gamma_n)_{0\leq n \leq N}$ satisfying \eqref{STEP}. Assume that the innovations $(U_i)_{i\geq1}$ of \eqref{RM} satisfy \eqref{CONC_GAUSS} for some $\beta>0$. Then, there exists $\varepsilon_{\beta} >0$ which only depends on the law $\mu$ such that for all $\lambda <  \min(1, \varepsilon_{\beta} (8\eta \alpha C_{\alpha}^2\Pi_{2,N})^{-1})$, one has
$$
\sup_{0 \leq n \leq N}\log\left(\E_{\theta_0}\left[\exp(\lambda L^{\alpha}(\theta_n))\right]\right) \leq  (L^{\alpha}(\theta_0) + \underline{C}  \sum_{k=0}^{N-1} \gamma^2_{k+1} ) \Pi_{2,N}  \lambda + \left(\frac{1}{2} \sum_{k=0}^{N-1} \gamma^2_{k+1}\right) \log\left(\E\left[\exp\left(8 \eta \alpha C_{\alpha}^2\Pi_{2,N}  \lambda  |U|^2\right)\right]\right). 
$$

\noindent with $\Pi_{2,N}=\Pi_{2,N}(\alpha):=\prod_{k=0}^{N-1}(1+ (2 \eta \alpha C_h + \frac{\beta}{2}  \alpha^2 C_{\alpha}^2 )\gamma_{k+1}^2)$ and $\underline{C} =  4 \eta \alpha C^2_\alpha \E[|U|^2]$.
\end{PROP}

\begin{proof} The proof relies on similar arguments as those used in the proof of Proposition \ref{CONT_LAPLACEV}. Using the concavity of $x\mapsto x^{\alpha}$, $\alpha \in (0,1]$, a Taylor expansion of order 2 of the function $L$, and finally \A{HLS}$_\alpha$, for all $k \in \left\{0,\cdots, N-1\right\}$, we have  
\begin{align*}
L^{\alpha}(\theta_{k+1}) - L^{\alpha}(\theta_{k}) & \leq \alpha L^{\alpha-1}(\theta_{k})\left(\nabla L(\theta_{k}).(\theta_{k+1}-\theta_{k})+\eta |\theta_{k+1}-\theta_{k}|^2\right), \\
 & = -\gamma_{k+1}\alpha L^{\alpha-1}(\theta_{k}) \left\langle \nabla L(\theta_{k}), h(\theta_k)\right\rangle -\gamma_{k+1} \alpha L^{\alpha-1}(\theta_{k}) \left\langle \nabla L(\theta_{k}), (H(\theta_k,U_{k+1})-h(\theta_k)\right\rangle \\
 & \ \ \ \ + \alpha \eta \gamma^2_{k+1} L^{\alpha-1}(\theta_{k}) |H(\theta_k,U_{k+1})|^2, \\
 & \leq -\gamma_{k+1} \alpha L^{\alpha-1}(\theta_{k}) \left\langle \nabla L(\theta_{k}), H(\theta_k,U_{k+1})-h(\theta_k)\right\rangle + 2 \eta \alpha \gamma^2_{k+1} L^{\alpha-1}(\theta_{k}) |H(\theta_k,U_{k+1})-h(\theta_k)|^2 \\
 & \ \ \ + 2 \eta \alpha \gamma^2_{k+1} L^{\alpha-1}(\theta_{k})|h(\theta_k)|^2.
\end{align*}

\noindent Let us note that \A{HLS}$_\alpha$ implies that $\forall (\theta,u) \in \R^{d}\times\R^{q}$, $|H(\theta,u)-h(\theta)|^2 = |H(\theta,u)-\E[H(\theta,U)]|^2 \leq 2 C^2_{\alpha} L^{1-\alpha}(\theta)(\E[|U|^2]+|u|^2)$ which leads to
\begin{align*}
L^{\alpha}(\theta_{k+1}) - L^{\alpha}(\theta_{k}) & \leq -\gamma_{k+1} \alpha L^{\alpha-1}(\theta_{k}) \left\langle \nabla L(\theta_{k}), H(\theta_k,U_{k+1})-h(\theta_k)\right\rangle + 4 \eta \alpha C_{\alpha}^2\gamma_{k+1}^2 \E[|U|^2] + 4 \eta \alpha C_{\alpha}^2 \gamma_{k+1}^2 |U_{k+1}|^2 \\
& \ \ \ + 2 \eta \alpha  C_{h} \gamma^2_{k+1}  L^{\alpha}(\theta_k).
\end{align*}

\noindent Using again \A{HLS}$_\alpha$, $\forall \theta \in \R^{d}$ the functions $g(\theta,.):u\mapsto \frac{\left\langle \nabla L(\theta), H(\theta,u)-h(\theta)\right\rangle}{L^{1-\alpha}(\theta)}$ are Lipschitz and more precisely satisfy
$$
\forall \theta \in \R^d, \ \ \ \sup_{(u, u^{'}) \in (\R^{q})^{2}} \frac{|g(\theta,u)-g(\theta,u^{'})|}{|u-u^{'}|} \leq C_{\alpha} L^{\frac{\alpha}{2}}(\theta).
$$

Consequently, denoting $\underline{C} =  4 \eta \alpha C^2_\alpha \E[|U|^2]$, from the Cauchy-Schwarz inequality and since the law of the innovation satisfies \eqref{CONC_GAUSS} for some $\beta>0$, there exists $\epsilon_{\beta} >0$ such that for $\lambda < \min(1, \varepsilon_{\beta} (8\eta \alpha C_{\alpha}^2 \gamma_{1}^2)^{-1})$, one has
\begin{align*}
\E\left[\left. \exp(\lambda L^{\alpha}(\theta_{k+1})\right| \mathcal{F}_{k}\right] & \leq \exp(\lambda (1+2\eta \alpha C_h \gamma^2_{k+1} ) L^{\alpha}(\theta_k)) \exp(\underline{C} \gamma^2_{k+1} \lambda) \E\left[\left. \exp(-2\alpha\lambda \gamma_{k+1}g(\theta_k, U_{k+1}) )\right| \mathcal{F}_{k}\right]^{\frac{1}{2}}  \\
& \ \ \ \times \E\left[ \left. \exp(8 \eta \alpha \lambda C_{\alpha}^2 \gamma_{k+1}^2 |U_{k+1}|^2 )\right| \mathcal{F}_{k}\right]^{\frac{1}{2}} \\
& \leq \exp(\lambda (1+ (2 \eta \alpha C_h + \frac{\beta}{2} C^2_\alpha \alpha) \gamma_{k+1}^2) L^{\alpha}(\theta_k)) \exp(\underline{C} \gamma^2_{k+1} \lambda)  \E\left[ \exp( 8 \eta \alpha \lambda C_{\alpha}^2 \gamma_{k+1}^2 |U|^2 )\right]^{\frac{1}{2}}
\end{align*}

\noindent In the aim of simplifying notations, we define $\Pi_{2,n}:=\prod_{k=0}^{n-1}(1+(2 \eta \alpha C_h + \frac{\beta}{2} C^2_\alpha \alpha) \gamma_{k+1}^2)$ and temporarily set $L_{k}:=\frac{L^{\alpha}(\theta_k)}{\Pi_{2,k}}$, for $k \in \left\{0, \cdots, N\right\}$. Taking expectation in both sides of the previous inequality clearly implies
$$
\E_{\theta_0}\left[\exp(\lambda L_{k+1})\right] \leq \E_{\theta_0}\left[\exp(\lambda L_{k})\right] \exp\left(\underline{C} \frac{\gamma^2_{k+1}}{\Pi_{2,k+1}} \lambda\right) \E\left[\exp\left(8 \eta \alpha C_{\alpha}^2 \frac{\gamma_{k+1}^2}{\Pi_{2,k+1}} \lambda |U|^2\right)\right]^{\frac{1}{2}}
$$

\noindent and by a straightforward induction, for $n \in \left\{0, \cdots, N\right\}$ we have
$$
\E_{\theta_0}\left[\exp(\lambda L_{n})\right] \leq \exp(\lambda L_{0}) \exp\left( \underline{C} \sum_{k=0}^{n-1} \frac{\gamma^2_{k+1}}{\Pi_{2,k+1}} \lambda\right) \prod_{k=0}^{n-1} \E\left[\exp\left(8 \eta \alpha C_{\alpha}^2 \frac{\gamma_{k+1}^2}{\Pi_{2,k+1}} \lambda |U|^2\right)\right]^{\frac{1}{2}},
$$

\noindent which finally yields for $\lambda < \min(1, \varepsilon_{\beta} (8\eta \alpha C_{\alpha}^2 \gamma_{1}^2)^{-1})$
$$
\E_{\theta_0}\left[\exp(\lambda L^{\alpha}(\theta_n))\right] \leq \exp(\Pi_{2,n} L^{\alpha}(\theta_0)\lambda) \exp\left( \underline{C} \sum_{k=0}^{n-1} \frac{\Pi_{2,n}}{\Pi_{2,k+1}}\gamma^2_{k+1} \lambda\right) \prod_{k=0}^{n-1} \E\left[\exp\left(8 \eta \alpha C_{\alpha}^2  \frac{\Pi_{2,n}}{\Pi_{2,k+1}}\gamma_{k+1}^2  \lambda |U|^2\right)\right]^{\frac{1}{2}}.
$$

Up to a modification of a constant, we can assume without loss of generality that $\sup_{0 \leq n\leq N}\gamma_{n+1} = \gamma_1 \leq 1$ so that using the Jensen's inequality, the latter bound clairly provides the following control of the quantity of interest for $\lambda <  \min(1, \varepsilon_{\beta} (8 \eta \alpha C_{\alpha}^2\Pi_{2,N})^{-1})$
$$
\sup_{0 \leq n \leq N}\log\left(\E_{\theta_0}\left[e^{\lambda L^{\alpha}(\theta_n)}\right]\right) \leq  \left(L^{\alpha}(\theta_0) + \underline{C} \sum_{k=0}^{N-1}\gamma^2_{k+1}\right)\Pi_{2,N} \lambda + \left(\frac{1}{2} \sum_{k=0}^{N-1} \gamma^2_{k+1}\right) \log\left(\E\left[e^{8 \eta \alpha C_{\alpha}^2 \Pi_{2,N} \lambda |U|^2}\right]\right). 
$$

\end{proof}
\begin{COROL}\label{CONTLAPLACELBIS}
Assume that the assumptions of Proposition \ref{CONT_LAPLACEL} are satisfied. Then, for all $ \alpha \in (\frac12,1]$, one has
$$
\forall \lambda \geq 0, \ \ \sup_{0\leq n\leq N} \log\left(\E_{\theta_0}\left[\exp(\lambda L^{1-\alpha}(\theta_n))\right]\right) \leq K_{4.1} (\lambda \vee \lambda^{\frac{\alpha}{2\alpha-1}})
$$

\noindent where $K_{4.1}:=\max(\Psi_1(\gamma,\alpha,\theta_0,H),\Psi_2(\gamma,\alpha,\theta_0,H))$ and 
\begin{align*}
\Psi_1(\gamma,\alpha,\theta_0,H) & = \left(L^{1-\alpha}(\theta_0) + (8 \eta \alpha C^2_{\alpha}\E[|U|^2] \sum_{k=0}^{N-1} \gamma^2_{k+1})^{\frac{1-\alpha}{\alpha}} \right) \prod_{k=0}^{N-1} (1+2\eta (1-\alpha) C_h \gamma^2_{k+1}) +  \exp\left(\frac{2\alpha-1}{\alpha}\overline{\rho}^{-\frac{1-\alpha}{2\alpha-1}}\right)\\\ 
& \times \exp\left(\left(L^{\alpha}(\theta_0) + 2\alpha \underline{C} \sum_{k=0}^{N-1}\gamma^2_{k+1}\right)\Pi_{2,N} \overline{\rho}\frac{1-\alpha}{\alpha} + \left(\frac{1}{2} \sum_{k=0}^{N-1} \gamma^2_{k+1}\right) \log\left(\E\left[e^{\frac{\varepsilon_{\beta}(1-\alpha)}{2\alpha} |U|^2}\right]\right) \right) \\
\Psi_2(\gamma,\alpha,\theta_0,H) & = \frac{2\alpha-1}{\alpha}\overline{\rho}^{-\frac{1-\alpha}{2\alpha-1}} + \left(L^{\alpha}(\theta_0) + \underline{C} \sum_{k=0}^{N-1}\gamma^2_{k+1}\right)\Pi_{2,N} \overline{\rho}\frac{1-\alpha}{\alpha} + \left(\frac{1}{2} \sum_{k=0}^{N-1} \gamma^2_{k+1}\right) \log\left(\E\left[e^{\frac{\varepsilon_{\beta}(1-\alpha)}{2\alpha} |U|^2}\right]\right) \\
\overline{\rho} & = \frac12 \min(1, \varepsilon_{\beta} (8 \eta \alpha C_{\alpha}^2\Pi_{2,N})^{-1})
\end{align*}
\end{COROL}

\begin{proof} We only give a sketch of proof since it is rather similar to the one of Corollary \ref{CONTLAPLACEBIS}. For $\lambda \in [0,1]$, one has
$$
\E_{\theta_0}[\exp\left( \lambda L^{1-\alpha}(\theta_n)\right)] \leq \exp\left(\lambda(\E_{\theta_0}[L^{1-\alpha}(\theta_n)] + \E_{\theta_0}[\exp(L^{1-\alpha}(\theta_n))] \right).
$$

Tedious but simple computations in the spirit of Proposition \ref{CONT_LAPLACEL} easily show that 
$$
\sup_{0\leq n\leq N} \E_{\theta_0}[L^{1-\alpha}(\theta_n)]  \leq \sup_{0\leq n\leq N} \E_{\theta_0}[L^{\alpha}(\theta_n)]^{\frac{1-\alpha}{\alpha}} \leq \left(L^{1-\alpha}(\theta_0) + (8 \eta \alpha C^2_{\alpha}\E[|U|^2] \sum_{k=0}^{N-1} \gamma^2_{k+1})^{\frac{1-\alpha}{\alpha}} \right) \prod_{k=0}^{N-1} (1+2\eta (1-\alpha) C_h \gamma^2_{k+1}).
$$

Moreover, thanks to the Young type inequality $L^{1-\alpha}(\theta) \leq \frac{1-\alpha}{\alpha}\rho L^{\alpha}(\theta) + \frac{2\alpha-1}{\alpha}\rho^{-\frac{1-\alpha}{2\alpha-1}}$, for every $(\rho,\theta) \in \R^{*}_{+}\times \R^{d}$ and $\alpha \in(\frac12,1]$ and using Proposition \ref{CONT_LAPLACEL}, one obtains for $\rho=\overline{\rho}:=\frac12 \min(1, \varepsilon_{\beta} (8 \eta \alpha C_{\alpha}^2\Pi_{2,N})^{-1})$
\begin{align*}
\sup_{0\leq n\leq N} \E_{\theta_0}[e^{L^{1-\alpha}(\theta_n)}] & \leq \exp\left(\frac{2\alpha-1}{\alpha}\overline{\rho}^{-\frac{1-\alpha}{2\alpha-1}}\right)\exp\left(\left(L^{\alpha}(\theta_0) + \underline{C} \sum_{k=0}^{N-1}\gamma^2_{k+1}\right)\Pi_{2,N} \overline{\rho}\frac{1-\alpha}{\alpha} \right. \\
& \left. + \left(\frac{1}{2} \sum_{k=0}^{N-1} \gamma^2_{k+1}\right) \log\left(\E\left[e^{\frac{\varepsilon_{\beta}(1-\alpha)}{2\alpha} |U|^2}\right]\right) \right),
\end{align*}

\noindent so that for all $\lambda \in [0,1]$ 
$$
\E_{\theta_0}[\exp\left( \lambda L^{1-\alpha}(\theta_n)\right)] \leq \Psi_1(\gamma,\alpha,\theta_0,H) \lambda.
$$

\noindent Now, for $\lambda>1$, we use the Young-type inequality $\lambda L^{1-\alpha}(\theta_n) \leq \frac{2\alpha-1}{\alpha}\rho^{-\frac{1-\alpha}{2\alpha-1}} \lambda^{\frac{\alpha}{2\alpha-1}} + \frac{1-\alpha}{\alpha}\rho L^{\alpha}(\theta_n)$ to derive
\begin{align*}
\E_{\theta_0}[\exp(\lambda L^{1-\alpha}(\theta_n))] & \leq \exp\left( K\lambda^{\frac{\alpha}{2\alpha-1}} \right)
\end{align*}

\noindent with $K(\rho):= \frac{2\alpha -1}{\alpha} \rho^{-\frac{1-\alpha}{2\alpha-1}} + \log\E_{\theta_0}\left[\exp\left(\left(\frac{1-\alpha}{\alpha}\right)\rho L^{\alpha}(\theta_n)\right)\right]$ and $\frac{1-\alpha}{\alpha}\rho < \min(1, \varepsilon_{\beta} (8 \eta \alpha C_{\alpha}^2\Pi_{2,N})^{-1})$. We select $\rho = \overline{\rho}$ in the last inequality and use Proposition \ref{CONT_LAPLACEL} to bound the quantity $K(\overline{\rho})$.
\end{proof}
\begin{PROP}{(Control of the Lipschitz modulus of iterative kernels)}
\label{Lipschitz_control_SA}
Denote by $P_k$ and $P_{k,p}=P_{k} \circ \cdots \circ P_{p-1}$, $k, \ p \in \left\{0,\cdots, N-1\right\}$, $k \leq p$ the (Feller) transition kernel and the iterative kernels of the Markov chain $\theta$ defined by the scheme \eqref{RM}. Then for all Lipschitz function $f$ and for all $k, \ p \in \left\{0,\cdots, N-1\right\}$, $k\leq p$ the functions $P_{k}(f)$ are Lipschitz-continuous and one has
$$
[P_{k,p}(f)]_{1} : = \sup_{(\theta,\theta') \in (\R^{d})^2} \frac{\left|P_{k,p}(f)(\theta)-P_{k,p}(f)(\theta')\right|}{|\theta-\theta'|} \leq [f]_{1} \prod_{i=k}^{p-1}(1-2\underline{\lambda}\gamma_{i+1} + C_{H,\mu} \gamma^2_{i+1})^{\frac{1}{2}}
$$ 

\noindent where $[f]_{1}$ stands for the Lipschitz modulus of the function $f$ and $C_{H,\mu}:=2 C_{H}^2(1+\E[|U|^2])$.
\end{PROP}

\begin{proof} Using the Cauchy-Schwarz inequality, \A{HUA} then \A{HL}, for all $(\theta,\theta') \in (\R^{d})^2$, one has
\begin{align*}
|P_k(f)(\theta) - P_k(f)(\theta')| & \leq \E\left[\left|f(\theta-\gamma_{k+1} H(\theta,U_{k+1}))-f(\theta'-\gamma_{k+1} H(\theta',U_{k+1}))\right|\right] \\
& \leq [f]_1  \E\left[\left(\theta-\theta'-\gamma_{k+1} (H(\theta,U_{k+1})-H(\theta',U_{k+1}))\right)^2\right]^{\frac{1}{2}}\\
& \leq [f]_1 \left((\theta-\theta')^2-2\gamma_{k+1} \left\langle \theta-\theta', h(\theta)-h(\theta')\right\rangle + \gamma^2_{k+1} \E\left[|H(\theta,U_{k+1})-H(\theta',U_{k+1})|^2\right]\right)^{\frac{1}{2}} \\
& \leq [f]_1 (1-2\underline{\lambda}\gamma_{k+1} + 2 C_{H}^2(1+\E[|U|^2])\gamma^2_{k+1} )^{\frac{1}{2}} |\theta-\theta'|.
\end{align*}

\noindent A straightforward induction argument completes the proof.

\end{proof}
\begin{PROP}{(Control of the Laplace transform)}\label{CONT_LAPLACERM}
Denote by $\theta_{N}$ the value at step $N$ of the stochastic approximation algorithm \eqref{RM} with step sequence $\gamma:=(\gamma_{n})_{n\geq1}$ satisfying \eqref{STEP}. Assume that the innovations $(U_n)_{n\geq1}$ in \eqref{RM} satisfy \eqref{CONC_GAUSS} for some $\beta>0$. Let $f$ be a real-valued $1$-Lipschitz-continuous function defined on $\R^{d}$. Then, for all $\lambda \geq0$, for all $N\geq1$, for all $\alpha\in(\frac{1}{2},1]$, one has 
%$$
%\E_{\theta_0}[\exp(\lambda f(\theta_N))] \leq \exp(\lambda \E_{\theta_0}\left[f(\theta_{N})\right])\exp\left(\varphi(\gamma,\alpha,H)C^{\gamma}_N(\lambda^2\vee\lambda^{\frac{2\alpha}{2\alpha-1}})  \right) 
%$$
$$
\forall \lambda \geq 0, \ \ \E_{\theta_0}[\exp(\lambda f(\theta_N))]   \leq \exp\left(\E_{\theta_0}[\lambda f(\theta_N))] \right) \exp\left(\varphi_{\alpha}(\gamma,H, \theta_0) (C^{\gamma}_N\lambda^2 \vee C^{\gamma, \alpha}_N\lambda^{\frac{2\alpha}{2\alpha-1}})  \right)  
$$

\noindent with the two concentration rates  $C^{\gamma}_{N}:=\sum_{k=0}^{N-1}\gamma^2_{k+1}\frac{\Pi_{1,N}}{\Pi_{1,k}}$, with $\Pi_{1,N} := \prod_{k=0}^{N-1}(1-2\underline{\lambda}\gamma_{k+1} + C_{H,\mu} \gamma^2_{k+1})$ and $C^{\gamma, \alpha}_N := \sum_{k=0}^{N-1} \gamma^{\frac{2\alpha}{2\alpha-1}}_{k+1} (\frac{\Pi_{1,N}}{\Pi_{1,k}})^{\frac{2\alpha}{2\alpha-1}} ((k+1)\log^2(k+4))^{\frac{1-\alpha}{2\alpha-1}}$ for all $N\geq1$ and where $\varphi_{\alpha}(\gamma, H, \theta_0):= K_{4.1}2^{\frac{1-\alpha}{2\alpha-1}}   \frac{\beta C^2_\alpha}{4}\vee (\frac{\beta C^2_\alpha}{4})^{\frac{\alpha}{2\alpha-1}}\exp\left(\frac{1}{2\alpha-1} \sum_{k=0}^{N-1} \frac{1}{(k+1) \log^2(k+4)} \right)$.
%\noindent with $\varphi(\gamma,\alpha,H):= K_{4.1}(1+\gamma^2_1)^{\frac{1-\alpha}{2\alpha-1}} \frac{\beta C^2_\alpha}{4}\vee (\frac{\beta C^2_\alpha}{4})^{\frac{\alpha}{2\alpha-1}}\exp\left(\frac{C^{\gamma}_N}{2\alpha-1}\right)$.

If $\alpha=\frac{1}{2}$, then there exists two positive constants $\lambda_{4.1}$ and $\varphi_{1/2}(\gamma, H, \theta_0)$ such that
%\begin{align*}
%\E_{\theta_0}[\exp(\lambda f(\theta_N))] & \leq \exp(\lambda \E_{\theta_0}[f(\theta_N)]) \exp\left(C^{\gamma}_N\Psi\left(\lambda, N, \gamma\right) \right)
%\end{align*}
%
%\noindent where for all $\lambda < 2C_{1/2}\exp(-C^{\gamma}_N)(\beta (1+\gamma^2_1))^{-\frac12}\min(1, \varepsilon^{\frac12}_{\beta} (2 C_{1/2}^2\Pi_{2,N}(1/2))^{-\frac12})$ 
%$$
%\Psi(\lambda, N, \gamma):= \left(L^{\frac12}(\theta_0) + \underline{C} \sum_{p=0}^{N-1}\gamma^2_{p+1}\right)\tilde{\Pi}_{2,N} \frac{\beta C^2_{1/2}}{4} \lambda^2 + \left(\frac{1}{2} \sum_{p=0}^{N-1} \gamma^2_{p+1}\right) \log\left(\E\left[e^{\frac{\beta C_{1/2}^4 }{2}(1+\gamma^2_{1})\tilde{\Pi}_{2,N} \lambda^2 |U|^2}\right] \right)
%$$
%
%\noindent and $\Psi(\lambda, N, \gamma) = +\infty$ otherwise and $\tilde{\Pi}_{2,N}:=\Pi_{2,N}(1/2) \exp(2C^{\gamma}_N)$. 
\begin{align*}
\forall \lambda \in [0, \lambda_{4.1}/ \tilde{s}_{N}), \ \E_{\theta_0}[\exp(\lambda f(\theta_N))] & \leq \exp\left(\lambda \E_{\theta_0}[f(\theta_N)]\right) \exp\left( 2\varphi_{1/2}(\gamma, H, \theta_0) C^{\gamma}_N  \frac{(\lambda / \lambda_{4.1})^2}{1-(\lambda \tilde{s}_N / \lambda_{4.1})} \right)
\end{align*}

\noindent with $ \tilde{s}_N := \max_{0 \leq k \leq N-1} (k+1)^{1/2} \log(k+4) \gamma_{k+1} \left(\frac{\Pi_{1,N}}{\Pi_{1,k}}\right)^{\frac12} \exp(\sum_{p=0}^{N-1} \frac{1}{(p+1) \log^2(p+4)})$. 
%$S^{\gamma}_N := \sum_{k=0}^{N-1} (k+1) \log^2(k+4) \gamma^4_{k+1} \frac{\Pi^2_{1,N}}{\Pi^2_{1,k}}$  
%$$
%\varphi(N, \gamma):= e^{\sum_{k=0}^{N-1} \frac{1}{(k+1) \log^2(k+4)}} ((L^{\frac12}(\theta_0) + \underline{C} \sum_{p=0}^{N-1}\gamma^2_{p+1})\Pi_{2,N}\left(\frac12\right) \frac{\beta C^2_{1/2}}{4} + \frac{\beta \eta C_{1/2}^4 }{2}\Pi_{2,N}\left(\frac12\right) \E[|U|^2] + \sum_{p=0}^{N-1} \gamma^2_{p+1}).
%$$

\end{PROP}

\begin{proof} The proof relies on similar arguments as those used for the proof of Proposition \ref{CONT_LAPLACE}. For $\lambda \geq0$ and $k \in \left\{0,\cdots,N-1\right\}$, one has
$$
P_k(\exp(\lambda f))(\theta) \leq \exp\left(\lambda P_k(f) + \frac{\lambda^2}{4} \beta \gamma^2_{k+1}[f]^{2}_1 C^2_\alpha L^{1-\alpha}(\theta)\right)
$$

\noindent Taking expectation on both sides of the last inequality with $\theta=\theta_k$ and applying the H\"older inequality with conjugate exponents $(p_k,q_k)$ (to be fixed later on), one obtains
$$
\E_{\theta_0}[\exp(\lambda f(\theta_k))] \leq \E_{\theta_0}\left[\exp\left(\lambda p_k P_k(f)(\theta_k)\right)\right]^{\frac{1}{p_k}} \E_{\theta_0}\left[\exp\left(q_k\frac{\lambda^2}{4} \beta \gamma^2_{k+1}[f]^{2}_1 C^2_\alpha L^{1-\alpha}(\theta_k)\right)\right]^{\frac{1}{q_k}}
$$

\noindent and applying the last inequality to $f:=P_{k+1,N}(f)$ yields
\begin{equation}
\E_{\theta_0}[\exp(\lambda P_{k+1,N}(f)(\theta_k))] \leq \E_{\theta_0}\left[\exp\left(\lambda p_k P_{k,N}(f)(\theta_k)\right)\right]^{\frac{1}{p_k}} \E_{\theta_0}\left[\exp\left(q_k\frac{\lambda^2}{4} \beta \gamma^2_{k+1}[P_{k+1,N}(f)]^{2}_1 C^2_\alpha L^{1-\alpha}(\theta_k)\right)\right]^{\frac{1}{q_k}}.
\label{FIRSTREC}
\end{equation}

%\noindent We set $p_k=1+\gamma^2_{k+1}[P_{k+1,N}(f)]^2_1$, $q_k=\frac{1}{\gamma^2_{k+1}[P_{k+1,N}(f)]^2_1}(1+\gamma^2_{k+1}[P_{k+1,N}(f)]^2_1)$ and use Corollary \ref{CONTLAPLACELBIS} to obtain for $\alpha \in(\frac12,1]$

\noindent We use Corollary \ref{CONTLAPLACELBIS} to obtain for $\alpha \in(\frac12,1]$
\begin{align*}
\E_{\theta_0}\left[\exp\left(q_k\frac{\lambda^2}{4} \beta \gamma^2_{k+1}[P_{k+1,N}(f)]^{2}_1 C^2_\alpha L^{1-\alpha}(\theta_k)\right)\right]^{\frac{1}{q_k}}  & \leq \exp\left(K_{4.1}  \frac{\beta C^2_\alpha}{4}\vee \left(\frac{\beta C^2_\alpha}{4}\right)^{\frac{\alpha}{2\alpha-1}} \right. \\ 
 & \left. \times (\gamma^2_{k+1} [P_{k+1,N}(f)]^2_1\lambda^2 \vee \gamma^{\frac{2\alpha}{2\alpha-1}}_{k+1} [P_{k+1,N}(f)]^{\frac{2\alpha}{2\alpha-1}}_1 q^{\frac{1-\alpha}{2\alpha-1}}_{k}\lambda^{\frac{2\alpha}{2\alpha-1}}) \right) \\
& := f_k(\lambda) 
\end{align*}
%\begin{align*}
%\E_{\theta_0}\left[\exp\left(q_k\frac{\lambda^2}{4} \beta \gamma^2_{k+1}[P_{k+1,N}(f)]^{2}_1 C^2_\alpha L^{1-\alpha}(\theta_k)\right)\right]^{\frac{1}{q_k}}  & \leq \exp\left(K_{5.1}  \frac{\beta C^2_\alpha}{4}\vee \left(\frac{\beta C^2_\alpha}{4}\right)^{\frac{\alpha}{2\alpha-1}} (1+\gamma^2_{k+1}[P_{k+1,N}(f)]^2_1)^{\frac{1-\alpha}{2\alpha-1}} \right. \\
%& \left. \ \ \times \gamma^2_{k+1} [P_{k+1,N}(f)]^2_1(\lambda^2 \vee \lambda^{\frac{2\alpha}{2\alpha-1}}) \right) \\
%& \leq \exp(f_k(\lambda)) 
%\end{align*}

\noindent where we temporarily set $f_k(\lambda):= \exp\left(K_{4.1}  \frac{\beta C^2_\alpha}{4}\vee \left(\frac{\beta C^2_\alpha}{4}\right)^{\frac{\alpha}{2\alpha-1}} (\gamma^2_{k+1} [P_{k+1,N}(f)]^2_1\lambda^2 \vee\gamma^{\frac{2\alpha}{2\alpha-1}}_{k+1} [P_{k+1,N}(f)]^{\frac{2\alpha}{2\alpha-1}}_1 q^{\frac{1-\alpha}{2\alpha-1}}_{k}\lambda^{\frac{2\alpha}{2\alpha-1}}) \right)$ for all $\lambda \geq 0$ in the interests of simplifying notation and analysis. Now,  an elementary induction argument leads to
\begin{align}
\E_{\theta_0}[\exp(\lambda f(\theta_N))] & = \E_{\theta_0}[\exp(\lambda P_{N,N}f(\theta_N))] \nonumber \\
& \leq \E_{\theta_0}[\exp(\lambda \prod_{k=0}^{N-1} p_{k} P_{0,N}(f)(\theta_0))]^{\frac{1}{\prod_{k=0}^{N-1} p_{k}}}  \prod_{k=0}^{N-1}f_{N-1-k}\left(\lambda \prod_{i=1}^k p_{N-i}\right)^{\frac{1}{\prod_{i=1}^k p_{N-i}}} \label{REC_ALGOSTO}
%& \leq \exp(\lambda \E_{\theta_0}[f(\theta_N)]) \exp\left(K_{5.1} (1+\gamma^2_1)^{\frac{1-\alpha}{2\alpha-1}}  \frac{\beta C^2_\alpha}{4}\vee (\frac{\beta C^2_\alpha}{4})^{\frac{\alpha}{2\alpha-1}}e^{\frac{1}{2\alpha-1}\sum_{i=0}^{N-1} \gamma^2_{i+1} [P_{i+1,N}(f)]^2_1} \right. \\ 
%&  \left. \ \ \ \ \ \ \ \ \  \times \sum_{k=0}^{N-1}\gamma^2_{k+1} [P_{k+1,N}(f)]^2_1(\lambda^2 \vee\lambda^{\frac{2\alpha}{2\alpha-1}}) \right)  \\
% & \leq  \exp(\lambda \E_{\theta_0}[f(\theta_N)]) \exp\left(\varphi(N,\alpha,H)C^{\gamma}_N(\lambda^2 \vee \lambda^{\frac{2\alpha}{2\alpha-1}})  \right) 
\end{align}

We select $p_k:= 1 + \frac{1}{(k+1) \log^2(k+4)}$, $q_k= (1 + \frac{1}{(k+1) \log^2(k+4)}) (k+1) \log^2(k+4) \leq 2 (k+1) \log^2(k+4)$, $k=0, \cdots, N-1$ so that $\prod_{k=0}^{N-1} p_k$ converges and more precisely we have $\prod_{k=0}^{N-1} p_k < \exp(\sum_{k=0}^{N-1} \frac{1}{(k+1) \log^2(k+4)}) < \infty$. 
%noticing that $\forall k \in \left\{0,\cdots,N-1\right\}$, $\prod_{i=0}^{k} p_{N-i} \leq \exp(\sum_{i=0}^{N-1} \gamma^2_{i+1} [P_{i+1,N}(f)]^2_1)$,

%Now, noticing that $\forall k \in \left\{0,\cdots,N-1\right\}$, $\prod_{i=0}^{k} p_{N-i} \leq \exp(\sum_{i=0}^{N-1} \gamma^2_{i+1} [P_{i+1,N}(f)]^2_1)$, an elementary induction argument leads to
%\begin{align*}
%\E_{\theta_0}[\exp(\lambda f(\theta_N))] & \leq \E_{\theta_0}[\exp(\lambda P_{N,N}f(\theta_N))] \\
%& \leq \E_{\theta_0}[\exp(\lambda \prod_{k=0}^{N-1} p_{k} P_{0,N}(f)(\theta_0))]^{\frac{1}{\prod_{k=0}^{N-1} p_{k}}}  \prod_{k=0}^{N-1}f_{N-1-k}\left(\lambda \prod_{i=1}^k p_{N-i}\right)^{\frac{1}{q_{N-1-k}\prod_{i=1}^k p_{N-i}}} \\
%& \leq \exp(\lambda \E_{\theta_0}[f(\theta_N)]) \exp\left(K_{5.1} (1+\gamma^2_1)^{\frac{1-\alpha}{2\alpha-1}}  \frac{\beta C^2_\alpha}{4}\vee (\frac{\beta C^2_\alpha}{4})^{\frac{\alpha}{2\alpha-1}}e^{\frac{1}{2\alpha-1}\sum_{i=0}^{N-1} \gamma^2_{i+1} [P_{i+1,N}(f)]^2_1} \right. \\ 
%&  \left. \ \ \ \ \ \ \ \ \  \times \sum_{k=0}^{N-1}\gamma^2_{k+1} [P_{k+1,N}(f)]^2_1(\lambda^2 \vee\lambda^{\frac{2\alpha}{2\alpha-1}}) \right)  \\
%& \leq  \exp(\lambda \E_{\theta_0}[f(\theta_N)]) \exp\left(\varphi(N,\alpha,H)C^{\gamma}_N(\lambda^2 \vee \lambda^{\frac{2\alpha}{2\alpha-1}})  \right) 
%\end{align*}

\noindent We introduce for sake of simplicity $\varphi_{\alpha}(\gamma,H, \theta_0):= K_{4.1}2^{\frac{1-\alpha}{2\alpha-1}}   \frac{\beta C^2_\alpha}{4}\vee (\frac{\beta C^2_\alpha}{4})^{\frac{\alpha}{2\alpha-1}}\exp\left(\frac{1}{2\alpha-1} \sum_{k=0}^{N-1} \frac{1}{(k+1) \log^2(k+4)} \right)$. Now, using Proposition \ref{Lipschitz_control_SA} and Corollary \ref{CONTLAPLACELBIS}, we easily derive from \eqref{REC_ALGOSTO} 
\begin{align*}
\forall \lambda \geq 0, \ \ \E_{\theta_0}[\exp(\lambda f(\theta_N))]  & \leq \exp\left(\E_{\theta_0}[\lambda f(\theta_N))] \right) \exp\left(\varphi_{\alpha}(\gamma,H, \theta_0) (C^{\gamma}_N\lambda^2 \vee C^{\gamma, \alpha}_N\lambda^{\frac{2\alpha}{2\alpha-1}})  \right)  
\end{align*}

\noindent with $C^{\gamma, \alpha}_N := \sum_{k=0}^{N-1} \gamma^{\frac{2\alpha}{2\alpha-1}}_{k+1} (\frac{\Pi_{1,N}}{\Pi_{1,k}})^{\frac{2\alpha}{2\alpha-1}} ((k+1)\log^2(k+4))^{\frac{1-\alpha}{2\alpha-1}}$.
%\noindent with $\varphi(N,\alpha,H):= K_{5.1}(1+\gamma^2_1)^{\frac{1-\alpha}{2\alpha-1}}   \frac{\beta C^2_\alpha}{4}\vee (\frac{\beta C^2_\alpha}{4})^{\frac{\alpha}{2\alpha-1}}\exp\left(\frac{C^{\gamma}_N}{2\alpha-1} \right)$ where we used Proposition \ref{Lipschitz_control_SA} for the last inequality.

For $\alpha=\frac12$, we start from \eqref{FIRSTREC}. First, we use the control obtained in Proposition \ref{CONT_LAPLACEL} to derive
\begin{align*}
\E_{\theta_0}\left[\exp\left(q_k\frac{\lambda^2}{4} \beta \gamma^2_{k+1}[P_{k+1,N}(f)]^{2}_1 C^2_{1/2} L^{\frac12}(\theta_k)\right)\right]^{\frac{1}{q_k}} & \leq \exp \left(\left(L^{\frac12}(\theta_0) + \underline{C} \sum_{p=0}^{N-1}\gamma^2_{p+1}\right)\Pi_{2,N}\left(1/2\right) \frac{\beta C^2_{1/2}}{4}  \gamma^2_{k+1}[P_{k+1,N}(f)]^{2}_1 \lambda^2 \right. \\
& \left.  + \frac{1}{q_k} \left(\frac{1}{2} \sum_{p=0}^{N-1} \gamma^2_{p+1}\right) \right. \\
& \left. \times \ \log\E\left[\exp\left(\beta \eta C_{1/2}^4 \Pi_{2,N}\left(1/2\right) q_k \gamma^2_{k+1} [P_{k+1,N}(f)]^{2}_1 \lambda^2|U|^2\right)\right]  \right).
\end{align*}
%\begin{align*}
%\E_{\theta_0}\left[\exp\left(q_k\frac{\lambda^2}{4} \beta \gamma^2_{k+1}[P_{k+1,N}(f)]^{2}_1 C^2_{1/2} L^{\frac12}(\theta_k)\right)\right]^{\frac{1}{q_k}} & \leq \exp \left(\left(L^{\frac12}(\theta_0) + \underline{C} \sum_{p=0}^{N-1}\gamma^2_{p+1}\right)\Pi_{2,N}\left(1/2\right) \frac{\beta C^2_{1/2}}{4}  \gamma^2_{k+1}[P_{k+1,N}(f)]^{2}_1 \lambda^2 \right. \\
%& \left.  + \frac{\gamma^2_{k+1}[P_{k+1,N}(f)]^{2}_1}{1+\gamma^2_{k+1}[P_{k+1,N}(f)]^2_1} \left(\frac{1}{2} \sum_{p=0}^{N-1} \gamma^2_{p+1}\right) \right. \\
%& \left. \times \ \log\E\left[\exp\left(\frac{\beta C_{1/2}^4 }{2}(1+\gamma^2_{k+1}[P_{k+1,N}(f)]^{2}_1) \Pi_{2,N}\left(1/2\right) \lambda^2|U|^2\right)\right]  \right)
%\end{align*}
%\noindent which is valid for $\lambda < 2(C^2_{1/2}\beta (1+\gamma^2_1))^{-\frac12}\min(1, \varepsilon^{\frac12}_{\beta} (2 C_{1/2}^2\Pi_{2,N}\left(1/2\right))^{-\frac12})$.
\noindent  To simplify the latter bound, that is to obtain an explicit and computable formula for the second term appearing in the right hand side, we will need the following lemma:
\begin{LEMME}\label{CONTLEMMA} 
For all $\lambda \in [0,\lambda_{4.1}/s^{1/2}_N)$, one has
\begin{align*}
\log\E\left[\exp\left(\beta \eta C_{1/2}^4 \Pi_{2,N}\left(1/2\right) q_k \gamma^2_{k+1} [P_{k+1,N}(f)]^{2}_1 \lambda^2|U|^2\right)\right] & \leq  \beta \eta C_{1/2}^4 \Pi_{2,N}\left(1/2\right) \E[|U|^2] q_k \gamma^2_{k+1} [P_{k+1,N}(f)]^{2}_1 \lambda^2  \\ 
& + 2 q_k  \gamma^2_{k+1} [P_{k+1,N}(f)]^{2}_1  \frac{(\lambda / \lambda_{4.1})^2}{1-(\lambda s^{1/2}_N / \lambda_{4.1})},
\end{align*}

\noindent with $s_N:=\max_{0 \leq k \leq N-1 }q_k \gamma^2_{k+1} \frac{\Pi_{1,N}}{\Pi_{1,k}}$ and $\lambda_{4.1}$ satisfies $\E[\exp(\lambda_{4.1}^2 \beta \eta C_{1/2}^4 \Pi_{2,N}\left(1/2\right)|U|^2)]\leq 2$.
\end{LEMME}

\begin{proof} The proof is similar to the proof of Corollary \ref{CONTLAPLACETER}. By definition of $\lambda_{4.1}$, $ \lambda^{2p}_{4.1} (\beta \eta C^{4}_{1/2} \Pi_{2,N}(1/2)/2)^p \E[|U|^{2p}] \leq 2 p!$, $\forall p\geq 1$. Hence, setting $C_1:= \beta \eta C_{1/2}^4 \Pi_{2,N}\left(1/2\right)$ we easily deduce,
 \begin{align*}
\log\E\left[e^{\lambda^2 C_1  q_k \gamma^2_{k+1} [P_{k+1,N}(f)]^{2}_1|U|^2}\right] & - \lambda^2 C_1  q_k \gamma^2_{k+1} [P_{k+1,N}(f)]^{2}_1 \E[|U|^2]  \\
& \leq \sum_{p\geq2} \frac{\lambda^{2p}C^p_1 ( q_k \gamma^2_{k+1} [P_{k+1,N}(f)]^{2}_1)^p \E[|U|^{2p}]}{p!} \\
& \leq 2 \sum_{p\geq2} \left(\frac{\lambda^{2} q_k \gamma^2_{k+1} [P_{k+1,N}(f)]^{2}_1}{\lambda^2_{4.1}}\right)^p \\
& \leq \left\{ \begin{array}{l}
 2 q_k \gamma^2_{k+1} [P_{k+1,N}(f)]^{2}_1\frac{(\lambda /\lambda_{4.1})^2}{1- (\lambda s^{1/2}_N/\lambda_{4.1})}
, \ \mbox{if } \ \lambda < \lambda_{4.1}/s^{1/2}_N, \vspace*{0.2cm} \\
 + \infty ,\  \ \mbox{otherwise}.
\end{array}
\right.
\end{align*} 

\noindent This completes the proof.
\end{proof}

Using the previous lemma, we obtain for all $\lambda \in [0,\lambda_{4.1}/s^{1/2}_N)$, 
\begin{align*}
\E_{\theta_0}\left[\exp\left(q_k\frac{\lambda^2}{4} \beta \gamma^2_{k+1}[P_{k+1,N}(f)]^{2}_1 C^2_{1/2} L^{\frac12}(\theta_k)\right)\right]^{\frac{1}{q_k}} & \leq \exp \left(\Psi(N,\gamma, \theta_0) \gamma^2_{k+1}[P_{k+1,N}(f)]^{2}_1 \lambda^2 \right. \\
& \left.  + \left(\sum_{p=0}^{N-1} \gamma^2_{p+1}\right)   \gamma^2_{k+1} [P_{k+1,N}(f)]^{2}_1  \frac{(\lambda / \lambda_{4.1})^2}{1-(\lambda s^{1/2}_N / \lambda_{4.1})}\right),
\end{align*}

\noindent where we introduced the notation $\Psi(N,\gamma, \theta_0):= \left(L^{\frac12}(\theta_0) + \underline{C} \sum_{p=0}^{N-1}\gamma^2_{p+1}\right)\Pi_{2,N}\left(1/2\right) \frac{\beta C^2_{1/2}}{4} + \beta \eta C_{1/2}^4 \Pi_{2,N}\left(1/2\right) \E[|U|^2] $.

Now, as for $\alpha \in (\frac12,1]$, an induction argument in the spirit of \eqref{REC_ALGOSTO} yields for all $\lambda \in [0, \lambda_{4.1}/ \tilde{s}_N)$
\begin{align*}
\E_{\theta_0}[\exp(\lambda f(\theta_N))] & \leq \exp\left(\lambda \E_{\theta_0}[f(\theta_N)]\right) \exp\left(C^{\gamma}_N\Psi\left(N, \gamma, \theta_0\right) e^{\sum_{k=0}^{N-1} \frac{1}{(k+1) \log^2(k+4)}} \lambda^2 \right.  \\
 &  \left. + e^{\sum_{k=0}^{N-1} \frac{1}{(k+1) \log^2(k+4)}} \left(\sum_{p=0}^{N-1} \gamma^2_{p+1}\right)  C^{\gamma}_N \frac{(\lambda / \lambda_{4.1})^2}{1-(\lambda \tilde{s}_N / \lambda_{4.1})} \right), \\
 & \leq \exp\left(\lambda \E_{\theta_0}[f(\theta_N)]\right) \exp\left(2\varphi_{1/2}(\gamma, H, \theta_0) C^{\gamma}_N \left(  (\lambda/\lambda_{4.1})^2 \vee  \frac{(\lambda / \lambda_{4.1})^2}{1-(\lambda \tilde{s}_N / \lambda_{4.1})} \right) \right) \\
 & = \exp\left(\lambda \E_{\theta_0}[f(\theta_N)]\right) \exp\left(2\varphi_{1/2}(\gamma, H, \theta_0)C^{\gamma}_N  \frac{(\lambda / \lambda_{4.1})^2}{1-(\lambda \tilde{s}_N / \lambda_{4.1})}  \right)
\end{align*}

\noindent with $\varphi_{1/2}(\gamma, H, \theta_0):= \exp(\sum_{k=0}^{N-1} \frac{1}{(k+1) \log^2(k+4)}) (\lambda^2_{4.1} \Psi(N,\gamma, \theta_0) + \sum_{p=0}^{N-1} \gamma^2_{p+1}) $, $ \tilde{s}_N := s^{1/2}_N \exp(\sum_{k=0}^{N-1} \frac{1}{(k+1) \log^2(k+4)})$, 
%$S^{\gamma}_N := \sum_{k=0}^{N-1} (k+1) \log^2(k+4) \gamma^4_{k+1} \frac{\Pi^2_{1,N}}{\Pi^2_{1,k}}$, 
and where we used again $\prod_{k=0}^{N-1} p_k < \exp(\sum_{k=0}^{N-1} \frac{1}{(k+1) \log^2(k+4)})$.

%The latter inequality simplifies and writes as
%\begin{align*}
%\E_{\theta_0}\left[\exp\left(q_k\frac{\lambda^2}{4} \beta \gamma^2_{k+1}[P_{k+1,N}(f)]^{2}_1 C^2_\alpha L^{\frac12}(\theta_k)\right)\right]^{\frac{1}{q_k}} & \leq \exp\left(\gamma^2_{k+1}[P_{k+1,N}(f)]^{2}_1 \Psi(\lambda, N,\gamma) \right)
%\end{align*}
% 
%\noindent with 
%$$
%\Psi(\lambda,N,\gamma):= \left(L^{\frac12}(\theta_0) + \underline{C} \sum_{p=0}^{N-1}\gamma^2_{p+1}\right)\Pi_{2,N}\left(1/2\right)\frac{\beta C^2_{1/2}}{4} \lambda^2 + \left(\frac{1}{2} \sum_{p=0}^{N-1} \gamma^2_{p+1}\right) \log\left(\E\left[e^{\frac{\beta C_{1/2}^4 }{2}(1+\gamma^2_{1})\Pi_{2,N}\left(1/2\right) \lambda^2 |U|^2}\right] \right)
%$$
%
%\noindent for all $\lambda < 2(C^2_{1/2}\beta (1+\gamma^2_1))^{-\frac12}\min(1, \varepsilon^{\frac12}_{\beta} (2 C_{1/2}^2\Pi_{2,N}\left(1/2\right))^{-\frac12})$ and $\Psi(\lambda, N,\gamma) = +\infty$ otherwise.
%
%Now, just as in the case of $\alpha \in (\frac12,1]$, using that $p_k, \ q_k\geq 1$ for all $k\in\left\{0,\cdots,N-1\right\}$ , an elementary induction argument yields
%\begin{align*}
%\E_{\theta_0}[\exp(\lambda f(\theta_N))] & \leq \exp(\lambda \E_{\theta_0}[f(\theta_N)]) \exp\left(C^{\gamma}_N\Psi\left(\lambda \exp(C^{\gamma}_N),N, \gamma\right) \right).
%\end{align*}

\end{proof}
\medskip

In contrast to Euler like schemes, a bias appears in the non-asymptotic deviation bound for the stochastic approximation algorithm. Consequently, it is crucial to have a control on it. At step $n$ of the algorithm, it is given by $\delta_{n}:=\E[\left|\theta_{n}-\theta^{*}\right|]$. Under the current assumptions \A{HL}, \A{HLS}$_\alpha$, \A{HUA}, we have the following proposition.
\begin{PROP}[Control of the bias]
\label{CTR_BIAS}
For all $n \geq 1$, we have 
$$
\delta_{n} \leq \exp\left(-\underline{\lambda} \Gamma_{1,n} + C_{\alpha,\mu}\Gamma_{2,n}\right) \left|\theta_0 - \theta^{*}\right| + (2C_{\alpha,\mu})^{\frac{1}{2}} \left(\sum_{k=0}^{n-1} \gamma^2_{k+1} \exp\left(-2\underline{\lambda} (\Gamma_{1,n}-\Gamma_{1,k+1}) + 2C_{\alpha,\mu} (\Gamma_{2,n}-\Gamma_{2,k+1})\right)\right)^{\frac{1}{2}},
$$

\noindent where $\Gamma_{1,n} :=\sum_{k=1}^{n} \gamma_{k}$, $\Gamma_{2,n} :=\sum_{k=1}^{n} \gamma^2_{k}$, $C_{\alpha,\mu}:=\underline{\lambda}^2 /2+2 C_{\alpha}K\E[|U|^2]$ with $K>0$.
\end{PROP}

\noindent \textit{Proof.}
With the notations of Section \ref{RM_SEC}, we define for all $n\ge 1, \ \Delta M_n:= h(\theta_{n-1})- H(\theta_{n-1},U_n)= \E[\left. H(\theta_{n-1},U_n) \right| \mathcal{F}_{n-1}]- H(\theta_n,U_n)$. Recalling that $(U_n)_{n \geq 1} $ is a sequence of i.i.d. random variables we have that $(\Delta M_n)_{n\geq1}$ is a sequence of martingale increments w.r.t. the natural filtration $\mathcal{F}:=(\F_n:=\sigma(\theta_0,U_1,\cdots,U_n,); n\geq1)$.

From the dynamic \eqref{RM}, we now write for all $n\geq0$,
\begin{eqnarray*}
z_{n+1} & := & \theta_{n+1}-\theta^*=\theta_n-\theta^*- \gamma_{n+1}\left\{h(\theta_n)-\Delta M_{n+1}\right\}\\
 & = & \theta_{n}-\theta^*-\gamma_{n+1}  \int_0^1d\lambda Dh(\theta^*+\lambda  (\theta_n-\theta^*) )(\theta_n-\theta^*) +\gamma_{n+1}\Delta M_{n+1},
\end{eqnarray*}

\noindent where we used that $h(\theta^*)=0 $ for the last equality. Setting $J_n:=\int_0^1d\lambda Dh(\theta^*+\lambda (\theta_n-\theta^*) )$, we obtain $z_{n+1} = (I-\gamma_{n+1} J_n)z_n + \gamma_{n+1} \Delta M_{n+1}$ which yields 
\begin{align*}
\E_{\theta_0}[|z_{n+1}|^2] & = \E_{\theta_0}[|I-\gamma_{n+1} J_n|^2 |z_n|^2] + 2\gamma_{n+1} \E_{\theta_0}[(I-\gamma_{n+1} J_n) \Delta M_{n+1}] + \gamma^2_{n+1} \E_{\theta_0}[|\Delta M_{n+1}|^2] \\
& = \E_{\theta_0}[|I-\gamma_{n+1} J_n|^2 |z_n|^2] + \gamma^2_{n+1} \E_{\theta_0}[|\Delta M_{n+1}|^2].
\end{align*}

 From assumption \A{HLS}$_\alpha$, we deduce that $\forall (\theta,u) \in \R^d \times \R^q$, $|h(\theta) - H(\theta,u)|^2 \leq 2 C^2_\alpha L^{1-\alpha}(\theta) (\E[|U|^2]+|u|^2)$ which combined with the independence of $\theta_n$ and $U_{n+1}$ clearly implies
$$
\E_{\theta_0}[|h(\theta_n) - H(\theta_n,U_{n+1})|^2] \leq 4 C^2_\alpha \E[|U|^2] \E_{\theta_0}[L^{1-\alpha}(\theta_n)].
$$

Now, let us notice that $L$ has sub-quadratic growth so that there exists a constant $K>0$ such that
$$
\E_{\theta_0}\left[|\Delta M_{n+1}|^2\right]  = \E_{\theta_0}\left[|h(\theta_n)-H(\theta_n,U_{n+1})|^2\right]   \leq 4C^2_{\alpha}\E[|U|^2] \E_{\theta_0}\left[L^{1-\alpha}(\theta_n)\right] \leq 4 K C^2_{\alpha}\E[|U|^2](1+ \E_{\theta_0}[|z_n|^2]),
$$

\noindent which provides the following bound
\begin{align*}
\E_{\theta_0}[|z_{n+1}|^2] & \leq (1-\underline{\lambda}\gamma_{n+1})^{2}\E_{\theta_0}[|z_n|^2] + 4 K C^2_{\alpha} \E[|U|^2] \gamma^2_{n+1} \E_{\theta_0}[|z_n|^2]  \\
& \leq \left(1-2\underline{\lambda}\gamma_{n+1}+ 2 C_{\alpha,\mu}\gamma_{n+1}^2\right) \E_{\theta_0}[|z_n|^2]+ 2 C_{\alpha,\mu}\gamma^2_{n+1}.
\end{align*}

\noindent Temporarily setting $\tilde{\Pi}_{n}=\prod_{p=0}^{n-1}(1-2\underline{\lambda}\gamma_{p+1}+2C_{\alpha,\mu}\gamma_{p+1}^2)$, a straightforward induction argument provides
\begin{align*}
\E_{\theta_0}[|z_{n}|^2] & \leq \tilde{\Pi}_{n} |\theta_0-\theta^{*}|^2 + 2C_{\alpha,\mu} \sum_{k=0}^{n-1} \gamma^2_{k+1} \tilde{\Pi}_n \tilde{\Pi}^{-1}_{k+1} \\
& \leq e^{-2\underline{\lambda} \Gamma_{1,n} + 2C_{\alpha,\mu} \Gamma_{2,n} } |\theta_0-\theta^{*}|^2  + 2C_{\alpha,\mu} \sum_{k=0}^{n-1} \gamma^2_{k+1} e^{-2\underline{\lambda} (\Gamma_{1,n}-\Gamma_{1,k+1}) + 2C_{\alpha,\mu} (\Gamma_{2,n}-\Gamma_{2,k+1})}
\end{align*}

\noindent where we used the elementary inequality, $1+x \leq \exp(x)$,  $x \in \R$. This completes the proof.

\subsection{Proof of Theorem \ref{THM_TE_SA_AV}} \label{SEC:PROOF:SAAV}

\begin{PROP}{(Control of the Lipschitz modulus of iterative kernels)}
\label{Lipschitz_control_SA_AV}
Denote by $K_k$ and $K_{k,p}=K_{k} \circ \cdots \circ K_{p-1}$, $k, \ p \in \left\{0,\cdots, N-1\right\}$, $k \leq p$ the (Feller) transition kernel and the iterative kernels of the Markov chain $z=(\bar{\theta},\theta)$ defined by the scheme \eqref{RM}, \eqref{RM_AV}. Let $f:\R^d \rightarrow \R$ be a $1$-Lipschitz function. Then for all $k, \ p \in \left\{0,\cdots, N-1\right\}$, $k\leq p$ the functions $K_{k,p}(f):z\mapsto \E[\left. f(\bar{\theta}_{p+1}) \right| z_k=z]$ are Lipschitz-continuous. In particular, for all $(z,z') \in (\R^d \times \R^d)^2$, one has
$$
| K_{k,p}(f)(z)-K_{k,p}(f)(z') | \leq \frac{k+1}{p+1} |z_1-z'_1|+ \frac{1}{p+1}\sum_{j=k+1}^{p} \left(\frac{\Pi_{1,j}}{\Pi_{1, k}}\right)^{\frac12} |z_2-z'_2| 
$$ 

\noindent where $\Pi_{1,p}=\prod_{k=0}^{p-1}(1-2\underline{\lambda}\gamma_{k+1} + C_{H,\mu} \gamma^2_{k+1})$.
\end{PROP}

\begin{proof} Let $(z,z') \in (\R^d \times \R^d)^2$. We denote by $z^{k,z}_{p,1}=\bar{\theta}^{k,z}_{p+1}$ and $z^{k,z}_{p,2}=\theta^{k,z}_{p}$ the values at step $p$ of the two components of the stochastic approximation algorithm $(z_n)_{n\geq0}$ starting at point $z$ at step $k$. Using \eqref{RM_AV} and a straightforward induction, one easily derives
$$
\bar{\theta}^{k,z}_{p+1} = \frac{k+1}{p+1} z_1 + \frac{1}{p+1} \sum_{j=k+1}^{p} \theta^{k,z}_{j},
$$
 
\noindent so that taking conditional expectation in the previous equality and using Proposition \ref{Lipschitz_control_SA},we obtain
\begin{align*}
|K_{k,p}(f)(z)-K_{k,p}(f)(z')| & = | \E[f(\bar{\theta}^{k,z}_{p+1})] - \E[f(\bar{\theta}^{k,z'}_{p+1})] | \leq \E[|\bar{\theta}^{k,z}_{p+1} -\bar{\theta}^{k,z'}_{p+1} |] \\
& \leq \frac{k+1}{p+1} |z_1 - z'_1| + \frac{1}{p+1}\sum_{j=k+1}^{p} \E[|\theta^{k,z}_{j} - \theta^{k,z'}_j |] \\
& \leq  \frac{k+1}{p+1} |z_1-z'_1|+ \frac{1}{p+1}\sum_{j=k+1}^{p} \left(\frac{\Pi_{1,j}}{\Pi_{1,k}}\right)^{\frac12} |z_2-z'_2| \end{align*}

\end{proof}

Let $k \in \left\{0,\cdots, N-1\right\}$ and $f$ be a real-valued $1$-Lipschitz function defined on $\R^d$. Using that the law of the innovations of the scheme satisfies \eqref{CONC_GAUSS},  for all $\lambda \geq0$,  one has
\begin{align*}
\E\left[ \left.  \exp(\lambda K_{k,N-1}f(z_k))\right| z_{k-1}=z \right] & = \E\left[ \left. \exp\left(\lambda  K_{k,N-1}f(\frac{k}{k+1}\bar{\theta}_k +\frac{1}{k+1}\theta_k, \theta_k)\right)  \right| (\bar{\theta}_k, \theta_{k-1})=(z_1,z_2)\right] \\
& \leq \exp(\lambda K_{k-1,N-1}(f)(z)) \exp(\lambda^2 \frac{\beta}{4}[g]^2_1)
\end{align*}

\noindent where $g:u \mapsto K_{k,N-1}(f)\left( \frac{k}{k+1}z_1 + \frac{1}{k+1}z_2 - \frac{\gamma_k}{k+1} H(z_2,u), z_2 -\gamma_k H(z_2,u) \right)$. Combining Proposition \ref{Lipschitz_control_SA_AV} and \A{HLS}$_\alpha$, one easily obtains 
$$
[g]_1 \leq  C_\alpha L^{\frac{1-\alpha}{2}}(z_2) \gamma_k \left( \frac{1}{N} + \frac{1}{N}\sum_{j=k+1}^{N-1} \left(\frac{\Pi_{1,j}}{\Pi_{1,k}}\right)^{\frac12} \right) 
$$

\noindent so we deduce that 
$$
\E\left[ \left.  \exp(\lambda K_{k,N-1}f(z_k))\right| z_{k-1}\right] \leq \exp(\lambda K_{k-1,N-1}(f)(z_{k-1})) \exp\left(\lambda^2 \frac{\beta}{4} C^2_\alpha L^{1-\alpha}(z_{k-1}) \tilde{\gamma}^2_{k,N}\right)
$$

\noindent where we introduced the notation $\tilde{\gamma}_{k,N} := \frac{\gamma_k}{N} \left(1+ \sum_{j=k+1}^{N-1} \left(\Pi_{1, j}/\Pi_{1,k}\right)^{\frac12} \right)$. Hence, taking expectation in the previous inequality and using the H\"{o}lder inequality with conjugate exponents $(p_k,q_k)$, one clearly gets
$$
\E_{\theta_0}\left[ \exp(\lambda K_{k,N-1}(f)(z_k))\right] \leq \E_{\theta_0}[\exp(\lambda p_k K_{k-1,N-1}(f)(z_{k-1}))]^{\frac{1}{p_k}} \E_{\theta_0}\left[\exp\left(\lambda^2 \frac{\beta}{4} C^2_\alpha q_k L^{1-\alpha}(\theta_{k-1}) \tilde{\gamma}^2_{k,N}\right)\right]^{\frac{1}{q_k}}
$$

Similarly to the proof of Proposition \ref{CONT_LAPLACERM}, we set $p_k=1 + \frac{1}{(k+1)\log^2(k+4)}$, $q_k=(1 + \frac{1}{(k+1)\log^2(k+4)}) (k+1) \log^2(k+4) \leq 2(k+1) \log^2(k+4)$ and use Corollary \ref{CONTLAPLACELBIS} to obtain for $\alpha \in(\frac12,1]$
\begin{align*}
\E_{\theta_0}\left[\exp\left(\lambda^2 \frac{\beta}{4} C^2_\alpha q_k L^{1-\alpha}(\theta_{k-1}) \tilde{\gamma}^2_{k,N}\right)\right]^{\frac{1}{q_k}}  & \leq \exp\left(K_{4.1} 2^{\frac{1-\alpha}{2\alpha-1}} \frac{\beta C^2_\alpha}{4}\vee \left(\frac{\beta C^2_\alpha}{4}\right)^{\frac{\alpha}{2\alpha-1}} (  \tilde{\gamma}^2_{k,N} \lambda^2 \vee  \tilde{\gamma}^{\frac{2\alpha}{2\alpha-1}}_{k,N} q^{\frac{1-\alpha}{2\alpha-1}}_k\lambda^{\frac{2\alpha}{2\alpha-1}}) \right).
\end{align*}

An elementary induction argument allows to conclude
\begin{align*}
\E_{\theta_0}\left[ \exp(\lambda f(\bar{\theta}_N))\right] & = \E_{\theta_0}\left[ \exp(\lambda K_{N-1,N-1}(f)(z_{N-1}))\right] \\
& \leq \exp(\lambda \E_{\theta_0}[f(\bar{\theta}_N]) \exp\left(\varphi_{\alpha}(\gamma,H, \theta_0)  (\bar{C}^{\gamma}_N  \lambda^2 \vee \bar{C}^{\gamma,\alpha}_N \lambda^{\frac{2\alpha}{2\alpha-1}}) \ \right)
\end{align*}

\noindent \noindent with $\bar{C}^{\gamma}_N := \sum_{k=1}^{N-1}\tilde{\gamma}^2_{k,N} $, $\bar{C}^{\gamma, \alpha}_N := \sum_{k=1}^{N-1}\tilde{\gamma}^{\frac{2\alpha}{2\alpha-1}}_{k,N} ((k+1) \log^2(k+4))^{\frac{1-\alpha}{2\alpha-1}} $ and where we again introduced, for sake of clarity, the constant $\varphi_{\alpha}(\gamma,H, \theta_0):= K_{4.1} 2^{\frac{1-\alpha}{2\alpha-1}} \frac{\beta C^2_\alpha}{4}\vee \left(\frac{\beta C^2_\alpha}{4}\right)^{\frac{\alpha}{2\alpha-1}} e^{\frac{1}{2\alpha-1}\sum_{k=0}^{N-1} \frac{1}{(k+1)\log^2(k+4)}}$. 

For $\alpha=\frac12$, similarly to the proof of Proposition \ref{CONT_LAPLACERM} (actually use again Lemma \ref{CONTLEMMA}), we derive for all $\lambda \in [0, \lambda_{4.1} / \bar{s}^{1/2}_N)$
$$
\E_{\theta_0}\left[\exp\left(q_k  \frac{\lambda^2}{4}  \beta  \tilde{\gamma}^2_{k,N} C^2_{1/2} L^{1/2}(\theta_{k-1}) \right)\right]^{\frac{1}{q_k}}  \leq \exp\left(\Psi(N,\gamma, \theta_0) \tilde{\gamma}^2_{k,N} \lambda^2 + (\sum_{p=0}^{N-1} \gamma^2_{p+1}) \tilde{\gamma}^2_{k,N} \frac{(\lambda /\lambda_{4.1})^2}{1- (\lambda \bar{s}^{1/2}_N/\lambda_{4.1})} \right)
$$

\noindent with $\bar{s}_N : = \max_{1 \leq k \leq N-1} (k+1) \log^2(k+4) \tilde{\gamma}^2_{k,N}$,  $\Psi(N,\gamma, \theta_0):= \left(L^{\frac12}(\theta_0) + \underline{C} \sum_{p=0}^{N-1}\gamma^2_{p+1}\right)\Pi_{2,N}\left(1/2\right) \frac{\beta C^2_{1/2}}{4} + \beta \eta C_{1/2}^4 \Pi_{2,N}\left(1/2\right) \E[|U|^2]$ and an elementary induction argument clearly yields
\begin{align*}
\forall \lambda \in [0, \lambda_{4.1} / \hat{s}_N), \ \ \E_{\theta_0}[\exp(\lambda f(\theta_N))] & \leq \exp\left(\lambda \E_{\theta_0}[f(\theta_N)]\right) \exp\left(2\varphi_{1/2}(\gamma, H, \theta_0)\bar{C}^{\gamma}_N  \frac{(\lambda / \lambda_{4.1})^2}{1-(\lambda \hat{s}_N / \lambda_{4.1})}  \right)
\end{align*}

\noindent with $\hat{s}_N := \bar{s}^{1/2}_N \exp(\sum_{k=0}^{N-1} \frac{1}{(k+1) \log^2(k+4)})$ and  $\varphi_{1/2}(\gamma, H, \theta_0):= \exp(\sum_{k=0}^{N-1} \frac{1}{(k+1) \log^2(k+4)}) (\lambda^2_{4.1} \Psi(N,\gamma, \theta_0) + \sum_{p=0}^{N-1} \gamma^2_{p+1}) $.
%$$
%\Psi(\lambda, N, \gamma) := \left( L^{\frac12}(\theta_0) + \underline{C} \sum_{k=0}^{N-1} \gamma^2_{k+1} \right) \Pi_{2,N}(1/2) \frac{\beta}{4} \lambda^2 + \left(\frac{1}{2} \sum_{k=0}^{N-1} \gamma^2_{k+1} \right) \log \E\left[\exp(\frac{C^4_{1/2}}{2} \Pi_{2,N}(1/2) \beta \lambda^2 |U|^2)\right]
%$$

%\noindent for all $\lambda < 2 \beta^{-\frac12} C^{-1}_{1/2} \min(1, \varepsilon^{\frac12}_{\beta} (2 C_{1/2}^2\Pi_{2,N}\left(1/2\right))^{-\frac12})$ and $\Psi(\lambda, N,\gamma) = +\infty$ otherwise.

%An elementary induction argument finally yields
%\begin{align*}
%\E_{\theta_0}[\exp(\lambda f(\bar{\theta}_N))] & \leq \exp(\lambda \E_{\theta_0}[f(\bar{\theta}_N)]) \exp\left(\bar{C}^{\gamma}_N\Psi\left(\lambda \exp(\bar{C}^{\gamma}_N),N, \gamma\right) \right),
%\end{align*}
%
%which completes the proof.

\appendix \label{APPEND_A}

\section{Technical results}
\subsection{Proof of Proposition \ref{PROP_CONT_W1}}

Let $e_{\sigma} := \frac{1}{2\sigma}\exp(-|x|/\sigma)$ be the density of the exponential distribution with variance $2\sigma^2$ on $\R$. If $\mu$ is a probability measure on $\R^d$, we define $\mu^{\sigma}$ as the convolution of $\mu$ with $e_{\sigma}^{\otimes d}$, that is 
$$\mu^{\sigma}(dx) := \int{\underset{i = 1}{\stackrel{d}{\prod}}\frac{1}{2\sigma}\exp(-|x_i - y_i|/\sigma)\mu(dy)}.$$

\begin{LEMME} 
If $\mu$ is a probability measure on $\R^d$ with finite first moment, then $W_1(\mu, \mu^{\sigma}) \leq \sqrt{2d}\sigma$.
\end{LEMME}

\begin{proof}
Let $X$ and $Y$ be independent random vectors with laws $\mu$ and $e_{\sigma}^{\otimes d}$ respectively. Then $(X,X+Y)$ is a coupling of $\mu$ and $\mu^{\sigma}$, and 
$$
W_1(\mu, \mu^{\sigma}) \leq \mathbb{E}[|Y|] \leq \mathbb{E}[|Y|^2]^{1/2} \leq \sqrt{2d}\sigma.
$$
\end{proof}

We therefore have the bound
\begin{align}
W_1(\mu_n, \mu) &\leq W_1(\mu_n, \mu_n^{\sigma}) + W_1(\mu_n^{\sigma}, \mu^{\sigma}) + W_1(\mu^{\sigma}, \mu) \leq W_1(\mu_n^{\sigma}, \mu^{\sigma}) + \sqrt{8d}\sigma,
\end{align}

\noindent so what is left is to bound $\mathbb{E}[W_1(\mu_n^{\sigma}, \mu^{\sigma})]$ and to optimize in $\sigma$.

The density of $\mu_n^{\sigma}$ with respect to the Lebesgue measure is given by $g_{1,\sigma, n}(x) := \frac{1}{n} \sum e_{\sigma}^{\otimes d}(x - x_i)$, and the density of $\mu^{\sigma}$ is $g_{2,\sigma}(x) := \mathbb{E}_{\mu}( e_{\sigma}^{\otimes d}(x - X))$.

By the Kantorovitch-Rubinstein duality formula, we have
$$
W_1(\mu_n^{\sigma}, \mu^{\sigma}) = \sup_{f: [f]_1 \leq 1}\hspace{2mm} \int{f(x)g_{1,\sigma, n}(x)dx} - \int{f(x)g_{2,\sigma}(x)dx} \leq \int{|x||g_{1,\sigma, n}(x) - g_{2,\sigma}(x)|dx} 
$$

To bound this quantity, we shall use the following Carlson-type inequality: for any nonnegative  measurable function $f$ on $\R^d$, we have
$$
\int{f(x)dx} \leq C_d \sqrt{\int{(1 + |x|^{d+1})f(x)^2dx}}, \ \ \ C_d := \sqrt{\int_{\R^d}{\frac{1}{1 + |x|^{d+1}}dx}}.
$$

\noindent This can be proved by using Jensen's inequality with the finite measure $\frac{1}{1 + |x|^{d+1}}dx$. Using this inequality, we get the bound 
\begin{align}
W_1(\mu_n^{\sigma}, \mu^{\sigma}) &\leq C_d \sqrt{\int{(1 + |x|^{d+1})|x|^2|g_{1,\sigma, n}(x) - g_{2,\sigma}(x)|^2dx}} \leq C_d\sqrt{\int{(1 + 2|x|^{d+3})|g_{1,\sigma, n}(x) - g_{2,\sigma}(x)|^2dx}}. \notag
\end{align}

Therefore, 
\begin{align}
\mathbb{E}[W_1(\mu_n^{\sigma}, \mu^{\sigma})] &\leq C_d \mathbb{E}\left[\sqrt{\int{(1 + 2|x|^{d+3})|\frac{1}{n} \sum e_{\sigma}^{\otimes d}(x - X_i) - \mathbb{E}_{\mu}( e_{\sigma}^{\otimes d}(x - X))|^2dx}} \right] \notag \\
&\leq \frac{C_d}{\sqrt{n}} \sqrt{\int{(1 + 2|x|^{d+3}) \text{Var}_{\mu}(e_{\sigma}^{\otimes d}(x - X))dx}} \notag \\
&\leq \frac{C_d}{\sqrt{n}} \sqrt{\int{(1 + 2|x|^{d+3}) \mathbb{E}[e_{\sigma}^{\otimes d}(x - X)^2]dx}}. \notag 
\end{align}

Note that $e_{\sigma}^{\otimes d}(x)^2 = 2^{-2d}\sigma^{-d}e_{\sigma/2}^{\otimes d}(x)$, so that we get
\begin{align*}
\mathbb{E}[W_1(\mu_n^{\sigma}, \mu^{\sigma})] &\leq \frac{C_d}{2^d\sigma^{d/2}\sqrt{n}} \sqrt{\int{(1 + 2|x|^{d+3})\int{e_{\sigma/2}^{\otimes d}(x-y)\mu(dy)}dx}} \\
&\leq \frac{C_d}{2^d\sigma^{d/2}\sqrt{n}} \sqrt{\int{\int{(1 + 2|u + y|^{d+3})e_{\sigma/2}^{\otimes d}(u)du}\mu(dy)}} \\
&\leq \frac{C_d}{2^d\sigma^{d/2}\sqrt{n}} \sqrt{\int{\int{(1 + 2^{d+3}(|u|^{d+3} + |y|^{d+3}))e_{\sigma/2}^{\otimes d}(u)du}\mu(dy)}} \\
&\leq \frac{C_d}{2^d\sigma^{d/2}\sqrt{n}} \sqrt{1 + 2^{d+3}\int{|y|^{d+3}\mu(dy)} + 2^{d+3}\int{|u|^{d+3}e_{\sigma/2}^{\otimes d}(u)du}}  \\
&\leq \frac{C_d}{2^d\sigma^{d/2}\sqrt{n}} \sqrt{1 + 2^{d+3}\int{|y|^{d+3}\mu(dy)} + \sigma^{d+3}\int{|u|^{d+3}e_{1}^{\otimes d}(u)du}}  \\
&\leq \frac{C_d}{2^d\sigma^{d/2}\sqrt{n}} \sqrt{1 + 2^{d+3}\int{|y|^{d+3}\mu(dy)} + 2^{d+3}\sigma^{d+3}d(d+3)!} 
\end{align*}

In the end, assuming $\sigma \leq 1$, we obtain 
$$\mathbb{E}[W_1(\mu, \mu^{\sigma})] \leq \sqrt{8d}\sigma + \frac{C_d}{2^d\sigma^{d/2}\sqrt{n}} \sqrt{1 + 2^{d+3}\int{|y|^{d+3}\mu(dy)} + 2^{d+3}\sigma^{d+3}d(d+3)!} \leq C(d,\mu)(\sigma + \frac{\sigma^{-d/2}}{\sqrt{n}})
$$

Taking $\sigma = n^{-1/(d+2)}$, we get the upper bound we were aiming for.

\bibliographystyle{alpha}
\bibliography{bibli}

\begin{thebibliography}{AKHJ12}

\bibitem[AKHJ12]{Alfonsi2012}
A~Alfonsi, A.~Kohatsu-Higa, and B.~Jourdain.
\newblock Pathwise optimal transport bounds between a one-dimensional diffusion
  and its euler scheme.
\newblock {\em Preprint}, 2012.

\bibitem[BB06]{blow:boll:06}
G.~Blower and F.~Bolley.
\newblock Concentration inequalities on product spaces with applications to
  {M}arkov processes.
\newblock {\em Studia Mathematica}, 175-1:47--72, 2006.

\bibitem[BGV07]{boll:guil:vill:07}
F.~Bolley, A.~Guillin, and C.~Villani.
\newblock Quantitative concentration inequalities for empirical measures on
  non-compact spaces.
\newblock {\em Prob. {T}h. {R}el. {F}ields}, 137:541--593, 2007.

\bibitem[Boi11]{bois:11}
E.~Boissard.
\newblock Simple bounds for the convergence of empirical and occupation
  measures in 1-{W}asserstein distance.
\newblock {\em Electronic Journal of Probability}, 16, 2011.

\bibitem[BV05]{boll:vill:05}
F.~Bolley and C.~Villani.
\newblock Weighted {C}sisz\'ar-{K}ullback-{P}insker inequalities and
  applications to transportation inequalities.
\newblock {\em Annales de la {F}acult\'e des {S}ciences de {T}oulouse (6)},
  14(3):331--352, 2005.

\bibitem[Duf96]{Duflo1996}
Marie Duflo.
\newblock {\em Algorithmes stochastiques}, volume~23 of {\em Math\'ematiques \&
  Applications (Berlin) [Mathematics \& Applications]}.
\newblock Springer-Verlag, Berlin, 1996.

\bibitem[FM12]{Frikha2012}
Noufel Frikha and St\'ephane Menozzi.
\newblock Concentration bounds for stochastic approximations.
\newblock {\em Electron. Commun. Probab.}, 17:no. 47, 1--15, 2012.

\bibitem[GL10]{goz:leon:2010}
N.~Gozlan and C.~L{\'e}onard.
\newblock Transport inequalities. {A} survey.
\newblock {\em Markov Process. Related Fields}, 16(4):635--736, 2010.

\bibitem[JSZ85]{Johnson1985}
W.~B. Johnson, G.~Schechtman, and J.~Zinn.
\newblock Best constants in moment inequalities for linear combinations of
  independent and exchangeable random variables.
\newblock {\em Ann. Probab.}, 13(1):234--253, 1985.

\bibitem[KM02]{kona:mamm:02}
V.~Konakov and E.~Mammen.
\newblock Edgeworth type expansions for euler schemes for stochastic
  differential equations.
\newblock {\em Monte Carlo Methods Appl.}, 8--3:271--285, 2002.

\bibitem[KY03]{Kushner2003}
Harold~J. Kushner and G.~George Yin.
\newblock {\em Stochastic approximation and recursive algorithms and
  applications}, volume~35 of {\em Applications of Mathematics (New York)}.
\newblock Springer-Verlag, New York, second edition, 2003.
\newblock Stochastic Modelling and Applied Probability.

\bibitem[LM10]{lema:meno:10}
V.~Lemaire and S.~Menozzi.
\newblock On some non asymptotic bounds for the {E}uler scheme.
\newblock {\em Electronic Journal of Probability}, 15:1645--1681, 2010.

\bibitem[LP12]{Laruelle2012}
Sophie Laruelle and Gilles Pag{\`e}s.
\newblock Stochastic approximation with averaging innovation applied to
  finance.
\newblock {\em Monte Carlo Methods Appl.}, 18(1):1--51, 2012.

\bibitem[MT06]{Malrieu2006}
Florent Malrieu and Denis Talay.
\newblock Concentration inequalities for {E}uler schemes.
\newblock In {\em Monte {C}arlo and quasi-{M}onte {C}arlo methods 2004}, pages
  355--371. Springer, Berlin, 2006.

\bibitem[PJ92]{Polyak1992}
B.~T. Polyak and A.~B. Juditsky.
\newblock Acceleration of stochastic approximation by averaging.
\newblock {\em SIAM J. Control Optim.}, 30(4):838--855, 1992.

\bibitem[RM51]{Robbins1951}
Herbert Robbins and Sutton Monro.
\newblock A stochastic approximation method.
\newblock {\em Ann. Math. Statistics}, 22:400--407, 1951.

\bibitem[RR98]{rach:rusch:1998}
Svetlozar~T. Rachev and Ludger R{\"u}schendorf.
\newblock {\em Mass transportation problems. {V}ol. {II}}.
\newblock Probability and its Applications (New York). Springer-Verlag, New
  York, 1998.
\newblock Applications.

\bibitem[Rup91]{Ruppert1991}
David Ruppert.
\newblock Stochastic approximation.
\newblock In {\em Handbook of sequential analysis}, volume 118 of {\em Statist.
  Textbooks Monogr.}, pages 503--529. Dekker, New York, 1991.

\bibitem[TT90]{tala:tuba:90}
D.~Talay and L.~Tubaro.
\newblock Expansion of the global error for numerical schemes solving
  sto\-chastic differential equations.
\newblock {\em Stoch. Anal. and App.}, 8-4:94--120, 1990.

\end{thebibliography}

 \end{document}